\documentclass[12pt]{article}
\usepackage{amsmath, amsfonts, amssymb,amsthm, amscd}

\usepackage{pictexwd,dcpic}
\usepackage{graphicx}
\usepackage{epsf}
\usepackage{epsfig}
\usepackage{rlepsf}
\usepackage{psfrag}
\usepackage{color}
\usepackage{hyperref}

\newtheorem{thm}{Theorem}[section]
\newtheorem*{tthm}{Theorem}
\newtheorem{lem}[thm]{Lemma}
\newtheorem{prop}[thm]{Proposition}

\newtheorem{defn}[thm]{Definition}

\newtheorem{rem}[thm]{Remark}

\providecommand{\mcal}[1]{{\mathcal #1}}

\newcommand{\norm}[1]{\left\Vert#1\right\Vert}
\newcommand{\abs}[1]{\left\vert#1\right\vert}

\newcommand{\ssc}{\text{sc}}

\renewcommand{\epsilon}{\varepsilon}

\newcommand{\wh}{\widehat}

\newcommand{\wt}{\widetilde}

\newcommand{\ov}{\overline}

\newcommand{\id}{\operatorname{id}}
\newcommand{\ind}{\operatorname{Ind}}

\newcommand{\supp}{\operatorname{supp}}

\newcommand{\R}{{\mathbb R}}

\renewcommand{\qedsymbol}{$\blacksquare$}

\newcommand{\mo}{\mathcal O}
\providecommand{\codim}[1]{$\text{codim}\ {#1}$}
\providecommand{\ker}[1]{$\text{ker}\ {#1}$}
\providecommand{\dim}[1]{$\text{dim}\ {#1}$}
\providecommand{\supp}[1]{$\text{supp}\ {#1}$}

\newcommand{\N}{{\mathbb N}}

\def\abs#1{\mathopen|#1\mathclose|}

\def\norm#1{\mathopen\|#1\mathclose\|}

{\catcode`"=12 \gdef\hex{"}}

\mathchardef\laplace=\hex0001

\mathchardef\nabla=\hex0272

\makeatletter

\def\@@dalembert#1#2{\setbox0\hbox{$#1\mathrm I$}

  \vrule height\ht0 depth\z@ width.04\ht0

  \rlap{\vrule height\ht0 depth-.96\ht0 width.8\ht0}

  \vrule height.1\ht0 depth\z@ width.8\ht0

  \vrule height\ht0 depth\z@ width.1\ht0 }

\def\dalembert{\mathbin{\mathpalette\@@dalembert{}}\,}

\makeatother

\begin{document}

\title{{ A General Fredholm Theory {II}:\\
Implicit Function Theorems}\\
by\\
 H. Hofer, K. Wysocki and E. Zehnder}

\maketitle

\tableofcontents

\section{Introduction}
This paper is a continuation of \cite{HWZ2}. There we introduced a
new concept of smoothness in infinite dimensional spaces extending
the familiar finite-dimensional smoothness concept. This allowed us
to introduce the concept of a splicing leading to new local models
for smooth spaces. These local models can be of finite or infinite
dimension and it is an interesting feature that the local dimensions
 need not to be constant. In \cite{HWZ2} we defined smooth maps between
these local models and constructed a tangent functor. Applying
classical constructions from differential geometry to these new
objects we  constructed  a generalized differential geometry, the
so-called {\bf splicing based differential geometry}.  The
generalization of an orbifold to the splicing world leads to the
notion of a polyfold and the generalization  of a manifold to the notion  of a M-polyfold.
Not surprisingly the generality of the new theory allows
constructions which do not parallel any classical construction. 
Benefits of this new differential geometric world include 
new types of spaces needed for an abstract and efficient
treatment of theories like Floer-Theory, Gromov-Witten or Symplectic
Field Theory from a common point of view.  From our point of view
the afore-mentioned theories are built on a suitable counting of
zeros of a smooth Fredholm section of a strong polyfold bundle. The
fact that the Fredholm sections  on different connected components are related in subtle ways leads to interesting algebraic
structures which can be captured on an abstract level, see
\cite{HWZ6}.  An abstract Sard-Smale-type perturbation theory takes
care of the usual intrinsic transversality difficulties well-known
in al these theories.

The implicit function theorem and Sard-type transversality theorems
are  crucial analytical tools in differential geometry and nonlinear
analysis (where one usually refers to a nonlinear Fredholm theory
and a Sard-Smale-type transversality theory). In the present
paper we develop the analogous `splicing based' theory. We carry out
the analysis needed for a generalized Fredholm theory for
sections of strong M-polyfold bundles. The analytical results are a
crucial technical tool in the theory of Fredholm functors and polyfolds presented  in the forthcoming part III in \cite{HWZ3} which is needed for
symplectic field theory treated in  \cite{HWZ5}.

Before we describe the contents of the present paper  we should  briefly
mention the prerequisites, namely the material introduced in part I of this series in \cite{HWZ2}. For the convenience of the  reader a glossary about the new concepts is added  in section 7. In there we recall, in particular, the notion of sc-smoothness, the notions of an M-polyfold and of a strong bundle splicing. Moreover, the definitions of an $\ssc^+$-section of a strong M-polyfold bundle and of a linearized section are recalled. A more comprehensive   treatment can be found in \cite{HWZ-polyfolds1}.

The main technical results concern the very special 
sc-smooth maps called contraction germs (Definition
\ref{def5.1}). They arise in  the normal forms for Fredholm sections
introduced in Definition \ref{defoffred}. Contraction germs allow an
'infinitesimal smooth implicit function theorem'  (this extremely
local version of an implicit function theorem is a special feature
of the sc-world). The difficult part is the
regularity of the solution, while the existence of a continuous solution germ is  an immediate consequence of the contraction principle in complete spaces
(see Theorems \ref{thm5.3} and \ref{newthm5.4}). Armed with the
notion of a Fredholm section and the infinitesimal implicit function
theorem we derive more  familiar versions of the implicit function
theorem. Of particular interest are  the Theorems \ref{newtheoremA} and
\ref{LOCAX1} which describe the local structure of the solution set of Fredholm sections.
Based on  these results we then develop in section \ref{globalfred} the global Fredholm
theory. It consists of a transversality and perturbation theory. Two results of particular interest are the
Theorems \ref{thmsec6-trans} and \ref{poil}. These  results look like
those familiar from classical Fredholm theory. The reader should,
however, keep in mind that the ambient spaces are M-polyfolds which
are much more general spaces than Banach manifolds and that the
notion of  a Fredholm section is much more general as well.
Here is, as a sample, a version of Theorem \ref{poil}.
\begin{tthm}
Let $p:Y\rightarrow X$ be a  fillable strong M-polyfold bundle and $f$ a
proper Fredholm section. Then given an  open neighborhood $U$
of $f^{-1}(0)$ there exist arbitrarily small $\ssc^+$-sections $s$
supported in $U$ having the property  that the Fredholm section $f+s$ is in general position to the boundary $\partial X$ (as defined in Definition \ref{generalpos})  so that the solution set 
$(f+s)^{-1}(0)$ is a smooth compact manifold with boundary with
corners contained in $U$.
\end{tthm}

If  $s_0$ and $s_1$ we are  two such $\ssc^+$-sections 
which are sufficiently small (Theorem \ref{poil} formulates this quantitatively) and  ${\mathcal M}_i=(f+s_i)^{-1}(0)$)  are associated solutions sets,
  then we can find a smooth arc $s_t$ for $t\in [0,1]$ of
$\ssc^+$-sections connecting $s_0$ with  $s_1$ so that
$$
(t,x)\mapsto  f(x)+s_t(x)
$$
has a regular compact solution set ${\mathcal M}=\{(t,x)\ |\ f(x)+s_t(x)=0\}$ which  lies in general position to the boundary of $[0,1]\times X$. In particular,  ${\mathcal M}$ is a manifold with boundary with corners and moreover
the subsets of points $(t,x)$ in ${\mathcal M}$ with $t=0$ or $t=1$ are the solution sets ${\mathcal M}_i$ for $i=0,1$. 

Under the assumptions of our theorem, we conclude  in the case $\partial X=\emptyset$  that the (un-oriented) cobordism class of the solution set is an invariant. If we are  dealing with oriented Fredholm sections  we arrive at  invariants in the oriented cobordism category.

In order to describe another sample result we make use of  the  notion of an auxiliary norm   which will be explained in chapter  \ref{globalfred}. It allows to measure the size of perturbations. We shall also introduce later on  in section \ref{someinvariants} our version of the de Rham cohomology 
$H_{\text{dR}}^*(X, \R)$ in  the sc-smooth setting.

\begin{tthm}[I{\bf  nvariants}]
Let $p:Y\rightarrow X$ be a fillable strong M-polyfold bundle,  where the M-polyfold $X$ has no boundary. We assume  that 
$f$ is a proper oriented Fredholm section. We suppose further that $N$ is an
auxiliary norm and $U$ an open neighborhood of  the solution set $f^{-1}(0)$,  so that
$(N,U)$ controls compactness. Then there is a well-defined map
$$
\Phi_f:H_{dR}^{\ast}(X,{\mathbb R})\rightarrow {\mathbb R}
$$
satisfying 
$$
\Phi_f([\omega]) = \int_{{\mathcal M}^{f+s}}\omega 
$$
for every generic solution set ${\mathcal M}^{f+s}=(f+s)^{-1}(0)$, where $s$ is an  $\ssc^+$-section having  support in $U$ and satisfying $N(s)<1$. 
Moreover, if $t\rightarrow f_t$ is a proper homotopy of oriented
Fredholm sections then 
$$\Phi_{f_0}=\Phi_{f_1}.$$
\end{tthm}

There are also results  for $M$-polyfolds  with boundaries  $\partial X$ involving
the relative de Rham cohomology $H^\ast_{dR}(X,\partial X)$. We have proved
such generalizations for Fredholm sections of strong polyfold bundles in \cite{HWZ3} and will not address this topic here.

As already mentioned,  the traditional concept of nonlinear Fredholm maps is not sufficient for our purposes.   In order to motivate our new definition of a Fredholm section of sc-vector bundles we first recall from \cite{BSV} the classical local analysis of  a smooth map $f:E\to F$ between Banach spaces defined near the origin in $E$ and satisfying $f(0)=0$.
We assume  that the derivative
$Df(0)\in {\mathcal L}(E, F)$ is a Fredholm operator. Consequently,
there are the topological splittings
$$E=K\oplus X\quad \text{and}\quad F=C\oplus Y$$
where $K=\text{ker} Df(0)$ and $Y=\text{im} Df(0)$ and $\text{dim} K<\infty$, $\text{dim}C<\infty$. In particular,
$$Df(0):X\to Y$$
is a linear isomorphism of Banach spaces. It allows to introduce the linear isomorphism
$\sigma:K\oplus X\to K\oplus Y$ defined by $\sigma (k, x)=(k, Df(0)x).$
In the new coordinates of the domain, the map $f$ becomes the composition $h=f\circ \sigma^{-1}:K\oplus Y\to C\oplus Y,$
satisfying $h(0)=0$ and
$$Dh(0)[k, y]=[0,y].$$
If we write $h=(h_1, h_2)$ according to the splitting of the target space and denote by $P:C\oplus Y\to Y$ the canonical projection, then the mapping $h_2(k, y)=P\circ f\circ\sigma^{-1}(k, y)$ has the following normal form
$$h_2(k, y)=y+B(k, y),$$
where $B(0)=0$ and $DB(0)=0$. Hence  the maps
$$y\mapsto B(k, y)$$
are contractions near the origin in $Y$ having arbitrary small contraction constants $0<\varepsilon<1$ if $k$ and $y$ are sufficiently close to $0$ depending on the contraction constant.  In other words, taking a suitable coordinate representation of $f$ and a suitable projection onto a subspace of finite
codimension in the target space, the non linear Fredholm
 map looks locally near $0$ like a parametrized contraction
perturbation of of the identity mapping.  

In the setting of splicing based differential geometry one cannot rely anymore on the classical implicit function theorem. This classical  theorem gives an insight into the behavior near a point at which only the linearization is known. In our approach to  Fredholm sections we shall instead start from the contraction normal form and shall call a section Fredholm, if in appropriate coordinates it has a contraction normal form of the type above. Consequently, we begin the analysis with an implicit function theorem for contraction germs.
\mbox{}\\[1ex]

{\bf Acknowledgement:} We  would like to thank P.
       Albers, B. Bramham, J. Fish, U. Hrynievicz, J. Johns, R. Lipshitz and C. Manolescu, A. Momin, S. Pinnamaneni and K. Wehrheim  for useful  discussions and valuable suggestions. We thank the referees for improvements.

\section{Analysis of Contraction Germs}
In this section we study  the existence and  the regularity of solution germs
 of  special equations called contraction germs. The
results form the analytical back bone for our implicit function
theorems and Fredholm theory.
\subsection{Contraction Germs}
In \cite{HWZ2} we have introduced an $\ssc$-structure on a Banach space $E$. It consists of a nested sequence
$$
E=E_0\supset E_1\supset E_1 \supset \cdots \supset \bigcap_{m\geq 0}E_m=:E_{\infty}
$$
of Banach spaces $E_m$ having the properties that the inclusions $E_n\to E_m$ are compact operators if $m<n$ and the
vector space $E_{\infty}$ is dense in $E_m$ for every $m\geq 0$.  Points and sets contained in $E_{\infty}$ are called {\bf smooth points}  and {\bf smooth sets}.  A Banach space equipped with an $\ssc$-structure is called an
$\ssc$-Banach space.
We write $E^m$ to emphasize that we are dealing with the Banach space $E_m$ equipped with the sc-structure $(E^m)_k:=E_{m+k}$ for $k\geq 0$.

A {\bf  partial quadrant}  $C$ in the $\ssc$-Banach space $E$ is a closed  subset which under a suitable sc-isomorphism
$T:E\rightarrow {\mathbb R}^n\oplus W$ is mapped onto
$[0,\infty)^n\oplus W$. 

We denote by $\mo(C, 0)$ an $\ssc$-germ of
C-relative open neighborhoods of $0$ consisting of a decreasing
sequence
$$U_0\supset U_1\supset U_2\supset \cdots \supset U_m\supset \cdots $$
of open neighborhoods $U_m$  of $0$ in $C_m=C\cap E_m$. The associated {\bf tangent germ} $T\mo (C,0)$ of ${\mathcal O}(C,0)$ consists of the decreasing
sequence
$$
U_1\oplus E_0\supset U_2\oplus E_1\supset  \cdots \supset U_{m+1}\oplus
E_m \supset \cdots .
$$
Assume that $F$ is a second Banach space equipped with the $\ssc$-structure $(F_m)$.
 Then  an {$\mathbf{sc}^{\boldsymbol{0}}${\bf -germ}}
 $$f:\mo(C, 0)\to (F, 0)$$
is a continuous map $f:U_0\to F_0$ satisfying $f(0)=0$ and
$f(U_m)\subset F_m$ such that
$$f:U_m\to F_m$$
is continuous for every $m\geq 0$. We extend the idea of an
$\ssc^1$-map as defined in Definition 2.13 in  \cite{HWZ2} to the germ level as follows.
An $\mathbf{sc}^{\boldsymbol{1}}${\bf -germ}  $f:\mo(C, 0)\to (F,
0)$ is an $\ssc^0$-germ which is  also of class $\ssc^1$ in  the
following sense.  For every $x\in U_1$ there exists the linearisation
$Df(x)\in {\mathcal L}(E_0, F_0)$ such that for all $h\in E_1$ with $x+h\in U_1$,
$$\dfrac{1}{\norm{h}_1}\norm{f(x+h)-f(x)-Df(x)h}_0\to 0\quad \text{as\, $\norm{h}_1\to 0$.}$$
Moreover, the tangent map $Tf:U_1\oplus E_0\to TF$, defined by
$$Tf(x,h)=(f(x),Df(x)h)$$
for $(x,h)\in U_1\oplus E_0$,  satisfies $Tf(U_{m+1}\oplus E_m)\subset
F_{m+1}\oplus F_m$ and $Tf:T{\mathcal O}(C,0)\rightarrow (TF,0)$ is
an  $\ssc^0$-germ.

\begin{defn}\label{def5.1}
Let $E$ be an sc-Banach space and let $V$ be a partial quadrant in a finite-dimensional vector space $F$.  Then an $\ssc^0$-germ $f:\mo(V\oplus  E, 0)\to  (E, 0)$ is
called  an {\bf $\mathbf{sc}^{\boldsymbol{0}}$-contraction germ} if
it has the form
\begin{equation}\label{cgerm}
f(v,u)=u-B(v,u)
\end{equation}
so that the following holds. For every level $m$ and every $0<\varepsilon<1$
we have the estimate
$$
\norm{B(v,u)-B(v,u')}_m\leq \varepsilon\cdot\norm{u-u'}_m
$$
for all $(v,u)$ and $(v,u')$ close to $(0,0)$ in $V\oplus E_m$. Here the notion of close depends on the level $m$ and the contraction constant $\varepsilon$.
\end{defn} 

The following existence  theorem is an immediate  consequence of a parameter
dependent version of Banach's fixed point theorem.
\begin{thm}\label{thm5.2}
Let $f:\mo(V\oplus {E},0)\rightarrow (E,0)$ be an sc$^0$-contraction
germ. Then there exists a uniquely determined sc$^0$-germ
$\delta:\mo(V,0)\rightarrow ({ E},0)$ so that the associated graph
germ $gr(\delta):V\to V\oplus E$,  defined by $ v\mapsto (v,\delta(v))$,
satisfies
$$
f\circ gr(\delta)=0.
$$
\end{thm}

\subsection{Regularity of Solution Germs}
Our  main concern now is the regularity of the unique solution $\delta$ of the equation $f(v, \delta (v))=0$ in Theorem \ref{thm5.2}. We
shall prove that the solution germ $\delta$ is of class $\ssc^k$ if
the given germ $f$ is of class $\ssc^k$. We first study the case
$k=1$ and abbreviate
$$V=[0,\infty )^l\times \R^{n-l}\subset \R^n.$$

\begin{thm}\label{thm5.3}
If the $\ssc^0$-contraction germ $f:\mo(V\oplus  E,0)\rightarrow (
E,0)$ is of class $\ssc^1$, then the solution germ $\delta:\mo(V,
0)\to (E, 0)$ in Theorem \ref{thm5.2} is also of class  $\ssc^1$.
\end{thm}

We are going to prove Theorem \ref{thm5.3} under the following weaker  assumptions on the $\ssc^0$-germ $f:{\mathcal O}(V\oplus E, 0)\to (E, 0)$ having the above form 
$f(u, v)=u-B(v, u)$. We merely assume that for every level $m\geq 0$ there exists a constant $0<\rho_m<1$ such that 
$$\norm{B(v, u)-B(v, u')}_m\leq \rho_m \norm{u-u'}_m$$
for $(v,u)$ and $(v,u')$ close to $(0,0)$ in $V\oplus E_m$ where  the notion of close depends on the level $m$. However, we should point out that the stronger contraction  assumption (for every contraction constant $0<\varepsilon<1$) is crucial  later on for the stability of $\ssc^0$-contraction germs under perturbations. 

\begin{proof}[Proof of Theorem \ref{thm5.3}]
We fix  $m\geq 0$ and  first show   that  the set
\begin{equation}\label{thm5.3eq1}
\frac{1}{\abs{b}}\norm{\delta (v+b)-\delta (v)}_m
\end{equation}
is bounded for $v$ and $b\neq 0$ belonging to a small ball around
zero in $V$  whose radius depends on $m$. Since the map $B$ is of
class $\ssc^1$, there exists at $(v, u)\in U_{m+1}$ a bounded linear
map $DB(u, v)\in {\mathcal L}(\R^n\oplus E_m, F_m)$. We introduce
the following notation for the partial derivatives,
\begin{equation*}
\begin{split}
DB(v, u)(\wh{v}, \wh{u})&=
DB(v, u)(\wh{v}, 0)+DB(v, u)(0, \wh{u})\\
&=D_1B(v, u)\wh{v}+D_2B(v, u)\wh{u}.
\end{split}
\end{equation*}
Since $v\mapsto \delta (v)$  is a continuous map into $E_{m+1}$  and
since the map $B$ is of class $C^1$ as a map from an open
neighborhood of $0$ in $V\oplus E_{m+1}$ into $E_m$, we have the
identity
\begin{equation*}
B(v+b, \delta (v+b))-B(v, \delta (v+b))= \biggl(
\int_0^1D_1B(v+sb,\delta (v+b))ds\biggr)\cdot  b.
\end{equation*}
As a consequence,
\begin{equation}\label{thm5.3eq2}
\begin{split}
\frac{1}{\abs{b}}\norm{B(v+b, \delta (v+b))&-B(v, \delta (v+b))}_m\\
&\leq \int_0^1 \norm{D_1B(v+sb, \delta (v+b))}_mds \leq C_m.
\end{split}
\end{equation}
Recalling $\delta (v)=B(v, \delta (v))$ and $\delta (v+b)=B(v+b,
\delta (v+b))$, we have the identity
\begin{equation}\label{thm5.3eq3}
\begin{split}
\delta(v+b)-\delta(v)& -(B(v,\delta(v+b))-B(v,\delta(v)))\\
& = B(v+b,\delta(v+b))-B(v,\delta(v+b)).
\end{split}
\end{equation}
From the contraction property of $B$ in the second variable one
concludes

\begin{equation}\label{thm5.3eq4}
\norm{B(v, \delta (v+b))-B(v,\delta (v))}_m\leq \rho_m\cdot \norm{\delta (v+b)-\delta (v)}_m.\mbox{}\\[1pt]
\end{equation}
Now, using $0<\rho_m<1$,  one derives from \eqref{thm5.3eq3} using
\eqref{thm5.3eq2} and \eqref{thm5.3eq4} the estimate
\begin{equation*}
\begin{split}
\tfrac{1}{\abs{b}}
\norm{\delta (v+b)&-\delta (v)}_m\\
&\leq  \frac{1}{1-\rho_m}\cdot \frac{1}{\abs{b}}
\norm{ B(v+b,\delta(v+b))-B(v,\delta(v+b))}_m\\
&\leq \frac{1}{1-\rho_m}\cdot C_m
\end{split}
\end{equation*}
as claimed in \eqref{thm5.3eq1}.  Since $B$ is of class $C^1$ from
$V\oplus E_{m+1}$ into $E_m$ and since $\norm{\delta (v+b)-\delta
(v)}_{m+1}\leq C_{m+1}'\cdot  \abs{b}$ by \eqref{thm5.3eq1}, the estimate

\begin{equation}\label{thm5.3eq5}
\delta (v+b)-\delta (v)-DB (v,\delta (v))\cdot
 (b, \delta (v+b)-\delta (v))=o_m(b)\mbox{}\\[3pt]
\end{equation}
holds true, where $o_m(b)\in E_m$ is a function satisfying
$\frac{1}{\abs{b}}o_m(b)\to 0$ in $E_m$ as $b\to 0$ in $V$. We next
prove

\begin{equation}\label{thm5.3eq6}
\norm{D_2B(v, \delta (v))h}_m\leq \rho_m \cdot \norm{h}_m\mbox{}\\[1pt]
\end{equation}
for all $h\in E_{m+1}$. Fixing $h\in E_{m+1}$, we can estimate

\begin{equation*}
\begin{gathered}
\norm{D_2B(v,\delta (v))h}_m\\
\leq
\frac{1}{\abs{t}}\cdot \norm{B(v,\delta(v)+th)-B(v,\delta(v))-D_2B(v,\delta(v))[th]}_m\\
+\frac{1}{\abs{t}}\cdot \norm{B(v,\delta(v)+th)-B(v,\delta(v))}_m.
\end{gathered}
\end{equation*}
In view of the postulated contraction property of $B$, the second
term is bounded by $\rho_m\cdot \norm{h}_m$, while the first term
tends to $0$ as $t\to 0$ because $B$ is of class $C^1$.  Hence the
claim \eqref{thm5.3eq6} follows. Using \eqref{thm5.3eq6}  and the
fact that $E_{m+1}$ is dense in $E_m$, we derive for  the continuous
linear operator $D_2B(v,\delta (v)):E_m\to E_m$ the bound
\begin{equation}\label{thm5.3eq7}
\norm{D_2B(v,\delta (v))h}_m\leq \rho_m\cdot \norm{h}_m
\end{equation}
for all $h\in E_m$. Thus, since  $\rho_m<1$, the continuous linear
map
\begin{gather*}
L(v):E_m\to E_m\\
L(v):=1-D_2B(v, \delta (v))
\end{gather*}
is an isomorphism. Applying the inverse $L(v)^{-1}$ to both sides of
\eqref{thm5.3eq5}  we obtain  the estimate
$$\delta (v+b)-\delta (v)-L(v)^{-1}D_1B(v, \delta (v))b=o_m(b)$$
in $E_m$. Therefore, the map $v\mapsto \delta (v)$ from $V$ into
$E_m$ is differentiable and its derivative $\delta' (v)\in {\mathcal
L}(\R^n, E_m)$ is given by the formula
\begin{equation}\label{thm5.3eq8}
\delta' (v)=L(v)^{-1} D_1B(v, \delta (v)).
\end{equation}

It remains to show that $v\mapsto \delta '(v)\in {\mathcal L}(\R^n,
E_m)$ is continuous. To see this  we define the map $F:(V\oplus
\R^n)\oplus E_m\to E_m$ by setting
$$F(v, b, h)=D_1B(v, \delta (v))b+ D_2B(v, \delta (v))h.$$
The map $F$ is continuous and, in view of \eqref{thm5.3eq7}, it is a
contraction in $h$.  Applying a parameter dependent  version of
Banach's fixed point theorem to $F$ we find a continuous function
$(v, b)\mapsto h(v, b)$ from a small neighborhood of $0$ in $V\oplus
\R^n$ into $E_m$ satisfying $F(v, b, h(v, b))=h(v, b)$. Since we also
have $F(v, b, \delta'(v)b)=\delta'(v)b$, it follows from the
uniqueness that $h(v, b)=\delta'(v)b$ and so  the map $(v, b)\mapsto
\delta '(v)b$ is continuous.  Now using the fact that $V$ is contained in a
finite dimensional space,  we conclude that  $v\mapsto \delta '(v)\in
{\mathcal L}(\R^n, E_m)$ is a continuous map.  The proof of Theorem
\ref{thm5.3}  is complete.
\end{proof}

\subsection{Higher Regularity}
Theorem \ref{thm5.3} shows that the $\ssc^0$-contraction germ $f$
which is also of class $\ssc^1$ has a
 solution germ $\delta$ satisfying $f(v, \delta (v))=0$ which is also of class $\ssc^1$. Our next aim is
 to show
 that if $f$ is of class $\ssc^k$, then   $\delta$ is  also of class $\ssc^k$.
 To do so we begin with a construction.
\begin{lem}\label{newlemma5.4}
Assume that $f:{\mathcal O}(V\oplus E,0)\rightarrow (E,0)$ is an
$\ssc^0$-contraction germ and of class $\ssc^k$ with $k\geq 1$. Denote
by $\delta$ the solution germ and assume it is of class $\ssc^j$.
(By Theorem \ref{thm5.3}, $\delta$ is  at least of class $\ssc^1$.)
 Define the germ $f^{(1)}$ by
\begin{equation*}
f^{(1)}:{\mathcal O}(TV\oplus TE,0)\rightarrow TE
\end{equation*}
\begin{equation}\label{germequation}
\begin{split}
f^{(1)}(v,b,u,w)&=\left(u-B(v,u),w-DB(v,\delta(v))\left(b,w\right)\right)\\
&=(u,w)-B^{(1)}(v,b,u,w),
\end{split}
\end{equation}
where the last line defines the map $B^{(1)}$. Then $f^{(1)}$ is an
$\ssc^0$-contraction germ and of class $\ssc^{\min\{k-1,j\}}$.
\end{lem}
\begin{proof}
For $v$ small, the map $B^{(1)}$ has the contraction property with
respect to $(u, w)$.
 Indeed,   on the $m$-level of $(TE)_m=E_{m+1}\oplus E_m$, i.e.,
 for $(u, w)\in E_{m+1}\oplus E_m$ and for $v$ small we can estimate,  using \eqref{thm5.3eq6},
 \begin{equation*}
\begin{split}
&\norm{B^{(1)}(v,b,u',w')-B^{(1)}(v,b,u,w)}_m\\
&\phantom{===}=\norm{ B(v,u')-B(v,u) }_{m+1}\\
&\phantom{=====}+
\norm{DB(v,\delta(v))(b,w')-DB(v,\delta(v))(b,w)} _{m}\\
&\phantom{===}\leq \rho_{m+1}\norm{u'-u}_{m+1} +\norm{ D_2B(v,\delta(v))[w'-w] }_m\\
&\phantom{===}\leq \max \{\rho_{m+1},\rho_m\} \cdot \bigl(\norm{
u'-u}_{m+1}+\norm{w'-w}_m\bigr)\\
&\phantom{===}= \max \{\rho_{m+1},\rho_m\} \cdot \norm{ (u', w')-(u,
w)}_m.
\end{split}
\end{equation*}
Consequently,  the  germ $f^{(1)}$ is an  $\ssc^0$-contraction germ.
If  now $f$ is of class $\ssc^k$ and $\delta$ of class $\ssc^j$,
then the germ $f^{(1)}$ is of class $\ssc^{\min\{k-1,j\}}$, as one
verifies by comparing the tangent map $Tf$ with the map $f^{(1)}$
and using the fact that the solution $\delta$ is of class $\ssc^j$.
By Theorem \ref{thm5.2}, the solution germ $\delta^{(1)}$ of
$f^{(1)}$ is at least of class $\ssc^0$. It solves the equation
\begin{equation}\label{chap5eq9}
f^{(1)}(v,b, \delta^{(1)}(v, b))=0.
\end{equation}
But also the tangent germ $T\delta$ defined by $T\delta
(v,b)=(\delta (v), D\delta (v)b)$ is a solution of \eqref{chap5eq9}.
From the uniqueness we conclude  $\delta^{(1)}=T\delta $.
\end{proof}
For our higher regularity theorem we need the following lemma.
\begin{lem}\label{newcontrlem}
Assume we are given an $\text{sc}^0$-contraction germ  $f$ of class
$\text{sc}^k$ and a solution germ $\delta$ of class $\text{sc}^j$
with $j\leq k$. Then  there exists an $\text{sc}^0$-contraction germ
$f^{(j)}$ of class $\text{sc}^{\min\{k-j, 1\}}$ having
$\delta^{(j)}:=T^j\delta$ as the solution  germ.
\end{lem}
\begin{proof}
We prove the lemma by induction with respect to $j$. If $j=0$ and
$f$ is an $\ssc^0$-contraction germ of class $\ssc^k$, $k\geq 0$,
then we  set $f^{(0)}=f$ and $\delta^{(0)}=\delta$.  Hence the
result holds true if $j=0$. Assuming the result has been proved for
$j$,  we show it is true for $j+1$. Since $j+1\geq 1$ and $k\geq
j+1$,   the map $f^{(1)}$ defined by \eqref{germequation}  is of
class $\ssc^{\min \{k-1, j+1\}}$ in view of Lemma \ref{newlemma5.4}.
Moreover, the solution germ $\delta^{(1)}=T\delta$  satisfies
$$
f^{(1)}\circ \text{gr}(\delta^{(1)})=0,
$$
and is of class $\ssc^j$.  Since $\min \{k-1, j+1\}\geq j$, by  the
induction hypothesis there exists  a map
${(f^{(1)})}^{(j)}=:f^{(j+1)}$ of regularity class
$\min\{\min\{k-1,j+1\}-j,1\}=\min\{k-(j+1),1\}$ so that
$$
f^{(j+1)}\circ\text{gr}({(\delta^{(1)})}^{(j)})=0.
$$
Setting  $\delta^{(j+1)}={(\delta^{(1)})}^{(j)}=T^{j}(T\delta
)=T^{j+1}\delta$   the result follows.
\end{proof}

Now comes the main result of this section.

\begin{thm}[{\bf Germ-Implicit Function Theorem}] \label{newthm5.4}
If $f:\mo(V\oplus {E},0)\to  ({ E},0)$ is  an
$\text{sc}^0$-contraction germ which is, in addition, of class
$\ssc^k$, then the solution germ
$$\delta: \mo(V,0)\rightarrow ({ E},0)$$
satisfying
$$f(v,\delta (v))=0$$
is also of class $\ssc^k$.
\end{thm}
In case the germ $f$ is sc-smooth at $0$, it follows  that for every $m$
and $k$  there is an open neighborhood $V_{m,k}$ of $0$ in $V$ on
which the map $v\mapsto \delta (v)$ goes into the $m$-level and
belongs to $C^k$. In particular, the solution $\delta$ is sc-smooth
at the  smooth point  $0$.  The above theorem will be one of the key
building blocks for all future versions of implicit function
theorems as well as for the transversality theory.
\begin{proof}
Arguing by contradiction assume that the solution germ $\delta$ is
of class $\text{sc}^j$ but not of class $\text{sc}^{j+1}$ with
$j<k$. In view of Lemma \ref{newcontrlem}, there exists an
$\text{sc}^0$-contraction germ $f^{(j)}$ of class $\text{sc}^{\min
\{k-j, 1\}}$ so that  $\delta^{(j)}=T^j\delta$ satisfies
$$
f^{(j)}\circ \text{gr}(\delta^{(j)})=0.
$$
Since also $k-j\geq 1$, it follows that $f^{(j)}$ is at least of
class $\text{sc}^1$. Consequently,  the solution germ $\delta^{(j)}$
is at least of class $\text{sc}^{1}$. Since
$\delta^{(j)}=T^j\delta$, we conclude that $\delta$ is at least of
class $\text{sc}^{j+1}$ contradicting our assumption. The proof of
the theorem is complete.
\end{proof}

The same discussion remains valid  for germs $f$ defined
on $V\oplus E$, where $V$ is any  finite-dimensional partial quadrant.
\begin{thm}\label{newthmboundary5.4}
Let $V$ be a finite-dimensional partial quadrant. If $f:\mo(V\oplus
{E},0)\to ({ E},0)$ is an $\text{sc}^0$-contraction germ which is,
in addition, of class $\ssc^k$, then the solution germ
$$\delta: \mo(V,0)\rightarrow ({ E},0)$$
satisfying
$$f(v,\delta (v))=0$$
is also of class $\ssc^k$.
\end{thm}
\section{Fredholm Sections}

In this section we introduce the notion of a Fredholm section and
develop its local theory.

\subsection{Regularizing Sections}
Let $p:Y\rightarrow X$ be a strong M-polyfold bundle as defined in Definition 4.9 in \cite{HWZ2}. By
$\Gamma(p)$ we denote  the vector space of $\ssc$-smooth sections of the bundle $p$ and by
$\Gamma^+(p)$ the linear subspace of $\ssc^+$-sections. The following definition can be viewed as a formalization of ``elliptic bootstrapping''.
\begin{defn}\label{defreg}
A section $f\in\Gamma(p)$ is said to be {\bf regularizing} provided
$f(q)\in Y_{m,m+1}$ for a point $q\in X$ implies that $q\in
X_{m+1}$. We denote by $\Gamma_{reg}(p)$ the subset of $\Gamma(p)$
consisting of regularizing sections.
\end{defn}
We observe that if $f\in\Gamma_{reg}(p)$ and $s\in\Gamma^+(p)$, then  the
sum $f+s$ belongs to $\Gamma_{reg}(p)$. Indeed, if $(f+s)(q)=y\in
Y_{m,m+1}$, then  $q\in X_{m}$. Hence $f(q)=y-s(q)\in Y_{m,m+1}$
implies  $q\in X_{m+1}$ since $f$ is regularizing. Therefore, we have the
map
$$
\Gamma_{reg}(p)\times \Gamma^+(p)\rightarrow
\Gamma_{reg}(p)$$
defined by $(f,s)\rightarrow f+s$. Note however that the sum of two regularizing sections is, in general, not regularizing.
\subsection{Fillers and Filled Versions}\label{subsection3.2}
The filler is a very convenient technical devise in the local investigation of the solution set of a section of a strong M-polyfold bundle.  Namely, it turns the local study of the section, which perhaps is defined on a complicated space having varying dimensions into the equivalent local study of a filled section which is defined on a fixed open set of a partial quadrant  in an $\ssc$-Banach space and which has its image in a fixed $\ssc$-Banach space.
In the following we  summarize the discussion about fillers from  \cite{HWZ2}.
We consider the fillable strong M-polyfold bundle $p:Y\rightarrow X$ as defined in
Definition  4.10 in \cite{HWZ2}. Then, given  the section $f$ of the bundle $p$ and a smooth  point $q\in X$, there exists locally a filled version of $f$ which is an $\ssc$-smooth map
$$
\bar{\bf f}:\wh{O}\rightarrow F
$$
between an open set $\wh{O}$ of a partial cone in an $\ssc$-Banach space and an   $\ssc$-Banach space $F$. This filled version is obtained as follows.

We choose a strong bundle chart  (as in Definition 4.8 in \cite{HWZ2})
$$\Phi:Y|U\rightarrow K^{\mathcal R}\vert O$$
covering the $\ssc$-diffeomorphism $\varphi: U\to O$ between an open neighborhood $U\subset X$ of the point $q$ and  the open subset $O$ of the splicing core $K^{\mathcal S}=\{(v, e)\in V\oplus E\vert \, \pi_v (e)=e\}$ associated with the splicing ${\mathcal S}=(\pi, E, V)$.  Here $V$ is an open subset of a partial quadrant in an $\ssc$-Banach space $W$ and $E$ is an $\ssc$-Banach space. We assume that $\varphi (q)=0$.

The bundle $K^{{\mathcal R}}\vert O$ is defined by
$K^{{\mathcal R}}\vert O:=\{(w, u)\in O\oplus F\vert \, \rho_{w}( u)=u\}$ where we have abbreviated $w=(v, e)$. The bundle  is the splicing core of the strong bundle splicing ${\mathcal R}=(\rho,F, (O, {\mathcal S}))$ in which  $F$ is an $\ssc$-Banach space. 

In the notation of \cite{HWZ2} which is recalled in the glossary
$$K^{\mathcal R}=K^{{\mathcal R}^0}=K^{\mathcal R}(0).$$
Having the splicing ${\mathcal S}=(\pi, E, V)$, we introduce the complementary splicing ${\mathcal S}^c=(1-\pi, E, O)$ whose  associated splicing core is given by 
$K^{{\mathcal S}^c}=\{(v, \varepsilon )\in V\oplus E\vert \, \pi_v (\varepsilon )=0\}.$  According to the decomposition $E=\pi_v (E)\oplus (1-\pi_v )(E)$ every element of $E$ has the unique representation
$$e=e'+\varepsilon, \quad \pi_v (e')=e', \quad \pi_v (\varepsilon )=0.$$
 Using the natural projections $O\to V$ and $K^{{\mathcal S}^c}\to V$ we can  form the Whitney-sum 
 $O\oplus_V K^{{\mathcal S}^c}\subset K^{\mathcal S}\oplus_V K^{{\mathcal S}^c}$ which can be identified  with the  open subset of $\wh{O}\subset V\oplus E$ defined by 
\begin{equation*}
\begin{split}
\wh{O}&=\{(v, e)\in V\oplus E\vert \, (v, \pi_v (e))\in O\}\\
&=\{(v, e'+ \varepsilon )\in V\oplus E\vert \, \text{$\pi_v (e' )=e',\, \pi_v (\varepsilon )=0$,  and $ (v, e')\in O$} \}.
\end{split}
\end{equation*}
We define  the projection $r:\wh{O}\to O$ by $r(v, e)=(v, \pi_v (e))$.Then  the preimage of a point  $(v, \pi_v(e))$,
$$r^{-1}(v,\pi_v (e))=\{(v, \pi_v (e))\} \times \ker \pi_v \subset \{(v, \pi_v (e))\} \times  E,$$
has the structure of a Banach space so that we may view $r:\wh{O}\to O$ set-theoretically as a bundle over $O$. 

Associated with the strong bundle  splicing ${\mathcal R}=(\rho , F, (O, {\mathcal S} ))$ there is the complementary 
strong bundle splicing ${\mathcal R}^c=(1-\rho , F, (O, {\mathcal S}))$ whose splicing core is given by 
$$K^{{\mathcal R}^c}=\{(w, u)\in O\oplus F\vert \, \rho_w (u)=0\}$$
where we have abbreviated $w=(v, e)\in O$.

By  definition, a {\bf filler} for ${\mathcal R}$ is an sc-smooth map 
$$f^c:\wh{O}\to K^{{\mathcal R}^c}$$
covering the identity $O\to O$ so that $f^c$ is fiber-wise a {\bf linear isomorphism}.  More in detail, for every $(v, e')\in O$, the map 
\begin{align*}
\ker \pi_v \to \ker \rho_{(v, e')}\\
\varepsilon\mapsto f^c (v, e'+\varepsilon)
\end{align*}
is a linear isomorphism.

If $f$ is a  given section of $K^{{\mathcal R}}\to O$, we introduce the composition 
$$f\circ r:\wh{O}\to K^{\mathcal R}$$
and define the {\bf filled section} $\ov{f}$ of the bundle $\wh{O}\triangleleft F\to \wh{O}$ by the formula
\begin{equation*}
\begin{split}
\ov{f}(v, e)&=((v,e), \ov{\bf f} (v, e))\\
&=f\circ r(v, e)+f^c(v,e)
\end{split}
\end{equation*}
for all $(v, e)\in \wh{O}$. 
The principal part $\ov{\bf f}:\wh{O}\to F$ of the filled section $\ov{f}$ is an sc-smooth map which splits into a sum of two sc-smooth maps
$$\ov{\bf f}(v, e)={\bf f}(v, e)+{\bf f}^c(v, e)\in \ker (1-\rho_{(v, \pi_v (e)) } )\oplus \ker  \rho_{(v, \pi_v (e))}=F$$
where ${\bf f}$ is the principal part of the given section $f\circ r:\wh{O}\to K^{\mathcal R}$ and
 ${\bf f}^c$ is the principal part of the filler $f^c:\wh{O}\to K^{{\mathcal R}^c}$. 

Now we assume that  $(v, e)\in \wh{O}$ is a zero of the filled section,  
$$\ov{\bf f}(v, u)=0.$$
Then ${\bf f}(v, e)=0$ and ${\bf f}^c(v, e)=0$. We set $e=e'+\varepsilon$  with $\pi_v(e')=e'$ and $\pi_v (\varepsilon )=0$. By assumptions of the filler $f^c$, the map $\varepsilon\mapsto  {\bf f}^c(v, e'+\varepsilon)$ is a linear isomorphism. 
We conclude that $\varepsilon=0$ and so $r(v, e)=(v, \pi_v (e))=(v, e)\in O$. Therefore, 
$$0={\bf f}(v, e),$$
 implying that $(v, e)$ is a zero of the given section $f$ of  the bundle $K^{\mathcal R}\to O$.

Having the above construction in mind we recall from \cite{HWZ2} that a strong M-polyfold bundle $p:Y\to X$ is called {\bf fillable} if around every point $q\in X$ there exists a compatible strong M-polyfold  bundle chart $(U, \Phi, (K^{\mathcal R}, {\mathcal R}))$ whose strong bundle splicing ${\mathcal R}$ admits a filler as introduced above. In principle there can be many different bundle charts admitting different fillers.

The representation of a  given section $g$ of the M-polyfold bundle $p:Y\to X$ in the above strong bundle chart $\Phi:Y\vert O\to K^{\mathcal R}\vert O$ is the push-forward 
$$f:=\Phi_\ast g=\Phi\circ g\circ \varphi^{-1}$$
which is an sc-smooth section of the above strong bundle chart $K^{\mathcal R}\to O$.  Following the above construction we consider the filled section $\ov{f}$ of the local bundle $\wh{O}\triangleleft F\to \wh{O}$  given by 
$$\ov{f}(v, e)=f\circ r(v, e)+f^c (v, e).$$
Then $\ov{f}(v, e)=0$ if and only if $r(v, e)=(v, e)$, i.e., $(v, e)\in O$, so that $x=\varphi^{-1}(v, e)$ solves $g(x)=0$.

Now we assume that the smooth point $q\in X$ corresponds to $\varphi (q)=(0, 0)\in O\subset \wh{O}$ and assume also that 
$$f(0, 0)=0.$$
Then the linearization of the local section $f$ at the solution $(0, 0)$,
$$f'(0,0):T_{(0, 0)}O\to \ker (1-\rho_{(0, 0)}),$$
is equal to $D\ov{\bf f}\vert T_{(0, 0)}O$ where  the tangent space is given by
$$T_{(0, 0)}O=\{ (\delta w, \delta e)\in W\oplus E\vert \, \pi_0 (\delta e)=\delta e\}.$$
In \cite{HWZ2} we have verified the following relation between $f'(0, 0)$ and the linearization
$$D\bar{\bf f}(0, 0): W\oplus E\to F$$
of the filled section $\bar{\bf f}$. Abbreviating the splittings $E=\text{ker}\ (1-\pi_0 )\oplus \text{ker}\ \pi_0=E^+\oplus E^-$ and
$F=\text{ker}\ (1-\rho_{(0,0)} )\oplus \text{ker}\ \rho_{(0,0)} =F^+\oplus F^-$, we use the notation $\delta e=(\delta a, \delta b)\in E^+\oplus E^-$. Then
$$D\bar{\bf f}(0, 0): ( W\oplus E^+)\oplus E^-\to F^+\oplus F^-$$
has the following matrix form
$$
\begin{bmatrix}
\delta w\oplus \delta a\\ \delta b
\end{bmatrix}\mapsto
\begin{bmatrix}
f'(0, 0)&0\\
0&C
\end{bmatrix}
\begin{bmatrix}
\delta w\oplus \delta a\\ \delta b
\end{bmatrix}.
$$
where
$$
C={\bf f}^c(0, \cdot ):E^-\to F^-$$
is an $\ssc$-isomorphism by  the definition of a filler.

We conclude that the linearization $f'(0, 0)$ of the section $f=\Phi_{\ast}g$ is  surjective if and only if the linearization $D\bar{\bf f}(0, 0)$ of the filled section   is surjective.
Moreover,  the kernel $\ker f'(0, 0)$ is equal to the kernel of
 $D\bar{\bf f}(0, 0)$ in the sense that
 $$\ker f'(0,0)\oplus \{0\}=\ker D{\bf f}^c(0,0)$$
 where the splitting on the left side refers to the $(W\oplus E^+)\oplus E^-$.  One also reads off from the matrix form that $f' (0, 0)$ is an $\ssc$-Fredholm operator if and only if $D\bar{\bf f}(0, 0)$ is an $\ssc$-Fredholm operator and in this case their Fredholm indices  agree.

Summarizing, the properties of a section germ
$[g,q]$ of a fillable strong M-polyfold bundle $p:Y\to X$ are fully reflected in the properties of the filled section germ  $[{\bar f}, (0,0)]$ of the local bundle  $\wh{O}\triangleleft F\rightarrow \wh{O}$. 

Let us finally point out  once more that the fillers are not unique. A section germ $[f, q]$ of a fillable M-polyfold bundle has,
 in general, many filled versions. They all have the same properties.

\subsection{Basic sc-Germs and Fredholm Sections}\label{sect5.3}
The aim of this section is the introduction of the crucial concept of a Fredholm section of a fillable strong M-polyfold bundle.
We start  with local considerations  and consider  $\ssc$-smooth germs
$$
f:\mo (C,0)\rightarrow F,
$$
where $C\subset E$ is a partial quadrant in the sc-Banach space $E$ and where $F$ is another $\ssc$-Banach space. We do not require that $f(0)=0$ but we observe that $ f(0)$   is also a smooth point, since $0$ is a smooth point  and $f$ is an $\ssc$-smooth map.
We view these germs  as $\ssc$-smooth section germs of  the local strong bundle $C\triangleleft F\rightarrow C$.  The definition of a special local strong bundle from \cite{HWZ2} is recalled in the glossary.

The product $C\triangleleft F$ is the product $C\oplus F$ equipped with the two  filtrations 
 $$(C\triangleleft  F)_{m, m}=C_m\oplus F_{m}\quad \text{and}\quad (C\triangleleft  F)_{m, m+1}=C_m\oplus F_{m+1}$$
 for all $m\geq 0$.
A map $\Phi:C\triangleleft F\to C'\triangleleft F'$  which is of the form 
$$
\Phi(x,h)=(a(x), b(x,h))
$$
and linear  in $h$ is called a strong bundle isomorphism if $\Phi$ is a bijection and $\Phi$ and its inverse are  of class $\ssc^{\infty}_{\triangleleft}$.  The notion of an  $\ssc^{\infty}_{\triangleleft}$-map   introduced in \cite{HWZ2}  is recalled in the glossary. 

\begin{defn}\label{stronglyeq}
Two sc-smooth germs $f:\mo (C, 0)\to F$ and $g:\mo (C', 0)\to F'$ are called {\bf strongly equivalent}, denoted by $f\sim g$, if there exists a germ
of strong bundle isomorphism $\Phi:C\triangleleft F\rightarrow
C'\triangleleft F'$ near $0$ covering the $\ssc$-diffeomorphism germ $\varphi:(C, 0)\to (C', 0)$ so that $ g$ equals the push-forward $\Phi_\ast f=\Phi\circ f\circ \varphi^{-1}$.
\end{defn}

 We recall from Section 4.4 in \cite{HWZ2} that an $\ssc$-smooth  germ $[f,0]$ is called  linearized Fredholm  at the smooth point $0$, if the
linearization
$$f'_{[s]}(0):=D(f-s)(0):E\rightarrow F$$
is  an $\ssc$-Fredholm operator, where $[s, 0]$ is any $\ssc^+$-germ satisfying $s(0)=f(0)$. The integer
$$\ind(f,0):=i(f'_{[s]}(0)),$$
where $i$ on the right hand side denotes the Fredholm index, is independent of the
$\ssc^+ $-germ $[s, 0]$ as is demonstrated in \cite{HWZ2}. Using the chain rule for
$\ssc$-smooth maps, the following proposition follows immediately from the definitions.

We  denote by $\mathfrak{C}$ the class of all section germs.

\begin{prop}\label{propA}
If $[f,0]$ and $[g,0]$ are strongly equivalent elements in
$\mathfrak{C}$ and one of them is linearized Fredholm at $0$, then so is the other
 and in this case their Fredholm indices agree,
$$
\ind(f,0)=\ind(g,0).
$$
\end{prop}

Among all elements of the section germs in $\mathfrak{C}$ there is a distinguished class of
special elements called basic class,  denoted by  $\mathfrak{C}_{basic}$ and defined as follows.

\begin{defn}\label{defn3.4}
The {\bf basic class}   $\mathfrak{C}_{basic}$  consists of  all
$\ssc$-smooth germs $[g, 0]$ in   $\mathfrak{C}$ of the form
$$
g:{\mathcal O}(([0,\infty)^k\oplus{\mathbb R}^{n-k})\oplus
W,0)\rightarrow {\mathbb R}^N\oplus W,
$$
where $W$ is an  sc-Banach space, so that
with  the projection $P:{\mathbb R}^N\oplus W\rightarrow W$,
the germ
$$
P\circ (g-g(0,0)):{\mathcal O}(([0,\infty)^k\oplus {\mathbb
R}^{n-k})\oplus W,0)\rightarrow W
$$
is an  $\ssc^0$-contraction germ in the sense of Definition
\ref{def5.1}.
\end{defn}
The basic elements have special properties as the following proposition demonstrates.\begin{prop}\label{computation}
A germ  $g:{\mathcal O}(([0,\infty)^k\oplus {\mathbb R}^{n-k})\oplus
W,0)\rightarrow {\mathbb R}^N\oplus W$  in $\mathfrak{C}_{basic}$ is
linearized Fredholm  at $0$ and its Fredholm index is equal to
$$
\ind(g,0)=n-N.
$$
\end{prop}
\begin{proof}
Abbreviate the partial  quadrant $[0, \infty )^k\oplus \R^{n-k}$ in $\R^n$ by $C$.
Denoting by $(r, w)$ the elements in $C\oplus W$, we take
the $\ssc$-germ
$$
s:{\mathcal O}(C\oplus
W,0)\rightarrow {\mathbb R}^N\oplus W
$$
defined as the constant section $s(r,w)=g(0,0)$.  Since $g(0,0)$ is a smooth point,  the section $s$ is   an $\ssc^+$-section. The germ  $f=g-s=g-g(0,0)$
satisfies $f(0)=0$. By assumption, the map
$$P\circ \bigl(g(r, w)-g(0, 0)\bigr)=w-B(r, w)$$
is an $\ssc^0$-contraction germ in the sense of Definition \ref{def5.1}. Hence the linearization
$$Df(0, 0):\R^n\oplus W\to \R^N\oplus W$$
of the section $f$ at the point $(0,0)\in C\oplus W$ is given by the formula

\begin{equation*}\begin{split}
&Df(0,0)(\delta r,\delta w)\\
&\phantom{=}=\delta w - D_2B(0,0)\delta w \\
&\phantom{=-}- D_1B(0,0)\delta r  + (I-P)Dg(0,0)(\delta r,\delta w).
\end{split}
\end{equation*}
As in the proof of Theorem \ref{thm5.3} one verifies that the linear map
$\delta w\mapsto D_2B(0,0)\delta w$ from $W_m$ to $W_m$ is a contraction for every $m\geq 0$, so that
$$I-D_2B(0, 0):W\to W$$
is an $\ssc$-isomorphism.  Therefore, the $\ssc$-operator
\begin{equation}\label{operator}
\begin{gathered}
\R^n\oplus W\to \R^N\oplus W\\
(\delta r, \delta w)\mapsto (0, \delta w- D_2B(0,0 )\delta w)
\end{gathered}
\end{equation}
is an $\ssc$-Fredholm operator whose kernel is equal to $\R^n\oplus \{0\}$ and whose cokernel is  equal to $\R^N\oplus \{0\}$, so that the Fredholm index is equal to $n-N$. We shall show that the linearized $\ssc$-operator is an $\ssc^+$-perturbation of the operator
\eqref{operator}.
Indeed, since $r\mapsto Pg(r, 0)$ is an $\ssc^+$-section, its linearization $\delta r\mapsto D_1B(0, 0)\delta r$ is an $\ssc^+$-operator. Moreover, because $I-P$ has its image in a
finite-dimensional smooth subspace, the operator
$
(\delta r,\delta w)\rightarrow (I-P)Dg(0,0)(\delta r,\delta w)
$
is also an $\ssc^+$-operator.  By Proposition 2.11 in \cite{HWZ2}, the perturbation of an
$\ssc$-Fredholm operator by an $\ssc^+$-operator is again an $\ssc$-Fredholm operator. Therefore, we conclude that the operator $Df(0, 0):\R^n\oplus W\to \R^N\oplus W$ is an $\ssc$-Fredholm operator. Because  an $\ssc^+$-operator is compact if considered on the same level, the Fredholm  index is unchanged and $\ind (g, 0)=n-N$ as claimed.
\end{proof}
Finally, we are in the position to introduce the crucial concept of  a
polyfold Fredholm section.  
\begin{defn}\label{defoffred}
Let $p:Y\rightarrow X$ be a fillable strong M-polyfold bundle. An
$\ssc$-smooth section $f$  of the bundle $p$ is called a {\bf
(polyfold) Fredholm section} if it possesses the following two
properties
\begin{itemize}
\item[$\bullet$] {\em  ({\bf Regularization property})} The section $f$ is regularizing according to Definition \ref{defreg}.
\item[$\bullet$]  {\em  ({\bf Basic germ property}) } For every smooth $q\in X$,  there {\bf exists a
filled version}  $\bar{\bf f}:\wh{O}\rightarrow F$  whose  germ
$[\bar{\bf f},0]$ is strongly equivalent to an element in
$\mathfrak{C}_{basic}$.
\end{itemize}
\end{defn}

In order to keep the notation short, we shall skip the word ``polyfold'' and refer to these sections simply as to Fredholm sections.

As we shall prove next,  a linearization $f'_{[s]}(q):T_qX\rightarrow
Y_q$ of a Fredholm section at a smooth point $q\in X$,  is an $\ssc$-Fredholm
operator.

\begin{defn}
 If $q\in X$ is a smooth point,  we call a section germ $[f,q]$ of a fillable strong
 M-polyfold bundle a {\bf  Fredholm germ}, if  $f$ is regularizing locally near $q$ and a filled
version $[\bar{\bf f},0]$ is strongly equivalent to an  element in
$\mathfrak{C}_{basic}$.
\end{defn}

\begin{prop}\label{prop5.9}
Let $[f,q]$ be a Fredholm germ of the fillable strong M-polyfold bundle
$p:Y\rightarrow X$ at the smooth point $q$. Assume a filled version
$$
\bar{\bf f}:{\mathcal O}(\wh{O},0)\rightarrow F
$$
is equivalent to the basic element $g$ in $\mathfrak{C}_{basic}$
given by
$$
g:{\mathcal O}(([0,\infty)^k\oplus {\mathbb R}^{n-k})\oplus
W,0)\rightarrow {\mathbb R}^N\oplus W.
$$
Then any linearization $f'_{[s]}(q):T_qX\rightarrow Y_q$ is an
$\ssc$-Fredholm operator and the Fredholm index
$\ind(f,q):=i(f'_{[s]}(q))$  satisfies
$$
\ind(f,q)=n-N.
$$
\end{prop}
\begin{proof}
From Proposition 4.16  in \cite{HWZ2} we know  that a linearization
$f'_{[s]}(q)$ is an $\ssc$-Fredholm operator if and only if  the linearization $\bar{\bf
f}'_{[t]}(0)$ of a filled version is an $\ssc$-Fredholm operator. The equivalence relation $\sim$  in Definition \ref{stronglyeq} respects the Fredholm
property as well as the Fredholm index. Hence the result follows from Proposition
\ref{computation}.
\end{proof}
\subsection{Stability of Fredholm sections}
Let us assume that $p:Y\rightarrow X$ is a fillable strong
M-polyfold bundle. Then $p^1:Y^1\rightarrow X^1$ is also a fillable
strong M-polyfold bundle. The filtrations are defined by ${(X^1)}_m=X_{m+1}$ and 
${(Y^1)}_{m,k}=Y_{m+1,k+1}$  for $m\geq 0$ and $0\leq k\leq m+1$ . We denote by
${\mathcal F}(p)$ the collections of all Fredholm sections of the
bundle $p$  and by
$\Gamma^+(p)$ the vector space of all $\ssc^+$-sections.
The main result in this subsection is the following
stability theorem for Fredholm sections.

\begin{thm}\label{prop5.21}
Let $p:Y\rightarrow X$ be a fillable strong
M-polyfold bundle. If $f\in {\mathcal F}(p)$ is a Fredholm section of $p$ and $s\in
\Gamma^+(p)$ an $\ssc^+$-section,  then $f+s$ is a Fredholm section of the bundle $p^1:Y^1\to X^1$, i.e., $f+s\in{\mathcal F}(p^1)$,  and the Fredholm
indices  satisfy
$$
\ind(f,q)=\ind(f+s,q)
$$
for every smooth point $q\in X$.  Further, if $s_1,...,s_k\in\Gamma^+(p)$, then the map
$$
F:{\mathbb R}^k\oplus X\rightarrow Y, \quad F(\lambda,x)=f(x)+\sum_{j=1}^k
\lambda_j s_j(x)
$$
is a  Fredholm section of  the bundle $(r^\ast Y)^1\rightarrow
({\mathbb R}^k\oplus X)^1$, where the map $r:{\mathbb R}^k\oplus
X\rightarrow X$ is the projection and $r^\ast Y$ the pull-back
bundle. Moreover,
$$
\ind(F,(0,q))=\ind(f,q) +k.
$$
\end{thm}

\begin{proof}
In view of $s\in \Gamma^+(p)$, with $f$ also the section $f+s$ has  the regularizing property. Since
 under the push-forwards of strong bundle maps  the  $\ssc^+$-sections stay
$\ssc^+$-sections,  it is sufficient to prove the result in a very
special situation. Namely,  we may assume that $f$ is already filled
and has the contraction normal form. More precisely, we assume that
the section $f$ is of the form
$$
f: O\subset ([0,\infty)^k\oplus {\mathbb R}^{n-k})\oplus
W\rightarrow {\mathbb R}^N\oplus W,
$$
where $O$ is a relatively open neighborhood of $(0, 0)$ which corresponds to $q$.
With the projection $P:{\mathbb R}^N\oplus W\rightarrow W$, the expression
$$
P[f(v,w)-f(0,0)]=w-B(v,w)
$$
has the contraction germ property near
$(0,0)$.

The linearization of the section $s$ with respect to the variable
$w\in W$ at the point $0$ is denoted by $D_2s(0)$. Since $s$ is an
$\ssc^+$-section and since the spaces  $W_{m+1}\subset W_m$  are
compactly embedded, the linear operator $D_2s(0):W_m\to \R^N\oplus
W_m$ is a compact operator for every level $m\geq 0$. Therefore,
introducing $A:=PD_2s(0):W\to W$ and using Proposition 2.11 in \cite{HWZ2},  the operator $1+A: W\to W$ is an sc-
Fredholm operator (as defined in Definition 2.8 in \cite{HWZ2}) of index $0$. The associated
sc-decompositions of $W$ are as follows,
$$1+A:W=C\oplus X\to W=R\oplus Z,$$
where $C=\text{ker}(1+A)$ and $R=\text{range}\ (1+A)$ and
$\text{dim}\ C=\text{dim}\ Z<\infty$.  Because  the section $s$ is
of class $C^1$ on every level $m\geq 1$, one deduces the
representation
$$
P[s(a, w)-s(0,0)]=Aw+S(a, w)
$$
where $D_2S(0,0)=0$. Hence $S$ is a contraction
(with arbitrarily small contraction constant)  with respect to the second
variable on every level $m\geq 1$ if $a$ and $w$ are sufficiently
small depending on the level $m$ and the contraction constant.  We can write
\begin{equation*}
\begin{split}
P[ (f+s)(a, w)-(f+s)(0,0) ]&=w-B(a, w)+Aw+S(a, w)\\
&=(1+A)w-[ B(a, w)-S(a, w) ]\\
&=(1+A)w-\ov{B}(a, w),
\end{split}
\end{equation*}
where we have abbreviated
$$\ov{B}(a, w)=B(a, w)-S(a, w).$$
By assumption,  the map $B$ belongs to the $\ssc^0$-contraction germ and hence the map $\ov{B}$ is a contraction in the second
variable on every level $m\geq 1$ with arbitrary small contraction
constant $\varepsilon>0$ if $a$ and $w$ are sufficiently small
depending on the level $m$ and the contraction constant $\varepsilon$.  Denoting the canonical projections by
\begin{align*}
P_1&:W=C\oplus X\to X\\
P_2&:W=R\oplus Z\to R,
\end{align*}
we abbreviate the map
\begin{equation*}
\begin{split}
\varphi (a, w)&:=P_2\circ P\circ[(f+s)(a,w)-(f+s)(0,0)]\\
&=P_2  [(1+A)w-\ov{B}(a, w)]\\
&=P_2[ (1+A)P_1 w- \ov{B}(a, w)].
\end{split}
\end{equation*}
We have used the relation  $(1+A)(1-P_1)=0$. The operator
$L:=(1+A)\vert X:X\to R$ is an sc-isomorphism. In view of
$L^{-1}\circ P_2\circ (1+A)P_1w=P_1w$, we obtain the formula
$$
L^{-1}\circ \varphi (a, w)=P_1w-L^{-1}\circ P_2\circ \ov{B}(a, w).$$

\noindent Writing $w=(1-P_1)w\oplus P_1w$,  we  shall consider $(a,
(1-P_1)w)$ as our  new parameter  and
correspondingly  define the map $\wh{B}$ by
$$
\wh{B}((a, (1-P_1)w),
P_1w)=L^{-1}\circ P_2\circ  \ov{B}(a, (1-P_1)w+P_1w).
$$
Since $\ov{B}(a, w)$ is a contraction in the second variable on every level $m\geq 1$ with arbitrary small contraction constant if $a$ and $w$ are sufficiently small depending on the level $m$ and the contraction constant,  the right hand side of
$$
L^{-1}\circ \varphi (a, (1-P_1)w+P_1w)=P_1w-\wh{B}(a, (1-P_1)w, P_1w)
$$
possesses the required  contraction normal
form with respect to the variable $P_1w$ on all levels $m\geq 1$,
again if $a$ and $w$ are small enough depending on $m$ and the contraction constant. 

 It remains
to prove that the above normal form is the result of an admissible
coordinate transformation of the perturbed section $f+s$. Choose a
linear isomorphism $\tau :Z\to C$ and define the fiber
transformation $\Psi:\R^N\oplus W\to \R^N\oplus X\oplus C$ by
\begin{equation*}
\begin{split}
\Psi (\delta a\oplus \delta w):= \delta a \oplus L^{-1}\circ P_2 \cdot \delta w \oplus  \tau  \circ (1-P_2)\cdot \delta w.
\end{split}
\end{equation*}
We shall view $\Psi$ as a strong bundle map covering the
sc-diffeomorphism $\psi:V\oplus W\to V\oplus C\oplus X$ defined by $\psi (a,
w)=(a, (1-P_1)w, P_1w)$ where $V=[0,\infty )^k\oplus \R^{n-k}.$ With
the canonical projection
\begin{gather*}
\ov{P}: (\R^N\oplus C)\oplus X\to X\\
\ov{P}(a\oplus (1-P_1)w\oplus P_1w)=P_1w,
\end{gather*}
and the relation  $\ov{P}\circ \Psi \circ (1-P)=0$,  we obtain the
desired formula
\begin{gather*}
\ov{P}\circ \Psi [(f+s)\circ \psi^{-1} (a, (1-P_1)w, P_1w)-(f+s)\circ \psi^{-1}(0,0, 0)]\\
=P_1w-\wh{B}(a, (1-P_1)w, P_1w).
\end{gather*}
We have proved that $[f+s, q]$ is a strong polyfold Fredholm germ of the
bundle $[b^1,q]$. Since at $0$ the linearisations of $f$ and $f+s$
only differ by an $\ssc^+$-operator the Fredholm indices are the
same.

The second part of the theorem follows along the same lines and is
left to the reader. The proof of Theorem \ref{prop5.21} is complete.
\end{proof}

In summary, starting with a fillable strong M-polyfold bundle
$p:Y\rightarrow X$ and a Fredholm section $f$
of $p$ we can consider the set of perturbations $\{f+s\ |\
s\in\Gamma^+(p)\}$. These are all Fredholm sections,   not of
the bundle $p$ but of the bundle of $p^1:Y^1\rightarrow X^1$. As we
shall see later,  most of these sections $f+s$ are transversal to the
zero section in the sense that at a zero $q$, which by the
regularizing property has to be smooth,  the linearization
$$
(f+s)'(q):T_qX^1\rightarrow Y_q^1
$$
is a surjective  sc-Fredholm operator. We shall also see that the solution set
near $q$ admits the structure of a smooth finite-dimensional
manifold in a natural way. If the underlying M-polyfold $X$ has a
boundary,  a generic $\ssc^+$-perturbation $s$ will not only make the
section $f+s$ transversal to the zero-section but it will also put
the solution set of $f+s=0$ into a general position to the boundary $\partial X$
so that the solution set is a smooth manifold with boundary with corners.

\section{Local Solutions of Fredholm Sections}\label{lsfrsect}
In this section we shall show that a Fredholm section $f$ of a fillable
strong M-polyfold bundle $p:Y\rightarrow X$ near a solution $q$ of
$f(q)=0$ has a smooth solution manifold provided the linearization
$f'(q):T_qX\rightarrow Y_q$ is surjective. We first consider the  case in which the  solution $q$ of $f(q)=0$ does not belong to the boundary $\partial X$ of the M-polyfold $X$. The boundary case will be studied in section \ref{boundarycase}. The boundary $\partial X$ of the M-polyfold $X$ is defined  by means of  the degeneracy index $d:X\to \N$ 
introduced in Section 3.4 of \cite{HWZ2}.

For the convenience of the reader we therefore  recall first the definition of the degeneracy index and the definition of a face of the M-polyfold $X$.
Around a point $x\in X$ we take a
M-polyfold chart $\varphi:U\rightarrow K^{\mathcal S}$ where
$K^{\mathcal S}$ is the splicing core associated with the splicing
${\mathcal S}=(\pi,E,V)$. Here $V$ is an open subset of a partial
quadrant $C$ contained in the sc-Banach space $W$. By definition
there exists a linear isomorphism from $W$ to ${\mathbb R}^n\oplus
Q$ mapping  $C$ onto $[0,\infty)^n\oplus Q$.  Identifying the
partial quadrant $C$ with $[0,\infty )^n\oplus Q$ we shall use the
notation $\varphi=(\varphi_1,\varphi_2)\in [0,\infty)^n \oplus
(Q\oplus E)$ according to the splitting of the target space of
$\varphi$. We associate with the point $x\in U$ the integer  $d(x)$
defined by
\begin{equation*}\label{di}
d(x)=\sharp \{\text{coordinates of $\varphi_1(x)$ which are equal to
$0$}\}.
\end{equation*}
By Theorem 3.11 in \cite{HWZ2}, the integer $d$ does not depend on the choice of the M-polyfold chart used. A point $x\in X$ satisfying $d(x)=0$ is called an interior point of $X$. The set $\partial X$ of the boundary points of $X$ is defined as
$$\partial X=\{x\in X\vert \, d(x)>0\}.$$
A point $x\in X$ satisfying $d(x)=1$ is called a {\bf good boundary point}. A point satisfying $d(x)\geq 2$ is called a corner and $d(x)$ is the order of this corner.

\begin{defn}\label{dindex}
The closure of a connected component of the set $X(1)=\{x\in X\vert \, d(x)=1\}$ is called a {\bf face} of the M-polyfold $X$.
\end{defn}

Around every point $x_0\in X$ there exists an open neighborhood $U=U(x_0)$ so that every $x\in U$ belongs to precisely $d(x)$ many faces of $U$. This is easily verified. Globally it is always true that $x\in X$ belongs to at most $d(x)$ many faces and the strict inequality is possible.

\subsection{Good Parameterizations}\label{goodpar}
We assume that ${\mathcal R}=(\rho, F, (O, {\mathcal S}))$ is a fillable strong
bundle splicing. The splicing ${\mathcal S}$ is  the triple ${\mathcal S}=(\pi, E, V)$. 
We begin with the case in which  {\bf  $V$ is an open subset of  the sc-Banach
space $G$}. This is equivalent to requiring that the open set $O$ in
the splicing core $K^{\mathcal S}\subset G\oplus E$ has no boundary,
i.e., that the degeneracy index vanishes.  Associated with the
splicing ${\mathcal R}$ is the splicing core $K=K^{\mathcal R}=
\{((v, e), u)\in O\oplus F\vert \, \rho_{(v, e)}(u)=u\}$ and the local
strong M-polyfold bundle
$$
b:K\to O.
$$
Now we study the linearized Fredholm section $f$ of the bundle $b$. Given a
smooth point $q\in O$ satisfying $f(q)=0$,  we can linearize the
section $f$ at this point and obtain the linear sc-Fredholm operator
$$
f'(q):T_qO\to K_q:=b^{-1}(q).
$$
Its kernel $N=\text{ker}\ f'(q)$ is  finite dimensional and consists
of smooth points in view of Proposition 2.9 in \cite{HWZ2}.  Hence according to Proposition 2.7 in \cite{HWZ2}, the kernel $N$ possesses an $\ssc$-complement. Any $\ssc$-complement of $N$ in $G\oplus E$ will be
denoted by $N^{\perp}$, so that
$$
N\oplus N^{\perp}=G\oplus E.
$$
The  following definition  of a good parametrization 
will be useful in the description of the
zero set $\{f=0\}$ of a Fredholm section $f$ of a fillable M-polyfold bundle $b:K\to O$ near a zero $q$ at which the linearization $f'(q)$ is surjective. 

In view of the traditional approach to Fredholm theory, the definition might look  at first sight like what ought to be the consequences of the assumption that the regularizing section has at every zero $q$ a linearization $f'(q)$ which is a surjective Fredholm operator. However, as we shall 
see later on (Proposition \ref{newtheoremA}) the existence of a good parametrization requires much stronger hypotheses  than simply linearized Fredholm and  $f'(q)$ surjective,  because in our setting of M-polyfolds we do not have the familiar implicit function theorem available in order to deduce the behavior of $f$ near a point $q$ from the properties of its linearization $f'(q)$ at the point $q$. This is the reason for the introduction of the new concept of a Fredholm section in Definition \ref{defoffred}.

\begin{defn}\label{newdefinition5.5.1}
Consider a local strong bundle $b:K\rightarrow O$ where $O$ has no
boundary. Assume that the section   $f$  of the  bundle $b$  is 
regularizing and at every solution $p\in O$ of $f(p)=0$ the section $f$ is linearized Fredholm and its  linearization 
$f'(p):T_pO\to K_p$ is surjective.
If  $q\in O$
is a solution of $f(q)=0$, the section $f$ is said to have a  {\bf
good parametrization of its solution set  near $\mathbf{q}$}, if
there exist an open neighborhood $Q\subset N$ of $0$ in the kernel
$N$ of $f'(q)$, an open neighborhood $U(q)$ in $O$, and an
sc-smooth map
$$
A:Q\rightarrow N^\perp
$$
into a complement $N^\perp$ such that the following holds true.
\begin{itemize}
\item[(1)]  $A(0)=0$ and $DA(0)=0$.
\item[(2)] The map
$$
\Gamma:Q\rightarrow G\oplus E, \quad n\mapsto q+n+A(n)
$$
has its image in $U(q)$ and parameterizes all solutions $p$ in
$U(q)$ of the equation $f(p)=0$.
\item[(3)]  If $p=\Gamma (n_0)\in U(q)$ is a solution of $f(p)=0$,  then the map
\begin{gather*}
\ker  f'(q)\rightarrow \ker f'(p)\\
\delta n\mapsto  \delta n+ DA(n_0)\delta n
\end{gather*}
is a linear isomorphism.
\end{itemize}
The map $\Gamma$ is called a  {\bf good parametrization of the zero
set near $\mathbf{q}$}.
\end{defn}

\begin{figure}[htbp]
\mbox{}\\[2ex]
\centerline{\relabelbox \epsfxsize 2.8truein \epsfbox{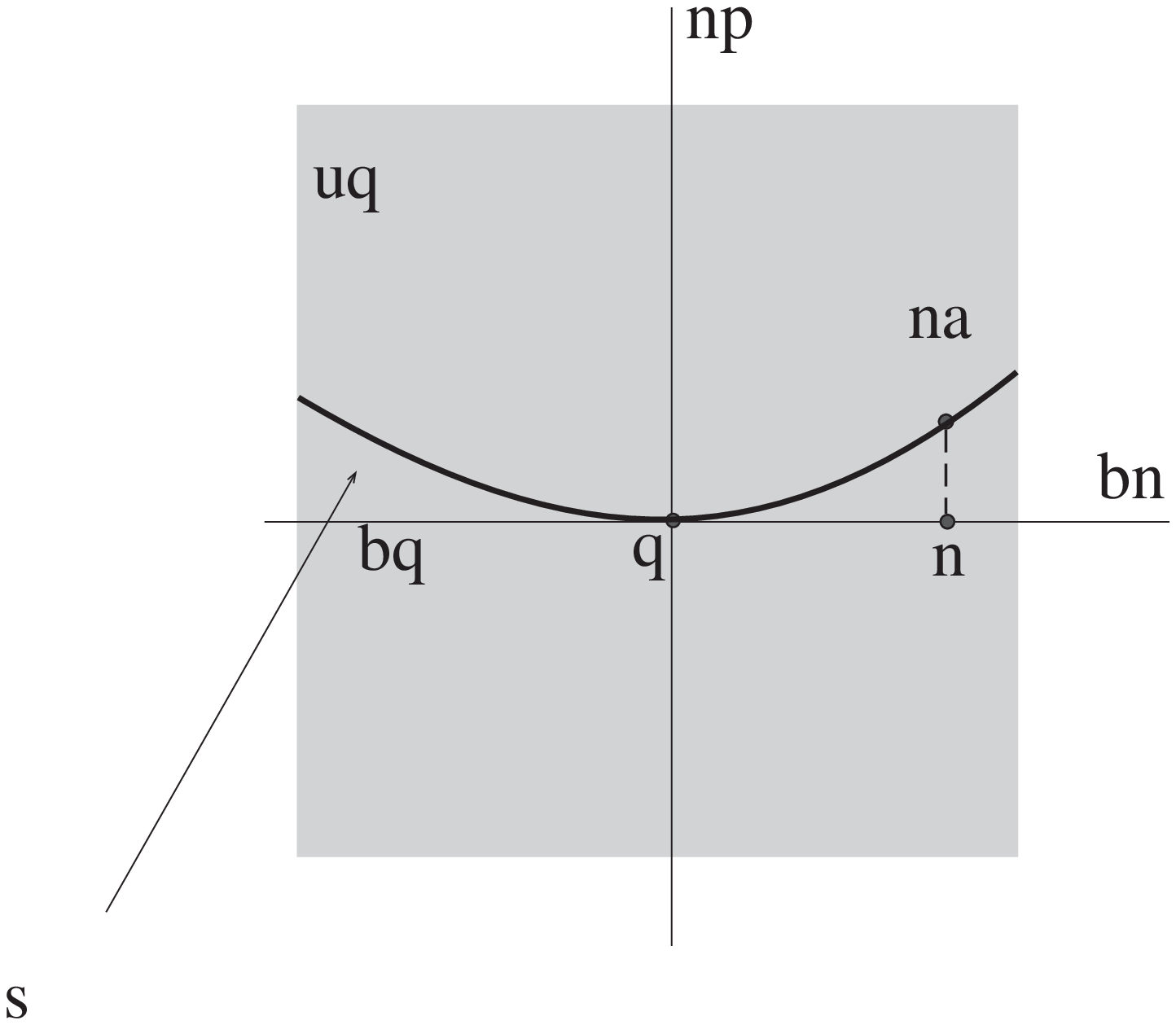}
\relabel {bq}{$Q$} \relabel {q}{$q$} \relabel {uq}{$U(q)$} \relabel
{n}{$n$} \relabel {bn}{$N$} \relabel {np}{$N^{\perp}$} \relabel
{na}{$n+A(n)$} \relabel {s}{$\small{\{ p\in U(q)\vert \,
f(p)=0\}=\Gamma (Q).}$}
\endrelabelbox}\label{Fig5.1}
\end{figure}

Every solution $q$ of the above equation $f(q)=0$ is smooth since
$f$ is assumed to be regularizing. The kernel $N=\text{ker}\ f' (q)$
at the solution is finite dimensional since $f$ is linearized
Fredholm on the solution set and can be identified
with the tangent space of the solution set at the point $q$.

The following propositions describe the  ``calculus of good
parameterizations". We  shall assume in  the following that
$b:K\rightarrow O$ is a local strong bundle and $f$ an sc-smooth
section which is regularizing and linearized sc-Fredholm at all
solutions of $f=0$.

\begin{prop}\label{newproposition5.4}
If\,  $q\in O$ is a zero of the section $f$ and if  there exists a good
parametrization $\Gamma:Q\rightarrow G\oplus E$ around $q$, where $Q$
is an open neighborhood of the origin in  the kernel $N$ of $f'(q)$,
then given any zero $q_0=\Gamma(n_0)$ of $f$ in $U(q)$, there exists
a good parametrization around $q_0$.
\end{prop}
\begin{proof}
By assumption, $f(q)=0$ and there is a good parametrization
$$
\Gamma:Q\rightarrow G\oplus E, \quad n\mapsto  q+n+A(n).
$$
In particular,  we have an open neighborhood $U(q)$ in $O$ so that
every solution $p\in U(q)$ of $f(p)=0$ lies in the image of $\Gamma$
and  at every solution $p \in U(q)$ the linearization
$f'(p):T_pO\rightarrow K_p$ is a surjective sc-Fredholm operator.
Now pick $n_0\in Q$ and let $q_0=\Gamma(n_0)$. Then $f(q_0)=0$. We
shall construct a good parametrization for the solutions set of $f$
around $q_0$. Denote the
projection from $G\oplus E=N\oplus N^{\perp}$ onto $N$ by $P$. The kernel $M$ 
 of $f'(q_0)$  is given by
$$
M=\{\delta n+DA(n_0)\delta n\ |\ \delta n\in N\}.
$$
If  $n=n_0+h$, then
\begin{equation*}
\begin{split}
\Gamma(n_0+h)&=q+ (n_0+h)+A(n_0+h)\\
&=q+(n_0+h) +A(n_0)+ DA(n_0)h \\
&\phantom{=}+[ A(n_0+h)-A(n_0)-DA(n_0)h]\\
&=: q_0 +(h+DA(n_0)h)+ \Delta(h),
\end{split}
\end{equation*}
where $q_0=\Gamma (n_0)$ and where  $\Delta(0)=0$ and $D\Delta(0)=0$. Observing that $\delta n \mapsto \delta n+DA(n_0)\delta n$ is a linear isomorphism
between $N$ and $M$ we take its inverse $\sigma$ and define for
$k\in M$ near $0$ the map $B:M\to N^{\perp}$ by
$$
B(k):= \Delta\circ\sigma(k).
$$
Then
$$
\Gamma\circ (n_0+\sigma(k))=q_0+k+B(k)
$$
and $M\oplus N^{\perp}=N\oplus N^{\perp}$. Hence the
reparametrization $k\rightarrow \Gamma(n_0+\sigma(k))$ of all the
solutions near $q_0$ has the desired form
$$
k\rightarrow q_0+k+B(k),
$$
where all the data are  sc-smooth by construction. Now take an open
neighborhood  $U(q_0)$ in   $U(q)$ and an open neighborhood $Q'$ of
$0\in M$  so that the map $k\rightarrow q_0+k+B(k)$ as a map from
$Q'$ into $U(q_0)$ parametrizes all solutions in $U(q_0)$. The proof of Proposition
\ref{newproposition5.4} is finished.
\end{proof}

Next we study  coordinate changes.
\begin{prop}\label{newproposition5.5}
Assume that $U(q)$ is a small open neighborhood of our zero $q$ and
$\Phi:K\vert U(q)\rightarrow K'\vert U'(q')$ is a strong bundle
isomorphism covering the diffeomorphism $\varphi$ and denote by $g=\Phi_\ast(f)$
the push-forward of the section  $f$. If the section $f$  has a good
parametrization $\Gamma_0$ near $q$, then there is  a good
parametrization $\Gamma$ for the section $g$ near $q'=\varphi  (q)$ constructed from
the good parametrization $\Gamma_0$  and the transformation $\Phi$.
\end{prop}
\begin{proof}
Since $f(q)=0$, we obtain for the linearized sections
$$
g'(q')\circ T\varphi (q)=\Phi (q)\circ f' (q).
$$
Hence the kernel $N'$ of $g'(q')$ is the image   under $T\varphi (q)$ of  the kernel $N$ of $f'(q)$. Consider  the following composition defined
on a sufficiently small open neighborhood $Q'$ of $0$ in $N'$
$$
\Gamma:Q'\rightarrow G'\oplus E', \quad n'\mapsto
\varphi (q+\sigma(n')+A(\sigma(n'))),
$$
where $\sigma:N'\rightarrow N$ is the linear isomorphism obtained
from the restriction of the inverse  $T\varphi(q)^{-1}$. This map is
well-defined since $q+n+A(n)$ belongs to $U(q)$ for $n$ near $0$ in $N$. Clearly,  $\Gamma$
parameterizes all solution of $g=0$ near $q'=\Gamma(0)$. Using that
$DA(0)=0$,   the linearisation of $\Gamma$ at $0$ is given by
$$
D\Gamma (0)\cdot \delta n'=T\varphi(q)\circ \sigma(\delta n')=\delta
n'.
$$
Hence $D\Gamma(0)=1$. Take a complement  ${(N')}^\perp$ and
consider the associated projection $P:G'\oplus E'=N'\oplus
(N')^{\perp}\rightarrow N'$. Then the map $\tau: Q'\rightarrow N'$,
defined  by
$$
 \tau(n'):=P(\Gamma(n')-q'),
$$
is a local isomorphism near $0$ preserving $0$ whose  linearization
at $0$ is  the identity. Define on a perhaps smaller open
neighborhood $Q''$ of $0\in N'$ the mapping
$$
A':Q''\rightarrow {(N')}^\perp, \quad n'\mapsto
(I-P)(\Gamma\circ\tau^{-1}(n')-q').
$$
Then
\begin{equation*}
\begin{split}
\Gamma\circ\tau^{-1}(n')&= q'+
P(\Gamma\circ\tau^{-1}(n')-q')+(1-P)(\Gamma\circ\tau^{-1}(n')-q')\\
&=q'+n'+A'(n').
\end{split}
\end{equation*}
In  other words, using the good  parametrization $n\mapsto
q+n+A(n)$ of the zero set of the section $f$ near $q$,  we can build a good  parametrization
$n'\mapsto q'+n'+A'(n')$ of  the zero set of  the push-forward section $\Phi_\ast f$  near $q'$ by only using the
parametrization  for $f$ and data coming from the transformation
$\Phi$.
\end{proof}

Next we study the smooth compatibility of good parameterizations. It is useful to recal from \cite{HWZ2} that an sc-smooth map between open sets of smooth inite  dimensional spaces is a $C^{\infty}$-map in the familiar sense of calculus.
\begin{prop}\label{newproposition5.6}
Let  $q_1$ and $q_2$ be zeros of the section $f$ and 
let $N_1$ and $N_2$ denote  the kernels of  the linearizations $f'(q_1)$ and $f'(q_2)$.  Assume that 
$$
\Gamma_i:Q_i\rightarrow N^\perp_i, \quad n \mapsto q_i+n+A_i(n)
$$
are good parameterizations of the zero set  of the section $f$ near $q_i$ for $i=1,2$. If $q\in U(q_1)\cap U(q_2)$ and
$\Gamma_1(n_1)=\Gamma_2(n_2)=q$, then the local transition map
$\Gamma^{-1}_1\circ\Gamma_2$  between open sets of smooth  finite dimensional vector spaces defined near the point $n_2$ is  of class $C^{\infty}$.
\end{prop}
\begin{proof}
By the properties of good parametrizations, solutions of $f(x)=0$ near $q$ are parametrized by the maps $\Gamma_1$ and $\Gamma_2$. Hence for every $n$ close to $n_2$ in $N_2$ there exists a unique $\sigma  (n)$ close to $n_1$ in $N_1$ satisfying
$$
q_1+\sigma (n)+A_1(\sigma (n))=q_2+n+A_2(n).
$$
Applying the projection $P:Q\oplus E=N_1\oplus N_1^{\perp} \to  N_1$, we find  $$
\sigma (n)= P(q_2-q_1+n+A_2(n)).
$$
The map $n\mapsto \sigma (n)$ defined for $n$ close to $n_2$ in $N_2$ into $N_1$ is sc-smooth as the  composition
of sc-smooth maps.  Hence it  is of class $C^{\infty}$  since $N_1$ and $N_2$ are smooth  finite dimensional vector spaces.
 \end{proof}

\subsection{Local Solutions Sets in the Interior
Case}

Our next aim is the proof of the existence of good parametrizations for the zero set of Fredholm sections as introduced in Definition \ref{defoffred}. We recall from 
Proposition \ref{prop5.9} that the linearization of a Fredholm section is an sc-Fredholm operator. But, as already pointed out, the mere requirement of the section to be linearized Fredholm with surjective linearization is not sufficient for the existence of good parametrizations.
\begin{thm}\label{newtheoremA}
Assume that $f$ is a Fredholm section of the fillable  strong
M-polyfold bundle $b:K\to O$. Assume that the set $O$ has no
boundary. If the smooth point $q \in O$ is a solution of $f(q)=0$
and if  the linearization  at this point,
$$
f'(q):T_qO\rightarrow K_q
$$
is surjective,  then there exists a good parametrization of the
solution set  $\{f=0\}$  near $q$.
\end{thm}

In the proof of this theorem we shall make repeated use of the following result applied to various filled versions of the germ section.

\begin{prop}\label{prop5.4.14}
Let $O$ be an open neighborhood of $0$ in some sc-Banach space
${\mathbb R}^n\oplus W$ and let  $f:O\rightarrow {\mathbb R}^N\oplus
W$ be an sc-smooth map satisfying $f(0)=0$. We assume that the
following holds.
\begin{itemize}
\item The map $f$ is  regularizing and its linearization $Df(0)$ at the smooth point $q=0\in O$  is  an sc-Fredholm operator.
\item  The linearization $Df(0):\R^n\oplus W\to \R^N\oplus W$ is surjective.
\item With  the projection $P:{\mathbb R}^N\oplus W \rightarrow W$
and writing an element in $O$ as $(v,w)$,  the $\ssc$-smooth map
$$
O\rightarrow W, \quad (v,w)\mapsto Pf(v,w)
$$
is an $\ssc^0$-contraction germ in the sense of Definition \ref{def5.1}. In particular,
$$
Pf(v,w)=w-B(v,w)
$$
where $B$ is a contraction in $w$  on every level for the
points $(v,w)$ sufficiently close to $(0,0)$ (the notion of close
depending on the level $m$).
\end{itemize}
Let $N$ be the kernel of $Df (0)$ and  let $N^\perp$ be an
sc-complement of $N$ so that $N\oplus N^{\perp}=\R^n\oplus W$. Then there exists an open neighborhood $Q$ of
$0$ in $N$, an open neighborhood $U$ of $0$ in $O$ and a $C^1$-map
$$
A:Q\rightarrow N^\perp, \quad n\mapsto  A(n)
$$
so that, with $\Gamma(n)=n+A(n)$, the following statements hold.
\begin{itemize}
\item[(1)] $A(0)=0$ and $DA(0)=0$ and for every $n\in Q$ the
point $\Gamma(n)$ is  smooth and solves the equation
$f(\Gamma(n))=0$.
\item[(2)] Every solution of $f(p)=0$ with $p\in U$ has the form
$p=n+A(n)$ for a unique $n\in Q$.
\item[(3)]For every given  pair $j,m\in {\mathbb N}$,  there exists an open
neighborhood $Q^{j,m}\subset Q$ of $0$ in $N$ so that
$$
\Gamma:Q^{m,j}\rightarrow {\mathbb R}^n\oplus W_m
$$
is  of class $C^j$.
\item[(4)] At  every $q=\Gamma(n)$ with $n\in Q$,  the linearization
$Df(q)$ is surjective.  Moreover, the projection $N\oplus N^{\perp}\to N$ induces an isomorphism $\text{ker}\ Df(q)\to \text{ker}\ Df(0)$ between the kernels of the linearizations.
\end{itemize}
\end{prop}

It is useful to recall from \cite{HWZ2} that if an $\ssc^0$-map $f:U\subset E\to F$ on the open subset $U$ of the  sc-Banach space $E$ to the sc-Banach space $F$ is of class $\ssc^k$, then $f:U_{m+k}:=U\cap E_m\to F_m$ is of class $C^k$ for every $m\geq 0$. On the other hand,  if for the $\ssc^0$-map $f:U\subset E\to F$ the induced maps $f:U_{m+k}\to F_m$ are of class $C^{k+1}$ for every $m, k\geq 0$, then $f$ is sc-smooth.

Note that  
the proposition guarantees the $\ssc$-smoothness of the map $n\mapsto A(n)$ at the distinguished point $0$ only and not at other points in $Q$.
\mbox{}\\[1ex]
\noindent{\em Proof of Proposition  \ref{prop5.4.14}}\: First, by  our
results about contraction germs,  there exists an open neighborhood
$V$ of $0$ in $ {\mathbb R}^n$ and a continuous map
$\delta:V\rightarrow W$ satisfying  $\delta(0)=0$ and
$$
P f(v,\delta(v))=0.
$$
By the regularizing property of the section $f$ we conclude  that
$\delta(v)$ is a smooth point  since $f(v,\delta(v))=(1-P)f(v,
\delta (v))\in \R^N$ is a smooth point. Hence
$$
\delta:V\rightarrow W_\infty.
$$
Since $f$ is sc-smooth,  the regularity results in the germ-implicit
function theorem (Theorem \ref{newthm5.4}) guarantee that there is a
nested sequence of open neighborhoods of $0$ in $\R^n$, say
$$
V=V_0\supset V_1\supset\cdots \supset V_j\supset \cdots
$$
so that the restrictions $\delta:V_j\rightarrow W_j$ are  of class
$C^j$ for every $j\geq 0$. Now define $G:V\rightarrow {\mathbb R}^N$
by
$$
G(v):=(1-P)f(v,\delta(v)).
$$
The map $G$  restricted to $V_j$ is of class $C^j$. Indeed,  the map
$V_j\rightarrow {\mathbb R}^n\oplus W_j$, $v\to (v,\delta(v))$ is
$C^j$ and $f$ as a map from level $j$ to level $0$ is of class
$C^j$. We need the following fact.
\begin{lem}\label{lem5.4.1}
The map $DG(0)\in {\mathcal L}(\R^n, \R^N)$ is surjective.
\end{lem}
\begin{proof}
The linearization of $G$ at the point $0$ is equal to
\begin{equation}\label{eq5.4.1.20}
DG(0)(b)=(1-P)Df(0)\cdot (b, \delta '(0)b)
\end{equation}
for $b\in \R^n$. Given $(r, 0)\in \R^N\oplus W$ we can solve the
equation
$$(r, 0)=Df(0)(b, h)$$
for $(b, h)\in \R^n\oplus W$, in view of the surjectivity of
$Df(0)$. Equivalently, there exists $(b, h)\in \R^n\oplus W$ solving
the two equations

\begin{equation*}\label{eq5.4.1.21}
\begin{aligned}
r&=(1-P)Df(0)(b, h)\\
0&=PDf(0)(b, h).\mbox{}\\[2pt]
\end{aligned}
\end{equation*}
Explicitly, the second equation is the following equation
\begin{equation}\label{eq5.4.1.22}
0=-D_1B(0)\cdot b+(1-D_2B(0))\cdot h
\end{equation}
Recalling the proof of Theorem \ref{thm5.3}, the operator
$1-D_2B(0):W\to W$ is a linear isomorphism. Therefore, given $b\in
\R^n$ the solution $h$ of the equation \eqref{eq5.4.1.22} is
uniquely determined. On the other hand, linearizing $Pf(a, \delta
(a))=0$ at the point $a=0$ leads to $PDf(0)(b,\delta' (0)b)=0$ for
all $b\in \R^n$. Hence, by uniqueness $h=\delta'(0)b$, so that the
linear map $DG(0)$ in \eqref{eq5.4.1.20} is indeed surjective.
\end{proof}
Continuing with the proof of Proposition  \ref{prop5.4.14},  we denote  the
kernel of $DG(0)$  by $C$ and take its orthogonal complement
$C^{\perp}$ in $\R^n$, so that
$$C\oplus C^{\perp}=\R^n.$$
Then $G:C\oplus C^{\perp}\to \R^N$ becomes a function of two
variables, $G(c_1, c_2)=G(c_1+c_2)$ for which $D_1G(0)=0$ while
$D_2G(0)\in {\mathcal L}(C^{\perp}, \R^N)$ is an isomorphism. By the
implicit function theorem there exists a unique map $c:C\mapsto
C^{\perp}$ solving
\begin{equation}\label{eq5.4.1.23}
G(r+c(r))=0
\end{equation}
for $r$ near $0$ and satisfying
\begin{equation*}\label{eq5.4.1.24}
c(0)=0,\quad Dc(0)=0.
\end{equation*}
Moreover, given any $j\geq 1$, there is an open neighborhood $U$ of
the origin in $C$ such that $c\in C^j(U, C^{\perp})$. Summarizing we
have demonstrated so far,   that all solutions $(v, w)\in \R^n\oplus
W$ of $f(v, w)=0$ near the origin are represented by
$$
f(r+c(r), \delta (r+c(r)))=0
$$
for $r$ in an open neighborhood of $0$ in $C=\text{ker} DG(0)\subset
\R^n$.  We introduce the function $\beta: C\to \R^n\oplus W$  defined near $0$ by
\begin{equation*}
\beta (r)= (r+c(r), \delta (r+c(r))).
\end{equation*}
Then $f(\beta (r))=0$ for $r$ near $0$ in $C$. The function $\beta$ satisfies $\beta (0)=0$ and, of course, for
every level $m$ and every integer $j\geq 1$, there exists an open
neighborhood $U$ of $0\in C$ such that if $r\in U$, then $\beta
(r)\in \R^n\oplus W_m$ and $\beta \in C^j(U, \R^n\oplus W_m)$.
Moreover, by the regularizing property of the section $f$, we know
that all points $\beta(r)$ are smooth.  In order to represent the
solution set as a graph over the kernel of the linearized equation
at $0$ we introduce
\begin{equation}\label{eq5.4.1.26}
\begin{split}
N:&=\text{ker}(Df(0)).
\end{split}
\end{equation}
We have an $\ssc$-splitting
$$
N\oplus N^\perp=\R^n\oplus W.
$$
By construction, the image of the linearization $D\beta (0)\in
{\mathcal L}(C, \R^n\oplus W)$ is equal to the kernel of $Df(0)$ and
$D\beta (0):C\to N$ is a linear isomorphism. We define the map
$\alpha:N\to \R^n\oplus W$ near $0\in N$ by
$$\alpha (n)=\beta (D\beta (0)^{-1}\cdot n).
$$
The solution set is now parametrized by $\alpha$ so that $f(\alpha
(n))=0$ for $n$ near $0\in N$. Let $P:N\oplus N^\perp\to N$ be the
 projection along $N^{\perp}$ and consider the map $P\circ \alpha :N\to N$ near
$0$. Since $D(P\circ \alpha )(0)=1$, it is a local
diffeomorphism leaving the origin fixed. We denote by $\gamma$ the
inverse of this local diffeomorphism satisfying $P\circ \alpha
(\gamma (n))=n$ for all small $n$. So,
\begin{equation*}
\begin{split}
\alpha (\gamma (n))&=P\circ \alpha (\gamma (n))+(1-P)\alpha (\gamma (n))\\
&=n+(1-P)\alpha (\gamma (n)).
\end{split}
\end{equation*}
Define the  map
$$
A:N\to N^\perp
$$
near the origin in $N$  by $A(n)=(1-P)\alpha (\gamma (n))$. Then
$A(0)=0$ and $DA(0)=0$   since the image of $\alpha' (0)$ is equal
to the kernel $N$. Moreover, $\alpha (\gamma (n))=n+A(n)$ and
$$
f(n+A(n))=0.
$$
In addition, given any level $m$ and any integer $j$ there exists an
open neighborhood $U$ of $0$ in
 $N$ such that if $n\in U$, then  $A(n)\in N^{\perp}\cap ( \R^n\oplus W_m)$ and
 $A\in C^j(U,  \R^n\oplus W_m)$.
 We have demonstrated that the solution set of $f(v, w)=0$
 near the transversal point $0$
 is represented as a graph over the kernel of
 the linearized map $Df(0)$ of a function which is sc-smooth at the point $0$.
 In order to complete the proof of Proposition   \ref{prop5.4.14}  it remains to
 verify  property (4) of  the proposition  for a suitable $Q$. This is the next lemma.
\begin{lem}\label{lem5.4.1.2}
If $n_0$ is small enough and $q_0=\Gamma(n_0)$, then  the
linearization $Df(q_0):\R^n\oplus W\to\R^N\oplus  W$ is surjective. Moreover,
$$
\text{dim}\ [\text{ker}\ Df(q_0)]=\text{dim}\ [\text{ker} \
Df(0)]=\text{dim}\ N.
$$
Setting $N'=\text{ker}\
Df(q_0)$, the natural projection $P:N\oplus N^\perp=\R^n\oplus W\to
N$ induces an isomorphism $P\vert N':N'\to N$.
\end{lem}
\begin{proof}
Consider the linearized operator
\begin{equation}\label{eq5.4.1.27}
Df(q_0):\R^n\oplus W\to \R^N\oplus W\mbox{}\\[3pt]
\end{equation}
at the solution  point
$$
q_0=(a, \delta (a))=\Gamma(n_0).
$$
In order to prove surjectivity we consider the equation $Df(q_0)(b,h)=(r, w)$. In matrix notation we interpret this equation as a system 
\begin{equation}\label{eq5.4.1.28}
\begin{bmatrix}
A&A_1\\A_2&A_3
\end{bmatrix}\cdot
\begin{bmatrix}
h\\b
\end{bmatrix}=
\begin{bmatrix}
w\\r
\end{bmatrix}
\end{equation}
with the linear operators
\begin{align*}
Ah&=(1-D_2B(q_0)) \cdot h\\
A_1b&=D_1B(q_0)\cdot b\\
A_2h&=(1-P)D_2f(q_0)\cdot h\\
A_3b&=(1-P)D_1f(q_0)\cdot b
\end{align*}
where $(b, h)\in \R^n\oplus W$ and $(r, w)\in \R^N\oplus W$. Since the point $q_0$ is smooth, the linear operator $A:W\to W$
is an $\ssc$-isomorphism  in view
of the proof of Theorem \ref{thm5.3}. Hence the equation \eqref{eq5.4.1.28} for
$(h, b)$ becomes
\begin{equation}\label{eq5.4.1.29}
\begin{gathered}
h=A^{-1}w-A^{-1}A_1b\\
(A_2A^{-1}A_1-A_3)b=A_2A^{-1}w-r.
\end{gathered}
\end{equation}
We abbreviate the continuous family of matrices
$$M(a):=[A_2A^{-1}A_1-A_3]\in {\mathcal L} (\R^n, \R^N ).$$
By assumption, $Df(0)$ is surjective. Therefore, if $a=0$, then the
matrix $M(0)$ is surjective. Consequently, for small $a$ the matrix
$M(a)$ is also surjective and so is the linear operator $Df(a,
\delta (a))$. Choosing $w=0$ and $r=0$ in \eqref{eq5.4.1.29}, the
kernel of $Df(q_0)$ is determined by the two equations
\begin{gather*}
h=A^{-1}A_1b\\
[A_2A^{-1}A_1-A_3]b=0.
\end{gather*}
Consequently, the kernel is determined by the kernel of the matrix
$M(a)$.  Recalling that $N'=\text{ker}Df(q_0)$ and $N=\text{ker}
Df(0)$ we conclude  from the surjectivity of $M(a)$ for small $a$
that $\text{dim}N'=\text{dim}N$. In addition, the natural projection
$P:N\oplus N^\perp=\R^n\oplus W\to N$ induces the isomorphism
$P\vert N':N'\to N$ and the lemma follows and therefore also Proposition
\ref{prop5.4.14} is proved.
\end{proof}
\begin{figure}[htbp]
\centerline{\relabelbox \epsfxsize 2.7truein \epsfbox{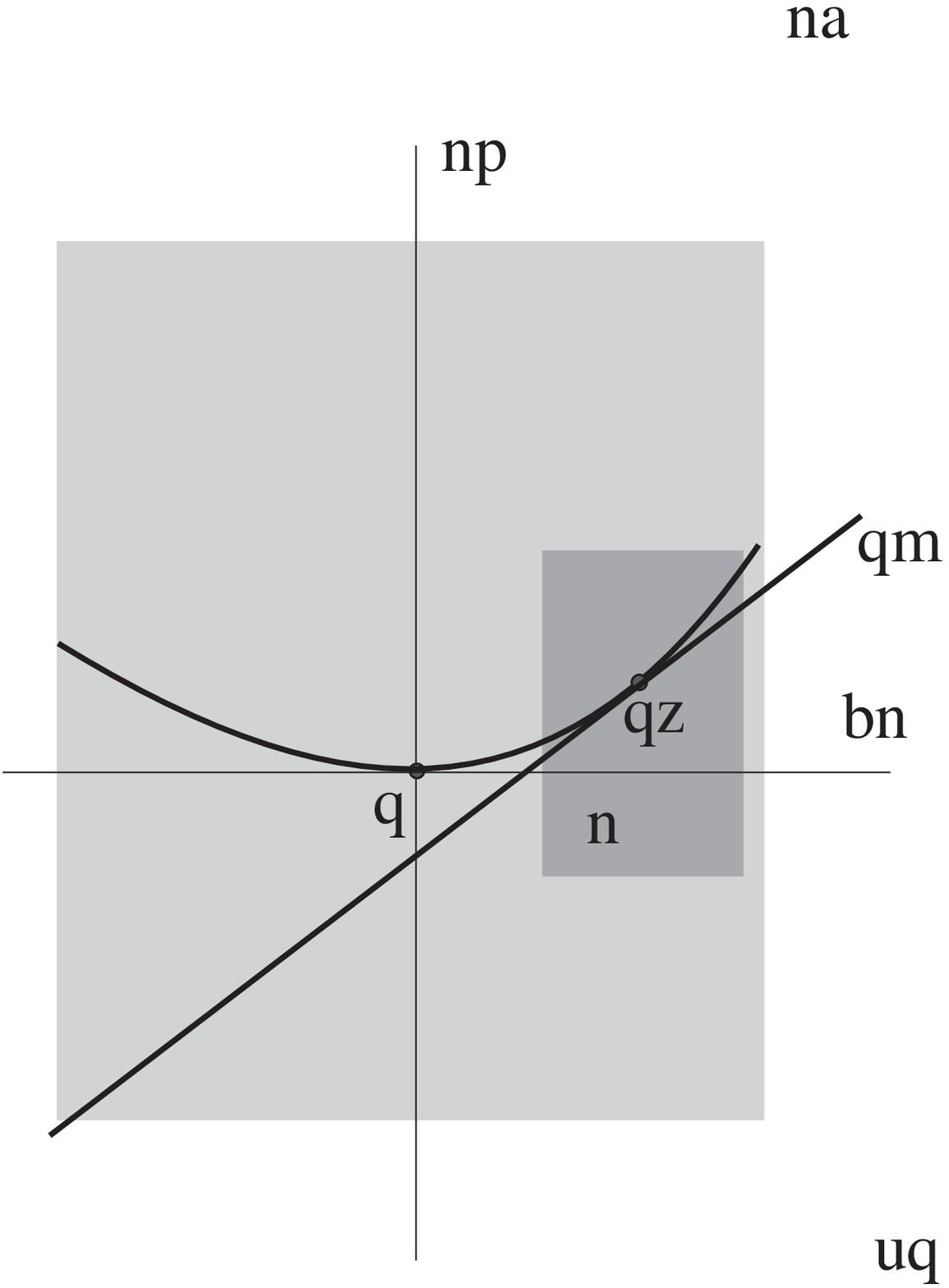}
\relabel {qm}{$q_0+M$} \relabel {q}{$q$} \relabel {n}{$U(q_0)$}
\relabel {bn}{$N$} \relabel {np}{$N^{\perp}$} \relabel {qz}{$q_0$}
\endrelabelbox}
\end{figure}
\begin{proof}[Proof of Theorem \ref{newtheoremA}]
After all these preparations  we are ready to prove Theorem
\ref{newtheoremA}  about the Fredholm section $f$ of the fillable  local
M-polyfold bundle $b:K\to O$. By  assumption,  the distinguished
smooth point $q\in O$ solves the equation $f(q)=0$ and   the
linearized map $f'(q)$ is surjective. In view of the Fredholm property of the section $f$ as defined in Definition \ref{defoffred} and the surjectivity of $f'(0)$, there exist coordinates
in which $q$ corresponds to $0$ in $\R^n\oplus W$ and in which the  the filled version
$\bar{\bf f}:\R^n\oplus W\to \R^N\oplus W$ of the section  $f$, which is defined in an open neighborhood of $0$,  meets the
assumptions of Proposition  \ref{prop5.4.14}. We have used here the fact proved in section \ref{subsection3.2} that the linearization $D\bar{\bf f}(0)$ of every  filled version is surjective if and only if $f'(q)$ is surjective. From section \ref{subsection3.2} we know that the zero set of $\bar{\bf f}$ is equal to the zero set of the local representation of $f$ in our coordinates. In view of the discussion  in section \ref{subsection3.2} and the invariance of good parametrizations (Proposition \ref{newproposition5.5}) under coordinate transformations, it is sufficient  to establish a good parametrization for the filled section $\bar{\bf f}$ near $0$.

From Proposition \ref{prop5.4.14} we deduce, using the notation of the proposition (but replacing $f$ by $\bar{\bf f}$),  that there exist an open neighborhood $Q$
of $0$ in  the kernel $ N$ of $\bar{\bf f}$, an open neighborhood $U(0)$ in $\R^n\oplus W$  and a mapping
$$\Gamma:Q\to U(0),\quad n\mapsto n+A(n)$$
of class $C^1$ satisfying $A(0)=0$,  $DA(0)=0$ and $A(n)\in
N^{\perp}$ and having the following additional properties.
\begin{itemize}\label{listx}
\item $\bar{\bf f}(\Gamma(n))=0$ for all $n\in Q$,  and
every solution $p\in U(0)$ of $\bar{\bf f}(p)=0$ lies in the image of  the map
$\Gamma$.
\item For every given  pair $j,m\in {\mathbb N}$,  there exists an open
neighborhood $Q^{j,m}\subset Q$ of $0$ in $N$ so that
$$
\Gamma:Q^{m,j}\rightarrow {\mathbb R}^n\oplus W_m
$$
is  of class $C^j$.
\item At every  point $q=\Gamma(n)$ with $n\in Q$,  the linearization
$D\bar{\bf f}(q)$ is surjective, and if $N'=\text{ker}\ D{\bf f} (q)$, then the projection $P:N\oplus N^{\perp}\to N$ induces an isomorphism $P\vert N':N'\to N$ between the kernels.
\end{itemize}
We conclude, in particular, that the map $\Gamma:Q\to U(0)$ is
sc-smooth at the distinguished point $0\in Q$ and we have to prove
that $\Gamma$ is sc-smooth at every other point of $Q$, since all
the other properties of a good parametrization near $0$
already hold true.  So, we choose $n_0\in Q$ and prove that the
map $\Gamma$ is sc-smooth at this point.   We already know  that
$q_0=\Gamma (n_0)$ solves $\bar{\bf f}(q_0)=0$ and that the linearization
$D\bar{\bf f}(q_0)$ is surjective.  Consequently, the linearization  $f'(q_0)$ is surjective and therefore also the linearization of every  other filled version at the point $q_0$ is surjective.
Now we observe that the section $\bar{\bf f}$ is a
filled version with respect to the distinguished point $q=0$ and not
necessarily with respect to the point $q_0$. Therefore, we cannot
conclude that $\bar{\bf f}$ possesses all the above nice properties near $q$.
However, by the definition of a Fredholm section, after some
coordinate change perhaps using some other filler and an additional
change of coordinates, there exists a contraction normal form near
$q_0$. In view of the surjectivity of the linearization at the point $q_0$, we can apply Proposition \ref{prop5.4.14} to this new
situation. Then going back with all these coordinate transformations
we obtain a map
\begin{gather*}
\Theta:Q'\rightarrow U(q_0)\subset U (0)\\
\Theta(m)=q_0+m+B(m),
\end{gather*}
where $Q'$ is an open neighborhood of $0$ in the kernel $M$ of
$D\bar{\bf f}(q_0)$. Moreover,  $B(m)$ lies in  some  $\ssc$-complement  $M^\perp$.
Near $m=0$, the map $B$ has the same properties as the map $A$, in
particular, $B$ is sc-smooth at $m=0$ in $M$. We know that $\Theta
(0)=q_0=\Gamma (n_0)$ and that for every $m$ near $0$, there exists
a unique  $n$ close to $n_0$ solving the equation  $\Theta(m)=\Gamma(n)$.
More explicitly,
\begin{equation}\label{plic}
q_0+m+B(m)=n+A(n).
\end{equation}
If  $P:{\mathbb R}^n\oplus W=N\oplus N^{\perp}\rightarrow N$ is
the projection along $N^\perp$,   we  define the map  $\alpha:M\to N$ near
$m=0$ by
$$
\alpha (m)=P(q_0+m+B(m)).
$$
It satisfies
$$
\alpha(0)=P(q_0)=P(n_0+A(n_0))=n_0.
$$
The linearization  $D\alpha(0)$ is equal to the projection
$P:M\rightarrow N$ which  is an isomorphism since $M$ is a graph of
a linear map $N\to N^\perp$. Moreover,  the map $\alpha$ has
arbitrarily high differentiability near $m=0$. From \eqref{plic} we
conclude
$$
n =\alpha (m)
$$
and taking the inverse near $m=0$, we  have $m=\alpha^{-1}(n)$.
Using \eqref{plic} again, we obtain $
n+A(n)=q_0+\alpha^{-1}(n)+B(\alpha^{-1}(n)), $ and applying $1-P$
to both sides we arrive at
$$
A(n)=(1-P)(q_0+\alpha^{-1}(n)+B(\alpha^{-1}(n))).
$$
The inverse map $\alpha^{-1}$ possesses for $n$ near $n_0$ arbitrary
high differentiability  and we conclude that the map $A$ is
sc-smooth at the point $n_0$. The argument applies to every $n_0\in
Q$ and we conclude that the map $\Gamma :Q\to N^{\perp}$ is
sc-smooth, hence a good parametrization of  the filled version $\bar{\bf f}$  of $f$ near $0$.  The  proof
of Theorem \ref{newtheoremA} is complete.
\end{proof}

\subsection{Good Parameterizations in the Boundary Case}\label{boundarycase}

We consider the fillable  local strong M-polyfold bundle
$$
b:K\rightarrow O
$$
where $O$ is an open subset in the splicing core $K^{\mathcal S}$.
This time the splicing ${\mathcal S}=(\pi,E,V)$ has the property
that the parameter set $V$ is a relatively open subset of  a partial
quadrant $C\subset G$. We may assume that  $C$ is of the form
$$
C=[0,\infty)^k\oplus W
$$
for some sc-Banach space $W$ and $0\in V$. We study a  section
$$
f:O\to K
$$
of the bundle $b$ and assume that the boundary point $0\in \partial O$ is a
solution of $f(q)=0$   and that the linearization
$$
f'(0):T_0O\rightarrow K_0
$$
is surjective. The aim is to describe the solution set $\{f=0\}$
near $0\in \partial O$ under the assumption that $f$ is a Fredholm
section.  Since $0$ is a boundary point we are confronted with  some
subtleties. If,  for example,  the kernel $N$ of $f'(0)$ is in a bad
position to the boundary,  there might be no solutions of $f=0$
apart from $0$, even if the Fredholm index is large and the
linearization $f'(0)$ is surjective.

\begin{defn}\label{defn5.6.1} Let $C$ be a partial quadrant in the sc-Banach space $E$.  
\begin{itemize}
\item[$\bullet$] The closed linear subspace $N$ of $E$ is in {\bf good position to the partial
quadrant} $C$  of $E$, if $N\cap C$ has a nonempty interior in $N$
and if there exist an sc-complement $N^{\perp}$ of $N$ in $E$ and a
positive constant $c$, so that for a point $(n,m)\in N\oplus_{sc}
N^{\perp}$ satisfying
$$
\norm{m}_E\leq c\cdot \norm{ n}_E
$$
the statements $n+m\in C$ and  $n\in C$ are equivalent. We call
$N^\perp$ a {\bf good complement}.

\item[$\bullet$]  A  finite-dimensional subspace  $N$ of $E$
is called {\bf neat with
respect to  the partial quadrant}  $C$  if  there exists an $\ssc$-complement $N^\perp$ of
$N$ in $E$ satisfying $N^\perp\subset C$.
\end{itemize}
\end{defn}
\begin{figure}[htbp]
\mbox{}\\[2ex]
\centerline{\relabelbox \epsfxsize 2.5truein \epsfbox{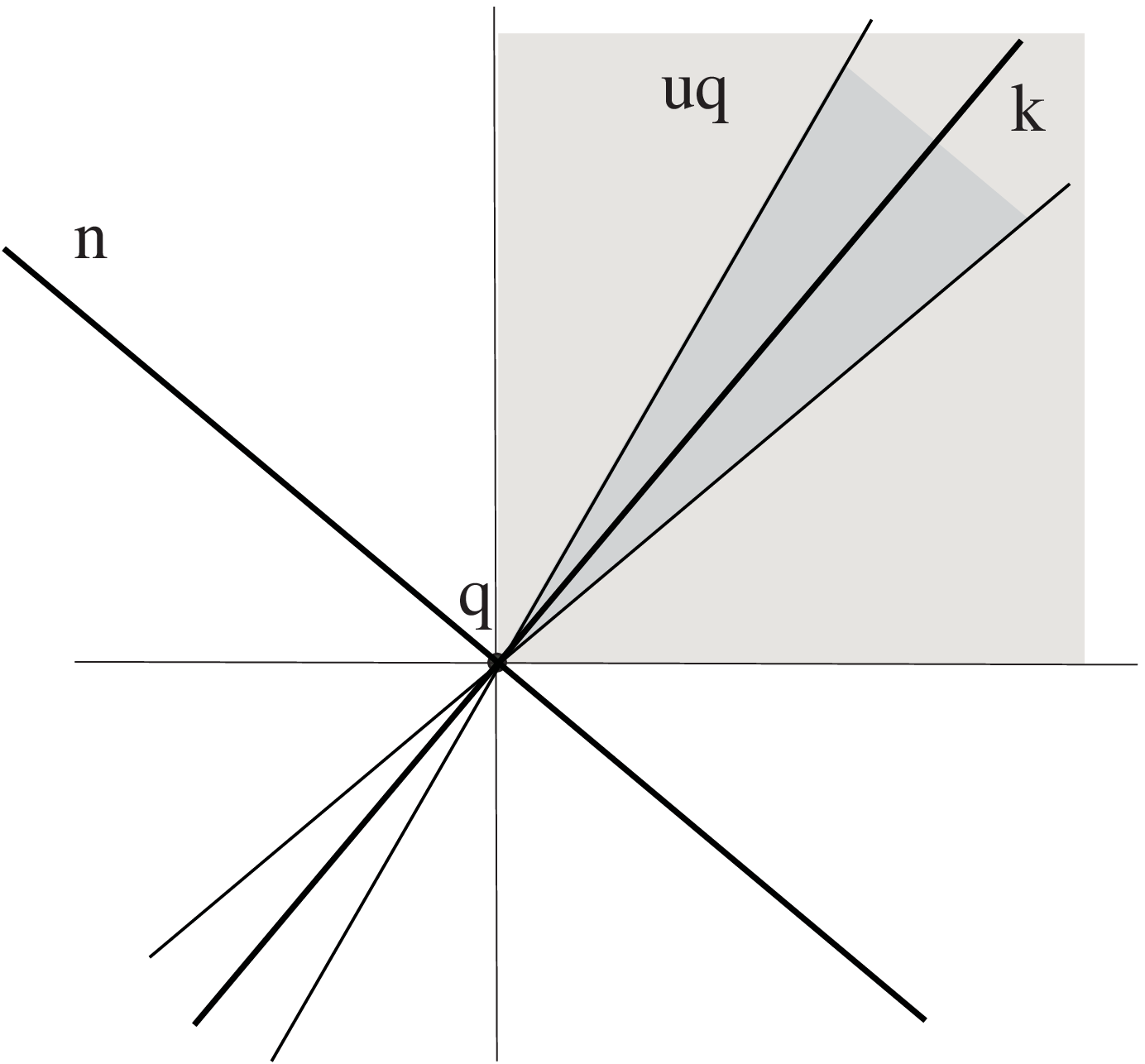}
\relabel {q}{$0$} \relabel {uq}{$C$} \relabel {k}{$N$} \relabel
{n}{$N^{\perp}$}
\endrelabelbox}
\end{figure}
The choice of the complement $N^\perp$ is important, i.e. if the
defining condition holds for one choice it does not need to hold for
another choice (at least in higher dimensions). This definition is
very practical for the local theory. Observe that  if
$N\subset {\mathbb R}^n\oplus W$ is a neat subspace with respect to the  partial quadrant
$C=[0,\infty)^n\oplus W$, then necessarily $\dim N\geq n$.
Moreover,  the following holds.
\begin{prop}\label{neat}
A neat subspace is in good position to the partial quadrant.
\end{prop}

The proof of Proposition \ref{neat} is carried out in  Appendix \ref{appendix-x} (Proposition \ref{ropp}). 

Now assume that $q$ is a smooth point in the boundary of the open
subset $O$ of the splicing core $K^{\mathcal S}$. Then $q=(v,e)\in
V\oplus E$ satisfies $\pi_v(e)=e$. Recall that
$$
V\subset C\subset \R^k\oplus W=G.
$$
The tangent space $T_qO$ is the space
$$
T_qO=\{(\delta v,\delta e)\in G\oplus E\vert \,  \delta
e=\pi_v\delta e +D_v(\pi_ve)\cdot \delta v\}.
$$
With every tangent vector  $(\delta v,\delta e)\in T_qO$ we can
associate the continuous path
$$
t\rightarrow (v+t\delta v,\pi_{v+t\delta v}(e+t\delta e))
$$
in $G\oplus E$. The  path starts at $q=(v,e)$ at time $t=0$ and
belongs to $O$ as long as  $v+t\delta v$ is close to $v$ and belongs
to $V$. Denote by $C_v\subset G$ the collection of all tangent
vectors $\delta v$ so that $v+t\delta v$ belongs to $V$ for $t> 0$
small enough. The set $C_v$ constitutes a partial quadrant in the
space $G$. Hence $C_v\oplus \ker(1-\pi_v)$ is a partial quadrant in
the space $G\oplus \ker( 1-\pi_v)$.  

Taking the derivative of the identity $\pi_v^2(e)=\pi_v (e)$ in the variable $v$ one  obtains  
$D_v\pi_v (\pi_v (e))\cdot \delta v+\pi_v (D_v \pi_v(e)\cdot \delta v)=D_v \pi_v (e)\cdot \delta v$ and one concludes, in view of 
$\pi_v (e)=e$, that 
$\pi_v (D_v\pi_v(e)\cdot \delta v)=0$. Assuming, in addition,  that $\pi_v (\delta e)=\delta e$ one finds  that 
$\delta e+D_v \pi_v (e)\cdot \delta v =\pi_v (\delta e+D_v \pi_v (e)\cdot \delta v)+D_v\pi_v (e)\cdot \delta v$. Thus,
$$
(\delta v, \delta e+D_v \pi_v (e)\cdot \delta v)\in T_qO
$$
for all $\delta v\in G$ and 
 $\delta e\in \ker (1-\pi_v)$, where $q=(v, e)$ satisfies $\pi_v (e)=e$. The map $A:G\oplus E\to G\oplus E$,  defined by 
$$A:(\delta v,\delta e)\mapsto  (\delta v,\delta e+
D_v\pi_v(e)\cdot \delta v \bigr), 
$$
maps  the space $ G\oplus\ker(1-\pi_v)$ sc-isomorphically onto the
tangent space $T_qO$ and consequently the partial quadrant 
$C_v\oplus \ker(1-\pi_v)$  onto the partial quadrant $C_q:=A(C_v\oplus \ker (1-\pi_v))$ in the tangent space $T_qO$.

\begin{defn}\label{defn5.6.4}
The partial quadrant $C_q\subset T_qO$ is called the {\bf partial
tangent quadrant} at the smooth point $q\in O$.
\end{defn}
The following result  is readily verified.
\begin{lem}
If  an sc-diffeomorphism $\varphi:O\rightarrow O'$ maps the smooth
point $q\in O$ to the smooth point $q'\in O'$, then  its tangent map
$T\varphi(q)$ at $q$ maps $C_q$ to $C_{q'}'$. If $N\subset T_{q}O$
is a  finite-dimensional smooth
subspace which is in good position to $C_q$ or which is neat with
respect to $C_q$,  then its image $T\varphi(q)N$ has the same property
with respect to $C_{q'}'$.
\end{lem}

The previous lemma allows to introduce the following intrinsic notions.
\begin{defn}\label{defn5.6.5}
Let $x$ be a smooth point  in a M-polyfold $X$ and let $\varphi:(U(x),x)\rightarrow (O,q)$ be a local chart around $x$.
\begin{itemize}
\item The {\bf  partial tangent quadrant} $Q_x$ in $T_xX$ is the
subset which is mapped under $T\varphi(x)$ onto the partial tangent
quadrant $C_q$ of $T_qO$.
\item A smooth finite-dimensional subspace $N$ of $T_xX$ is in
{\bf good position} to the corner structure of $X$ at $x$  (or in  {\bf good position} to the partial tangent quadrant
$Q_x$) if this is true for  the subspace $T\varphi(x)N$ with respect to $C_q$ in $T_qO$.
\item A smooth finite-dimensional subspace $N$ of $T_xX$ is {\bf neat}
with respect to the corner structure of $X$ at $x$ (or
{\bf neat} with respect to the partial tangent quadrant $Q_x$),  if the same holds for
$T\varphi(x)N$ with respect to $C_q$ in $T_qO$.
\end{itemize}
\end{defn}
\begin{rem} \label{rem0}
Since an sc-isomorphism $A:G\oplus E\to G\oplus E$ 
introduced above leaves the linear spaces $\{0\}\oplus \ker \pi_v$ invariant, the tangent spaces $T_qO$ at the smooth point $q=(v, e)\in O$ has $\ker \pi_v$ as natural sc-complement in $T_q(V\oplus E)$ so that 
$$
G\oplus E=T_q(V\oplus E)=T_qO\oplus \ker \pi_v.
$$
If $N\subset T_{(v,e)}O$ is in good position to the corner structure and
$N^\perp\subset T_qO$ is a good complement of $N$ in
$T_qO$, then $N^\perp \oplus \ker \pi_v$ is a good complement 
of $N$ in $G\oplus E$. In particular, $N\subset G\oplus E$ is in good position to the partial quadrant $C\oplus E$ in $G\oplus E.$
\end{rem}
Next we define what it means that a Fredholm section is in a good
position at a solution $q$ of $f(q)=0$.
\begin{defn}\label{defn5.6.6}
Consider a Fredholm section  $f$ of the fillable strong M-polyfold bundle
$p:Y\rightarrow X$ and assume that $f(q)=0$. Then $f$ is said to be
in {\bf good position to the corner structure}  at the point $x$
provided
\begin{itemize}
\item The linearization $f'(x)$ is surjective.
\item The kernel $N$ of $f'(x)$ is in good position to the corner
structure of $X$ at $x$.
\end{itemize}
\end{defn}

The previous concept of a good parametrization, adapted to the
situation with boundary, is useful to describe the solution set near
the boundary.

\begin{defn}\label{defn5.6.7}
Consider a local strong bundle $b:K\rightarrow O$ where $O$ has no
boundary. Assume that the section   $f$  of the  bundle $b$  is 
regularizing and at every solution $p\in O$ of $f(p)=0$  the section $f$ is   linearized Fredholm and its  linearization 
$f'(p):T_pO\to K_p$ is surjective.  Let $q\in \partial O$ be a solution of $f(q)=0$
at the boundary of $O$ and assume that the kernel $N$ of the linearization
$f'(q):T_qO\to K_q$ is in  good position to the partial tangent
quadrant $C_q\subset T_qO$. Then the section $f$ is said to have
  a {\bf good parametrization of the solution set} $\{f=0\}$ near $q$,
  if  there exist an open neighborhood $Q$ of $0$ in $N\cap C_q$  (which in view of
Proposition \ref{pointt} is a partial quadrant in $N$),
  an open neighborhood $U(q)$ of $q$ in $O$, and an sc-smooth
map
$$
A:Q\rightarrow N^{\perp}
$$
into  some good complement $N^{\perp}$ of $N$ in $G\oplus E$ such
that the following properties are satisfied.
\begin{itemize}
\item [(1)] $A(0)=0$ and $DA(0)=0$.
\item[(2)]  The map
$$
\Gamma:Q\rightarrow G\oplus E,\quad n\mapsto q+n+A(n)
$$
has its image in $U(q)$ and parameterizes all solutions $p\in U(q)$
of the equation $f(p)=0$.
\item [(3)] At every solution $p\in U(q)$ of $f(p)=0$,  the map
\begin{gather*}
\ker f'(q) \to  \ker  f'(p)\\
\delta n\mapsto  \delta n+DA(n_0)\cdot \delta n
\end{gather*}
is a linear isomorphism, where $p=q+n_0+A(n_0)$.
\end{itemize}
The map  $\Gamma$ is called a  {\bf good parametrization of the zero
set}  $\{f=0\}$  near  the zero $q\in \partial O. $
\end{defn}

 The calculus of good parameterizations in the
boundary case is similar to the previously discussed interior case.
Let us observe that,  in view of Proposition \ref{pointt}, the set
$C_q\cap N$ is a partial quadrant in $N$. Hence a good parametrization is
automatically parameterizing locally a manifold with boundary with
corners.

\subsection{Local  Solutions Sets in the Boundary
Case} The following main  result of this section guarantees the
existence of good parameterizations  of solution sets near  boundary
points.

\begin{thm}\label{LOCAX1}
Let $b:K \rightarrow O$ be a fillable  local strong M-polyfold bundle and
$f:O\to K$ a Fredholm section of the bundle $b$. Assume that $q$ is
a  smooth point in $\partial O$ solving the equation $f(q)=0$. If
the section $f$ is in good position to the corner structure 
at the point $q$ (in the sense of Definition \ref{defn5.6.6}),  then
there exists a good parametrization of the solution set $\{f=0\}$
near $q$.
\end{thm}
\begin{proof}
The proof goes along the same lines as the proof of Theorem
\ref{prop5.4.14} in the case without boundaries with some minor
modifications we describe next.
  Recall that the Fredholm property tells us that,  after a
possible change of coordinates,  there is a filler so that after
another change of coordinates we have a contraction germ. Hence,
without loss of generality we assume that $q=0\in \partial O$ and
that $f$ is already filled and
$$
f:O\subset ([0,\infty)^k\oplus{\mathbb R}^{n-k})\oplus W\to {\mathbb
R}^N\oplus W.
$$
Since we started with a section which at the point $q$ is in good
position to the corner structure of $O$,  the filled section is,  by Remark \ref{rem0}, in
good position to the corner structure at $0$. Moreover, if
$P:{\mathbb R}^N\oplus W\rightarrow W$ is the canonical projection,
we have the representation
$$
Pf(v,w)=w-B(v,w),
$$
where $v$ is near $0$ in the partial quadrant
$[0,\infty)^k\oplus{\mathbb R}^{n-k}$  of the space $\R^n$, and
where $w\in W$. The contraction property holds near $(0,0)$ where
the notion of near depends on the level $m$ in $W$. Thus, in view of 
Theorem \ref{newthm5.4},  there exists a
unique  map 
$$
v\rightarrow \delta(v)
$$
from a neighborhood  of $0$ in $[0,\infty)^k\oplus{\mathbb R}^{n-k}$
into $W$ solving the equation
$$
Pf(v,\delta(v))=0
$$
and satisfying
$\delta(0)=0$. The map $\delta$ has the same properties as  in the
case without boundary.  Introduce the mapping
$g:[0,\infty)^k\oplus{\mathbb R}^{n-k}\to \R^N$ defined near $0$ by
$$
g(v):=(1-P)f(v,\delta(v)).
$$
The map $g$ has,  as in in the case without boundary, a  surjective
linearization at $0$ and the kernel $N'$ of $g'(0):\R^n\to \R^N$ is
mapped by
$$
h\mapsto  (h,\delta'(0)h)
$$
isomorphically onto the kernel $N$ of the linearization $f'(0)$.
Abbreviate the partial quadrants
\begin{align*}
C&:=[0,\infty)^k\oplus{\R}^{n-k}\oplus W\quad \text{in $\R^n\oplus W$}\\
C'&:=[0,\infty)^k\oplus{\R}^{n-k}\quad \text{in $\R^n$}.
\end{align*}

Since $N$ is in good position to the partial quadrant
$[0,\infty)^k\oplus {\mathbb R}^{n-k}\oplus W$ we can take a good
complement  $N^\perp$ of $N$ in ${\mathbb R}^n\oplus W$. Consider
the linear $\ssc$-isomorphism
$$
{\mathbb R}^n\oplus W\rightarrow {\mathbb R}^n\oplus
W, \quad (h,k)\mapsto  (h,k+\delta'(0)h).
$$
It maps $C$ onto $C$.  Identifying  the kernel $N'\subset \R^n$ with $N'\oplus \{0\}$ in $\R^n\oplus W$, we see that he image of $N'$ under the above  map is  the kernel $N$.  Hence
the preimage ${(N')}^\perp$ of the good complement $N^\perp$ of $N$
is a good complement of $N'$ in ${\mathbb R}^n\oplus W$. In
particular,  $M={(N')}^\perp \cap {\mathbb R}^n\oplus\{0\}$ is a good
complement of $N'$ in ${\mathbb R}^n$.

 Since the kernel  $N'$ is in good position to $C'$ in ${\mathbb
R}^n$,  we see that $N'\cap C'$ is a partial quadrant in view of Proposition \ref{pointt} in Appendix.  By the
finite-dimensional implicit function theorem (one might first extend
the map $g$ to a $C^1$-map on an open neighborhood of $0$ in $\R^n$
by Whitney's extension theorem),  we obtain a solution germ of
$g(r+c(r))=0$ where
\begin{equation}\label{neweq5.6.1}
r\mapsto  r+c(r).
\end{equation}
is a function from $N'\cap C'$ into $N'\oplus M=\R^n$ defined near
$r=0$. In particular, $c(r)\in M$, moreover, $c(0)=0$ and $Dc(0)=0$.
As a side remark we observe that the function \eqref{neweq5.6.1} is
a good parametrization for our finite-dimensional problem in the
sense of the above  definition.  Define the map
\begin{gather*}
\Theta:N'\cap C'\to N'\oplus M\oplus W\\
\Theta (r)= r+c(r)+\delta(r+c(r)).
\end{gather*}
Take the bijective map $\sigma :N'\cap C'\to N\cap C$ defined by
$\sigma (r)=r+\delta '(0)r\in \R^n\oplus W$ and introduce the map
\begin{gather*}
\Gamma : N\cap C\to \R^n\oplus W\\
\Gamma (n)=\Theta\circ\sigma^{-1}(n).
\end{gather*}
Then (recall that $q=0$) we can write
$$
\Gamma (n)=q+n+A(n)
$$
with $A(n)\in N^{\perp}$ and $A(0)=0$ and $DA(0)=0$. Moreover,
$f(\Gamma (n))=0$, so that $\Gamma$ parameterizes the solution set
$\{f=0\}$ near $q=0$. Arguing as in the proof of Theorem
\ref{prop5.4.14} one sees that the map $\Gamma$ is sc-smooth.
Therefore, $\Gamma$ is a good parametrization and the proof of
Theorem \ref{LOCAX1} is complete.
\end{proof}

The above existence proofs of good parametrizations of solution sets of Fredholm sections prompts the following concept of a  finite dimensional  submanifold of an M-polyfold.  Such submanifolds have the structure of a manifold induced in a natural way. 

 \begin{defn}\label{submanifold}
Let $X$ be an M-polyfold and $M\subset X$ a subset  equipped with the
induced topology. The subset  $M$ is called a  {\bf finite dimensional submanifold} of $X$
provided the following holds.
\begin{itemize}
\item[$\bullet$]  The subset $M$ lies in $X_\infty$.
\item[$\bullet$] At  every point $m\in M$  there exists an  M-polyfold chart
$$(U, \varphi, (\pi,E,V))
$$
where  $m\in U\subset X$ and where $\varphi:U\rightarrow
O$ is a homeomorphism satisfying $\varphi (m)=0$, onto the  open neighborhood $O$ of $0$ in
the splicing core $K$ associated with the sc-smooth splicing $(\pi, E, V)$. Here $V$ is an open neighborhood of $0$ in a
partial quadrant $C$ of the sc-Banach space $ W$. Moreover,  there
exists a finite-dimensional smooth linear subspace $N\subset W\oplus
E$ in good position to $C$ and a  corresponding sc-complement
$N^\perp$, an open neighborhood $Q$ of $0\in C\cap N$ and an
$\ssc$-smooth map $A:Q\rightarrow N^\perp$ satisfying  $A(0)=0$, $DA(0)=0$
so that the map
$$
\Gamma:Q\rightarrow W\oplus E:q\rightarrow q+A(q)
$$
has its image in $O$ and the image of the composition
$\Phi:=\varphi^{-1}\circ\Gamma:Q\rightarrow U$ is equal to  $M\cap U$.
\item[$\bullet$] The map $\Phi:Q\rightarrow M\cap U$ is a homeomorphism.
\end{itemize}
We recall that  $N\subset W\oplus E$ is a smooth subspace  if  $N\subset (W\oplus E)_{\infty}=W_{\infty}\oplus E_{\infty}$. The map $\Phi:Q\rightarrow U$ is called a {\bf good parametrization}
of a neighborhood of $m\in M$ in $M$.
\end{defn}

In  other words  a  subset $M\subset X$   of an  M-polyfold $X$  consisting of smooth points
is a submanifold if for every $m\in M$ there is  a good
parametrization of an open neighborhood of $m$ in $M$. The following
proposition shows that the transition maps  $\Phi\circ \Psi^{-1}$ defined  by two good
parameterizations $\Phi$ and $\Psi$  are smooth, so the inverses of the good parametrizations define an atlas of smoothly compatible charts.  Consequently, a finite dimensional submanifold is in a natural way a manifold with boundary with corners.

\begin{prop} Any two parametrizations of a finite dimensional submanifold $M$ of the M-polyfold $X$ are smoothly compatible. 
\end{prop}
\begin{proof}
Assume that $m_0:=\varphi^{-1}(q_0+A(q_0))=\psi^{-1}(p_0+B(p_0))$
for two good parameterizations. Since both good parameterizations
are local homeomorphisms onto an open neighborhood of $m_0$ in $M$, 
we obtain a local homeomorpism $O(p_0)\rightarrow O(q_0)$, $p\mapsto 
q(p)$, where the domain and codomain are relatively open
neighborhoods in partial quadrants. We have
$$
q(p)+A(q(p))=\varphi\circ\psi^{-1}(p+B(p)).
$$
Recall that $q(p)\in N$ and $A(q(p)\in N^\perp$, where $N\oplus
N^\perp=W\oplus E$ is an sc-splitting. If  $P:W\oplus E\rightarrow
N$ is an  sc-projection along $N^{\perp}$, then 
$$
q(p)= P(\varphi\circ\psi^{-1}(p+B(p)).$$
The map $p\mapsto  q(p)$ is sc-smooth as a  composition of sc-smooth
maps. However,  since the domain and codomain
lie in finite dimensional smooth  linear spaces, the map is of class $C^{\infty}$.
\end{proof}
Examples of finite dimensional submanifolds of M-polyfolds are the solution sets of Fredholm sections in the case of transversality. We would like to mention that also the strong finite dimensional submanifolds introduced in Definition 3.19 of \cite{HWZ2} are finite dimensional submanifolds in the sense of Definition \ref{submanifold}. The induced manifold structures are the same.

\section{Global Fredholm Theory}\label{globalfred}

This section is devoted to the Fredholm theory in $M$-polyfold
bundles. Putting together the local studies of the previous section
it will be proved, in particular, that the solution set $f^{-1}(0)$
of a proper Fredholm section $f$ of a fillable strong  M-polyfold bundle $p:Y\to X$
carries in a natural way the structure of a smooth compact manifold
with boundary with corners, provided at every point $q\in f^{-1}(0)$
the section $f$ is in good position to the corner structure of $X$.

For the deeper study of Fredholm operators it is useful to introduce
first some auxiliary concepts.
\subsection{Mixed Convergence and Auxiliary Norms}\label{sec6.2} 
We begin with  a notion of convergence  in the strong M-polyfold bundle
$p:Y\to X$ called mixed convergence, referring to a mixture of
strong convergence in the base space $X$ and  weak convergence in
the level $1$ fibers $Y_{0,1}$, assuming the fibers are  reflexive
$\ssc$-Banach spaces. 
These spaces are characterized by the property that bounded sequences have weakly convergent  subsequences, a fact useful  in compactness proofs.
We first give  the local definition in a strong $M$-polyfold bundle chart. 

To do this we consider the local $M$-polyfold bundle 
$$K\rightarrow O$$
where $K=K^{\mathcal R}$ is the splicing core  
$$K=\{ (v, e, u)\in O\oplus F\vert \, \rho_{(v, e)}(u)=u\}$$
of the strong bundle splicing  $ {\mathcal R}=(\rho,F,(O,{\mathcal S}))$. Here $F$ is an sc-Banach space and $O$ is an open subset of the splicing core 
$
K^{{\mathcal S}}=\{ (v, e)\in V\oplus E\vert\ \pi_v(e)=e\}
$
associated with the splicing ${\mathcal S}=(\pi , E, V)$ in which $V$ is an open subset in the  partial quadrant $C$ of the sc-Banach space $G$.

Abbreviating the elements in $O$ by $x=(v, e)$ we consider a
sequence
$$(x_k,y_k)\in K_{0,1}=\{(x, y)\in O\oplus F_1\vert \, \rho_x (y)=0\}.$$

\begin{defn}[{\bf Mixed convergence}] \label{def6.3}
The sequence $(x_k,y_k)\in K_{0,1}$ is called ${\mathbf
m}${\bf-convergent} to  $(x, y)\in (G\oplus E)_0\oplus F_1$   if
\begin{align*}
&\text{{\em (1)}\quad \quad $x_k\to x$ in $O$}\\
&\text{{\em (2)}\quad \quad $y_k\rightharpoonup y$ in $F_1$.}
\end{align*}
 Symbolically,
$$(x_k, y_k)\xrightarrow{m} (x, y).$$
The half arrow  $\rightharpoonup$ denotes the  weak convergence.
\end{defn}
\begin{lem}\label{lem6.2.4}
If $(x_k, y_k)\in K_{0,1}\subset (G\oplus E)_0\oplus F_1$ and $(x_k,
y_k)\xrightarrow{m} (x, y)$, then $(x, y)\in K_{0,1}$.
\end{lem}
\begin{proof}
By assumption, $x_k=(v_k, e_k)$ and $\pi_{v_k}(e_k)=e_k$. Moreover,
$\rho_{x_k}(y_k)$ $ = y_k$. Since the inclusion  map $\imath :
F_1\to F_0$ is a compact operator, one concludes from the weak
convergence of $y_k\rightharpoonup y$ in $F_1$ that $y_k\to y$ in
$F_0$. Since $x_k\to x=(v, e)$,  it follows from the $\ssc^0$-continuity of the
projections $(x,y)\mapsto \rho_{x} (y)$ from $O\oplus F$ into $F$  that
$\rho_{x} (y)=y$ in $F_0$. Since $y\in F_1$,  one concludes
$\rho _{x} (y)=y$ in $F_1$ so that $(x, y)\in K_{0,1}$ as
claimed in the lemma.

\end{proof}

We shall  demonstrate next  that the concept of $m$-convergence is
compatible with $M$-polybundle maps provided the fibers are
reflexive spaces according to the following definition.
\begin{defn}\label{def6.5}
An  sc-Banach space $F$ is called {\bf reflexive} if all the spaces
$F_m$, $m\geq 0$, of its filtration are reflexive Banach spaces.
\end{defn}

In order to verify that the notion of mixed convergence is invariant under M-polyfold chart transformations we take  a strong  M-polyfold bundle map 
$$
\Phi:K^{\mathcal R}\rightarrow
K^{{\mathcal R}' }
$$
between the splicing cores associated with the strong bundle splicings ${\mathcal R}$
 and ${\mathcal R}'$.  We assume   that the underlying sc-Banach spaces $F$ and $F'$ are
reflexive. Abbreviating  $x=(v, e)\in O$,  the map $\Phi$ is of the form
$$
\Phi (x, y)= (\sigma(x),\varphi(x,y))
$$
where $\sigma:O\to O'$ is an sc-diffeomorphism and $\varphi$ linear in the fibers. 
Moreover,  $\Phi$ induces  two sc-diffeomorphisms $K^{\mathcal R}(i)\to K^{{\mathcal R}'}(i)$ for $i=0$ and $i=1$.  In the following we shall drop the indication of the strong splicings and simply write $K=K^{\mathcal R}$ and $K'=K^{{\mathcal R}'}$.

We consider a sequence $(x_k, y_k)\in K_{0,1}$. If $(x_k,
y_k)\xrightarrow{m} (x, y)$, then $(x, y)\in K_{0,1}$ in view of
Lemma \ref{lem6.2.4}, and $\sigma (x_k)\to \sigma (x)$. From $(x_k,
y_k)\to (x, y)$ in $K_{0,0}$ it follows that $\varphi  (x_k, y_k)\to
\varphi  (x, y)$ in $F'_0$, so that
$$
\Phi(x_k,y_k)\rightarrow \Phi(x,y)\:\:  \text{in $K'_{0,0}$.}
$$
By definition of a strong bundle map, the map $\Phi:K_{0,1}\to
K'_{0,1}$ is continuous. From the continuity we conclude $\lim_{k\to
\infty}\varphi (x_k, \rho_{x_k}y)=\varphi (x,\rho_x y)$ in $(F')_1$ for
all $y\in F_1$. By means of the uniform boundedness principle of linear functional analysis we
deduce that the sequence  of bounded linear operators
$$
T_k(h):=\varphi (x_k,\rho_{x_k}(h))
$$
belonging to $ {\mathcal L}(F_1, F'_1)$,   have uniformly bounded
operator norms so that  $\norm{T_k}\leq C$ for all $k$. Consequently,
$\norm{\varphi (x_k,y_k)}_1\leq C\cdot \norm{y_k}_1$. From the weak
convergence $y_k\rightharpoonup y$ in $F_1$ we know that
$\norm{y_k}_1$ is also a bounded sequence. Hence $\norm{\varphi
(x_k,y_k)}_1$ is a bounded sequence. Because $F_1'$ is a reflexive
Banach space, every subsequence of the bounded sequence $\varphi (x_k,
y_k)$ in $F_1'$ possesses a subsequence having a weak limit in
$F_1'$. The limit is necessarily  equal to $\varphi (x, y).$
Summarizing we have proved that
$$\sigma (x_k)\to \sigma (x)\:\:\:  \text{in $O'$}$$
and
$$\text{pr}_2\circ \Phi (x_k, y_k)\rightharpoonup
\text{pr}_2\circ \Phi (x, y) \quad \text{in $(F')_1$}$$
 i.e., on level $1$.

 The discussion shows that the {\bf mixed
 convergence is invariant under $M$-polyfold chart
 transformations}  and hence  an intrinsic concept for
 $M$-polyfold bundles having reflexive fibers so that
 we can introduce the following definition.
\begin{defn}\label{def6.6}
If $p:Y\to X$ is a strong $M$-polyfold bundle having  reflexive fibers, then
a sequence $y_k\in Y_{0,1}$ is said to {\bf converge in the m-sense}
to $y\in Y_{0,1}$, symbolically
$$y_k\xrightarrow{m}y\:\:\:  \text{in $Y_{0,1}$},$$
if the underlying sequence $x_k=p(y_k)\in X$ converges in $X$ to an
element $x$ and if there exists a strong $M$-polyfold bundle chart
$\Phi$ around the point $x\in X$ satisfying
$$\text{pr}_2\circ \Phi (x_k, y_k)\rightharpoonup \text{pr}_2\circ \Phi (x, y)$$
weakly in  $(F')_1$, i.e., on level $1$.
\end{defn}

As shown above, the definition does not depend on the choice of the
local $M$-polyfold bundle chart.

For the general perturbation theory we introduce another
concept in order  to estimate  the size of perturbations.
\begin{defn}
An {\bf auxiliary norm}  $N$ for the strong  M-polyfold bundle
$p:Y\rightarrow X$ consists of a continuous map
$N:Y_{0,1}\rightarrow [0,\infty)$ having the following properties.
\begin{itemize}
\item For every $x\in X$, the induced map $N|{(Y_{0,1})}_x\rightarrow
[0,\infty)$ on the fiber ${(Y_{0,1})}_x$ is a complete norm.

\item If $y_k\xrightarrow{m} y$, then
$$
N(y)\leq \liminf_{k\to \infty} N(y_k).
$$

\item If $N(y_k)$ is a bounded sequence and the underlying
sequence $x_k$ converges to $x\in X$,  then $y_k$ has an  m-convergent
subsequence.
\end{itemize}
\end{defn}
Using the
paracompactness of $X$,  one  can easily construct auxiliary norms as follows.
\begin{prop}\label{newprop6.8}
Let $p:Y\rightarrow X$ be a  fillable strong M-polyfold bundle having  reflexive
fibers.  Then there exists an  auxiliary norm.
\end{prop}
\begin{proof}
 Construct for every
$x\in X$ via local coordinates a norm $N_{U(x)}$ for $Y_{0,1}\vert
U(x)$ where $U(x)$ is a small open neighborhood of $x\in X$. This is
defined by $N_{U(x)}(y)=\parallel pr_2\circ \Psi(y)\parallel_1$,
where $\Psi$ is a strong M-polybundle chart and $pr_2$ is the
projection onto the fiber part. Observe that m-convergence of $y_k$
to some $y$ in $Y_{0,1}\vert U(x)$ implies weak convergence of
$pr_2\circ \Psi(y_k)$. Using the convexity of  the norm and standard
properties of weak convergence we conclude
$$
N_{U(x)}(y)\leq \liminf_{k\to \infty} N_{U(x)}(y_k).
$$
At this point  we have 
local expressions for auxiliary norms which cover $X$. Using the
paracompactness of $X$ we can find a subordinate partition of unity
$(\chi_{\lambda})_{\lambda \in \Lambda}$ and define
$$
N = \sum \chi_{\lambda} N_{U(x_{\lambda})}.
$$
Since the family $(\chi_{\lambda})$ is locally finite, the above  sum is well-defined. If $N(y_k)$ is bounded and the underlying  sequence $x_k$ converges
to some $x\in X$ it follows in local coordinates that the
representative of $y_k$ is bounded on level $1$. In view of the
reflexivity, the sequence  $y_k$ possesses an  m-convergent
subsequence. Hence $N$ is an  auxiliary  norm and the proof of
Proposition \ref{newprop6.8} is complete.
 \end{proof}
 
The following remarks will be important in the proof of the  compactness Theorem \ref{thm6.9}.
 Consider a local strong M-polyfold bundle $K^{\mathcal R}\to O$ in which $O$ is an open subset of the splicing core $K^{\mathcal S}$ associated with the splicing ${\mathcal S}=(\pi, E, V)$. The splicing core  $K^{\mathcal R}=\{((v, e), u)\in O\oplus F\vert \, \rho_{(v, e)}(u)=u\}$ is associated with the strong bundle splicing 
${\mathcal R}=(\rho, F, (O, {\mathcal S}))$.   Let $N:(K^{\mathcal R})_{0,1}\to [0,\infty )$ be an auxiliary norm of the bundle $K^{\mathcal R}\to O$. Associated with  the complementary splicing ${\mathcal R}^c=(1-\rho, F, (O, {\mathcal S}))$,  is  the local strong M-poylofd bundle 
$K^{{\mathcal R}^c} \to O$. In view of Proposition \ref{newprop6.8},  the  bundle $K^{{\mathcal R}^c}\to O$ can be equipped with an auxiliary norm $N^c:(K^{{\mathcal R}^c})_{0,1}\to [0,\infty )$.

Recalling that   $O\triangleleft F$ is the sc-Banach space $O\oplus F$ equipped with the bi-filtration $O_{m}\oplus F_{k}$ where $m\geq 0$ and $0\leq k\leq m+1$, and that 
$O\triangleleft F=K^{\mathcal R}\oplus_{O}K^{{\mathcal R}^c}$, we define the function $\ov{N}:O_0\oplus F_1\to [0,\infty )$ by 
\begin{equation}\label{eq00}
\ov{N}(x, h)=N(x, \rho_x(h))+N^c (x, (1-\rho_x)(h))
\end{equation}
where $x\in O_0$ and $h\in F_1$.
\begin{lem} The function $\ov{N}:O_0\oplus F_1\to [0,\infty )$ is an auxiliary norm.
\end{lem}
\begin{proof}
Clearly, $\ov{N}$ is continuous and for fixed $x\in O_0$, the restriction $\ov{N}(x, \cdot ):F_1\to [0,\infty )$ defines a complete norm. Assume that $(x_n, h_n)\xrightarrow{m} (x, h)$. Then $x_n\to x$ and $h_n\rightharpoonup  h$ in $F_1$. Since $F_1$ is compactly embedded in $F_0$, we have $h_n\to h$ in $F_0$ and since $\rho:O_0\oplus F_0\to F_0$ is continuous, it follows that $\rho_{x_n}(h_n)\to \rho_{x}(h)$ in $F_0$. 
By assumption,  ${\mathcal R}=(\rho, F, (O, {\mathcal S}))$ is a strong bundle splicing.  Hence  ${\mathcal R}^1=(\rho, F^1, (O, {\mathcal S}))$ is a general splicing. 
This implies implies that $\rho: O_0\oplus F_1\to F_1$ is continuous and $\rho_{x_n}:F_1\to F_1$ is a bounded linear projection for every $x_n$. Since the sequence $(h_n)$ being weakly convergent is bounded in $F_1$, it follows from the uniform boundedness principle that the sequence $(\rho_{x_n}(h_n))$ is bounded in $F_1$.  After possibly taking a subsequence we may assume that  there is $h'\in F_1$ such that $\rho_{x_n}(h_n)\rightharpoonup  h'$ in $F_1$.  Therefore,  $\rho_{x_n}(h_n)\to h'$ in $F_0$ and hence    $h'=\rho_x (h)$ and  $\rho_x (h)\in F_1$.  We conclude that $\rho_{x_n}(h_n)\rightharpoonup \rho_x (h)$ and $(1-\rho_{x_n})(h_n)\rho\rightharpoonup (1-\rho_{x}(h))$ in $F_1$. Consequently, 
\begin{equation*}
\begin{split}
\ov{N}(x, h)&\leq \ov{N}(x, \rho_{x}(h))+\ov{N}(x, (1-\rho_{x})(h))\\
&=N(x, \rho_{x}(h))+N^c(x, (1- \rho_{x})(h))\\
&\leq \liminf N(x_n, \rho_{x_n}(h_n))+ \liminf N^c(x_n, (1-\rho_{x_n})(h_n))\\
&\leq 
\liminf [N(x_n, \rho_{x_n}(h_n))+N^c(x_n, (1-\rho_{x_n})(h_n))]\\
&=\liminf \ov{N}(x_n, h_n)
\end{split}
\end{equation*}
as claimed. Finally, assume that $\ov{N}(x_n, h_n)$ is bounded and $x_n\to x$. 
Then $N(x_n, \rho_{x_n}(h_n))$ and $N^c(x_n, (1-\rho_{x_n})(h_n))$ are bounded and since $N$ and $N^c$ are auxiliary norms, it follows that after possibly taking subsequences 
$\rho_{x_n}(h_n)\rightharpoonup  h'$ and $(1-\rho_{x_n})(h_n)\rightharpoonup  h''$ in $F_1$. Hence $h_n=\rho_{x_n}(h_n)+(1-\rho_{x_n})(h_n)\rightharpoonup  h=h'+h''$ in $F_1$. This completes the proof of the lemma.
\end{proof}

The function  $\ov{N}:O_0\oplus F_1\to [0,\infty )$ is continuous and for  fixed $x\in O_0$, the restriction $\ov{N}(x,\cdot )$ to the space $F_1$  is a complete norm on $F_1$. It follows that there exists  a positive  constant $C_x$ such that 
$$\ov{N}(x, h)\leq C_x\cdot \norm{h}_1$$ 
for all $h\in F_1$. Therefore,  by the bounded inverse theorem of linear  functional analysis applied to the identity map 
from the Banach space $(F_1, \norm{\cdot}_1)$ to the Banach space $(F_1, \ov{N}(x,\cdot ))$, the two norms 
$\norm{\cdot}_1$ and $\ov{N}(x,\cdot )$ are equivalent, i.e., there exist two  positive constants $c_x<C_x$ such that 
\begin{equation}\label{equiv01}
c_x\cdot \norm{h}_1\leq \ov{N}(x, h)\leq C_x\cdot \norm{h}_1
\end{equation}
for  all $h\in F_1$. The next lemma claims a local uniform estimates in $x\in O$.

\begin{lem}\label{lemequiv01}
Let $\ov{N}:O_0\oplus F_1\to [0, \infty )$  be an auxiliary norm as defined \eqref{eq00}. Fix a   point  $x_0=(v_0, e_0)\in O_0$. Then there exist two positive constants $c<C$ and an open neighborhood $U\subset O$ of $x_0$ 
such that 
$$c\norm{h}_1\leq \ov{N}(x, h)\leq C\norm{h}_1$$
for all $x\in U$  and all $h\in F_1$.
\end{lem}
\begin{proof}
By the continuity of $\ov{N}$ we find an open neighborhood $V_1\subset O$ of $x_0$ and a constant $C>0$ such that $\ov{N}(x, h)\leq C\norm{h}_1$ for all $x\in V_1$ and $h\in F_1$. We claim that there exist a positive constant $c_0$ and an open neighborhood $V_2\subset O$ of $x_0$ such that 
\begin{equation}\label{eq000}
c_0\cdot \ov{N}(x_0, h)\leq \ov{N}(x, h)
\end{equation}
for all $x\in V_2$ and $h\in F_1$, so that the lemma follows from \eqref{equiv01}. In order to prove the estimate \eqref{eq000}  we
 argue by contradiction and assume that  there are two sequences $(x_n)\subset O$  and $(h_n)\subset F_1$ satisfying  $x_n\to x_0$, $\norm{h_n}_1=1$,  and 
$$\dfrac{1}{n}\ov{N}(x_0, h_n)\geq \ov{ N}(x_n, h_n)\quad \text{for $n\geq 0$}.$$
Set $h_n'=h_n/\ov{N}(x_n, h_n)$. Since $\ov{N}(x_n, h_n')=1$ and $x_n\to x_0$, it follows from  the property (3) of the auxiliary norm that,  after  possibly passing to a subsequence, 
the sequence $(h_n')$ converges weakly in $F_1$.  Hence  the sequence $\norm{h_n}_1$ is bounded and since the norms $\norm{\cdot }_1$ and  $N(x_0,\cdot )$ are equivalent on $F_1$, it follows that  the sequence $N(x_0, h_n' )$ is bounded.  Hence $\frac{1}{n}\ov{N}(x_0,h_n')\to 0$ contradicting our assumption  $\frac{1}{n}\ov{N}(x_0,h_n')\geq 1$.
The proof of the lemma is complete.
\end{proof}

We recall that  the open set $\wh{O}\subset V\oplus E$ is defined by  $\wh{O}=\{(v, e)\in V\oplus E\vert \, (v, \pi_v e)\in O\}$.  The auxiliary norm $\ov{N}:O\oplus F_1\to [0,\infty )$ can be extended to $\wh{O}\oplus F_1$  by  defining 
\begin{equation}\label{newnorm1}
\wh{N}((v, e), h):=\ov{N}((v, \pi_v (e)), h)\quad \text{for all  $(v, e)\in \wh{O}$ and $h\in F_1$.}
\end{equation}

The following result will be useful in compactness investigations.
\begin{thm} [{\bf Local Compactness}] \label{thm6.9}
Consider the fillable strong $M$-polyfold bundle $p:Y\to X$ having reflexive
fibers and let   $f$ be a Fredholm section of $p$. Assume that an
auxiliary norm $N:Y_{0,1}\to [0,\infty )$ is given. Then, for every smooth point $q\in X$,  there
exists an open neighborhood $U(q)$ in $X$ so that the following holds.
\begin{itemize}
\item The subset $Z\subset X$ defined by
$$
Z=\{x\in\overline{U(q)}\vert\ \text{$f(x)\in Y_{0,1}$  \text{and}
$N(f(x))\leq 1$}\}
$$
is a compact subset of $X$.
\item Every sequence $(x_k)$  in $\overline{U(q)}$ satisfying $f(x_k)\in Y_{0,1}$ and 
$$
\liminf_{k\to \infty}N(f(x_k))\leq 1
$$
possesses a  convergent subsequence.
\end{itemize}
\end{thm}

\begin{proof}
The theorem is of local nature. Hence we shall study in a strong M-polyfold chart the Fredholm section 
$f$ of the fillable  strong  local M-polyfold bundle $p:K^{\mathcal R}\to O$ where  as above $O$ is an open subset of the splicing core $K^{\mathcal S}$ and ${\mathcal R}=(\rho, F, (O, {\mathcal S}))$ a strong bundle splicing in which $(O, {\mathcal S})$ is the local model of the M-polyfold $X$. Let $x_0\in O$ be the  smooth point representing $q\in X$ in the local chart.  In view of  Definition \ref{defoffred} of a  Fredholm section, there exists a  filled version $\ov{f}:\wh{O}\to \wh{O}\triangleleft F$ of the section $f$ 
of the form 
$$\ov{f}(x)=f(x)+f^c(x)$$
for $x$ contained in the open set $\wh{O}\subset V\oplus E$ defined by  $\wh{O}=\{(v, e)\in V\oplus E\vert \, (v, \pi_v e)\in O\}$ where $f^c$ is the filler.  In the local chart we have  the auxiliary norms $N:(K^{\mathcal R})_{0,1}\to [0,\infty )$ and  $N^c:(K^{{\mathcal R}^c})_{0,1}\to [0,\infty )$. We define by the formula \eqref{eq00} the auxiliary norm $\ov{N}:O\oplus F_1\to [0,\infty )$ and finally obtain the auxiliary norm $\wh{N}((v, e), h):=\ov{N}((v, \pi_v (e)), h)$ on $\wh{O}\oplus F_1$. 
We claim that in order to prove Theorem \ref{thm6.9} for the section $f$, it  is sufficient to prove the following statement  
for the filled section $\ov{f}$ of the bundle $\wh{O}\triangleleft F\to \wh{O}$.\\[2pt]

{\bf ($\ast$)}\quad There exists an open neighborhood $\wh{U}(x_0)\subset \wh{O}$ of $x_0$ having the property that every sequence $(\wh{x}_n)\subset \wh{U}(x_0)$ satisfying $\ov{f}(\wh{x}_n )\in \wh{O}\oplus F_1$ and 
$\liminf \wh{N}(\ov{f}(x_n))\leq 1$ possesses a convergent subsequence.\\[2pt]

Indeed, assume that ($\ast$) holds true and define the open neighborhood $U(x_0)=\wh{U}(x_0)\cap  O$. Take a sequence $x_n\in U(x_0)$ satisfying $f(x_n)\in O\oplus F_1$ and $\liminf N(f(x_n))\leq 1$. Since $x_n\in O$, it follows from the definition that the filler term vanishes, $f^c(x_n)=0$, and hence $\wh{N}(\ov{f}(x_n))=N(f(x_n))+N^c(f^c(x_n))=N(f(x_n))$. Therefore, $\liminf \wh{N}(\ov{f}(x_n))\leq 1$ so that by  ($\ast$) the  sequence $(x_n)$ has indeed a converging subsequence as claimed.

It remains  to prove the statement ($\ast$) for the filled section. By Definition \ref{defoffred}, after a  further  change of coordinates,  the filled section if of the form 
$$f:V\oplus W\to \R^N\oplus W$$
where $V\subset \R^n$ is a partial quadrant and $f$ is defined in a sufficiently small closed neighborhood $\ov{U}_0$ of the origin  in $(V\oplus W)_0$.  Moreover, if $P:\R^N\oplus W\to W$ is the canonical projection, then 
$$
P[f(a,w)-f(0)]=w-B(a, w).
$$
is an $\ssc^0$-contraction germ in the sense of Definition \ref{def5.1}. In addition, we have the transported auxiliary norm $\wh{N}$ defined on $\ov{U}_0\oplus (\R^N\oplus W)_1$.  In view of  Lemma \ref{lemequiv01}, possibly choosing a smaller neighborhood,  we may assume that $\wh{N}(x, h)=\norm{h}_1$  on $(\R^N\oplus W)_1$.

We  now consider
  the set of $(a, w)\in \ov{U}_0$ satisfying $f(a, w)\in (\R^N\oplus W)_1$ and
 $$\norm{f(a, w)}_1\leq 1.$$
 The section $f$ splits according to the splitting of the target space into
 \begin{equation*}
 f(a, w)=\begin{bmatrix}(1 -P)f(a, w)\\
 Pf(a, w)\end{bmatrix}=
 \begin{bmatrix}(1-P)f(a, w)\\
 w-B(a, w)+Pf(0)\end{bmatrix}.
 \end{equation*}
Using the contraction property one finds for given $u\in W$ close to
$Pf(0)$ and given $a\in V$ close to $0$ a unique $w(a, u)\in W$
solving the equation $Pf(a, w(a, u))=u$. Moreover, the map $(a,
u)\mapsto w(a, u)$ is continuous on the $0$-level. Now, given a
sequence $(a_k, w_k)\in \ov{U}_0$ such that
$$
 f(a_k, w_k)=:(b_k, u_k)
$$
belongs to $(\R^N\oplus W)_1$ and satisfies
$$
\liminf_{k\rightarrow\infty}\ \norm{f(a_k, w_k)}_1\leq 1,
$$
 we have
the equations
\begin{gather*}
Pf(a_k, w_k)=u_k\\
(1 -P)f(a_k, w_k)=b_k\\
w_k=w(a_k, u_k).
\end{gather*}
We shall show that $(a_k, w_k)$ possesses a convergent subsequence
in $\ov{U}_0$. By assumption, $(\R^N\oplus W)_1$ is a reflexive
Banach space so that going over to a subsequence and using the compact embedding $W_1\subset W_0$,
\begin{gather*}
(b_k, u_k)\rightharpoonup (b', u')\quad \text{in $(\R^N\oplus W)_1$}\\
(b_k, u_k)\rightarrow (b', u')\quad \text{in $(\R^N\oplus W)_0$},
\end{gather*}
and $\norm{(b', u')}_1\leq 1$. In $\R^N$ we may assume $a_k\to a'\in
V$. Consequently, $w_k=w(a_k, u_k)\to w':=w(a', u')$. Hence $f(a',
w')=(b', u')$ and $\norm{f(a', w')}_1\leq 1$. The proof of
 Theorem \ref{thm6.9} is complete.
\end{proof}

\subsection{Proper Fredholm Sections}

In this section we introduce the important class of proper Fredholm
sections.
\begin{defn}
A Fredholm section $f$ of the fillable strong M-polyfold bundle  $p:Y\rightarrow X$ is  called {\bf proper}
provided $f^{-1}(0)$ is compact in $X$.
\end{defn}
The first observation is the following.
\begin{thm}[$\infty$-Properness]\label{infty-p}
Assume that $f$ is a proper  Fredholm section of the fillable strong
M-polyfold bundle $p:Y\rightarrow X$. Then $f^{-1}(0)$ is compact in
$X_{\infty}$.
\end{thm}
\begin{proof}
Properness implies by definition that $f^{-1}(0)$ is compact on
level $0$. Of course, $f^{-1}(0)$ is a subset of $X_{\infty}$ since
$f$ is regularizing. Assume that $(x_k)$ is a sequence of solutions of
$f(x)=0$. We have to show that it possesses  a subsequence
converging to some solution $x_0$ in $X_\infty$. After taking a
subsequence we may assume that $x_k\rightarrow x_0$ on level $0$. We
choose  a contraction germ representation for $[f,x_0]$. After a
suitable change of coordinates the sequence $((a_k,w_k))$ corresponds
to $(x_k)$, the point $(0,0)$ to $x_0$,  and
$$
w_k=B(a_k,w_k),
$$
with $(a_k,w_k)\rightarrow (0,0)$ on level $0$. Consequently,
$$
w_k=\delta(a_k)
$$
for the map $a\rightarrow \delta(a)$ constructed  using  Banach's
fixed point theorem on level $0$. We know that for every level $m$
there is an open neighborhood $O_m$ of $0$ so that
$\delta_m:O_m\rightarrow W_m$ is continuous  given by Banach's fixed
point theorem on the $m$-level. We may assume that $O_{m+1}\subset
O_m$. Fix a level $m$. For $k$ large enough,  we have $a_k\in O_m$.
Then $\delta_m(a_k)$ is the same as $\delta(a_k)=w_k$. Hence
$a_k\rightarrow 0$ implies $w_k\rightarrow 0$ on level $m$.
Consequently, $(a_k,w_k)\rightarrow (0,0)$ on every level, implying
convergence on the $\infty$-level. Hence $x_k\rightarrow x_0$ in
$X_{\infty}$ as claimed.
\end{proof}
As a consequence of the local Theorem \ref{thm6.9} we obtain the
following global result for proper Fredholm sections.
\begin{thm}\label{thm6.12}
Let $p:Y\to X$ be a fillable strong $M$-polyfold bundle with reflexive fibers
and assume that $f$ is a proper  Fredholm
section. Assume that $N$ is a given auxiliary norm. Then there
exists an open neighborhood $U$ of the compact set $S=f^{-1}(0)$ so
that the following holds true.
\begin{itemize}
\item[$\bullet$] For every section $s\in \Gamma^+(p)$ having
its support in $U$ and satisfying $N(s(x))\leq 1$, the section $f+s$
is a proper  Fredholm section  of the fillable strong M-polyfold bundle $p^1:Y^1\rightarrow X^1$ and $(f+s)^{-1}(0)\subset U.$
\item[$\bullet$] Every sequence $(x_k)$ in $\ov{U}$
satisfying $f(x_k)\in Y_{0,1}$ and
$$\liminf_{k\to \infty} N(f(x_k))\leq 1$$
possesses a convergent subsequence.
\end{itemize}
\end{thm}
\begin{proof}
We  know from the local Fredholm  theory (Theorem \ref{prop5.21})
that $f+s$ is a Fredholm  section of the
fillable bundle $p^1:Y^1\to X^1$  if $s$ is an $\ssc^+$-section.
For every $q\in S=f^{-1}(0)$ there exists an open neighborhood
$U(q)$ having the properties as  described in Theorem  \ref{thm6.9}.
Since $S$ is a compact  set  we find finitely many $q_i$ so that the
open sets $U(q_i)$ cover $S$. We denote their union by $U$. Next
assume  that the  support of $s\in \Gamma^+(p)$  is contained  in
$U$. If $(f+s)(x)=0$, then necessarily $x\in U$ because otherwise
$s(x)=0$ implying $f(x)=0$ and hence $x\in S\subset U$, a
contradiction. Consequently, the set of solutions of $(f+s)(x)=0$
belongs to $\{x\in U\vert \ \text{$f(x)\in Y_{0,1}$ and
$N(f(x))\leq 1$}\}$ which by construction  and Theorem \ref{thm6.9} is contained in a finite
union of compact sets. The proof of Theorem \ref{thm6.12} is
complete.\end{proof}

\subsection{Transversality and Solution Set}\label{sec6.4} From our
local considerations in the previous chapter  we shall deduce the
main results about the solution set of Fredholm sections.
\begin{defn}
A Fredholm section  $f$ of the fillable strong M-polyfold bundle $p:Y\rightarrow
X$ is said to be {\bf transversal} to the zero section if at every solution $x$ of
$f(x)=0$ the linearization $f'(x):T_xX\rightarrow Y_x$ is surjective.
\end{defn}
The first result is concerned with a proper Fredholm section of a fillable
strong M-polyfold bundle $p:Y\rightarrow X$ where $X$ has no boundary.

\begin{thm}\label{transthm}
Let $f$ be a proper and transversal Fredholm section of the fillable strong
M-polyfold bundle $p:Y\rightarrow X$ where $X$ has no boundary.
Then the solution set ${\mathcal M}=f^{-1}(0)$ is a smooth compact
manifold (without boundary).
\end{thm}
\begin{proof}
Near a solution $x$ of $f(x)=0$ we take suitable strong local bundle
coordinates $ \Phi:Y\vert U(x)\rightarrow K$ covering an
sc-diffeomorphism $\varphi:U(x)\rightarrow O$ satisfying $\varphi(x)=0$. Here $K\to O$ is the fillable local strong M-polyfold bundle. The local expression $f_{\Phi}$ of the Fredholm  section $f$ is also a  Fredholm section,  has $0$ as a solution,  and its linearization  $f_{\Phi}'(0)$  is, by assumption,  a surjective sc-Fredholm
operator. Hence,   by Theorem \ref{newtheoremA},  there exists a good parametrization of the solution set near $0$,
$$
\Gamma_\Phi(n)=n+A(n)
$$
defined on an open neighborhood of $0$ in the kernel of $f'_{\Phi}(0)$.
Then we define on a small open neighborhood of $0\in\ker f'(0)$ the
map  $\Gamma_x=\varphi^{-1}\circ\Gamma_\Phi\circ T\varphi (x)$. We can
carry out this construction near every  solution $x$ of $f(x)=0$. 
By the calculus of good parameterizations in section \ref{goodpar}, the transition maps $\Gamma_x^{-1}\circ \Gamma_y$ are sc-smooth and since they are defined and take values in smooth  finite dimensional  vector spaces they  are of class $C^{\infty}$. Consequently,  the inverses of the maps $\Gamma_x$  define a smooth atlas for ${\mathcal M}$.

\end{proof}

In the case where the ambient M-polyfold has a boundary,  the solution
set of a proper and transversal Fredholm section $f$ of the bundle
$p:Y\rightarrow X$ is in general not a smooth compact manifold with
boundary due to a possible bad position of the solution set to the
boundary. In order to draw conclusions about the nature of the
solution set we need some  knowledge about the position of the kernels  of the linearized
operators at solutions on  the boundary. Some of these positions are
in some sense generic and genericity questions are being
investigated in the next section. We continue with a definition.
\begin{defn}\label{gposition}
Let $f$ be a Fredholm section of the  fillable strong M-polyfold bundle
$p:Y\rightarrow X$, where $X$ possibly has a  boundary.  The section $f$
is  called in {\bf good position} if at every point $x\in f^{-1}(0)$ the
Fredholm section $f$  is in good position to the corner structure
of $X$  as defined in  Definition \ref{defn5.6.6}. Recall that this requires,  in particular,  that
  the linearisation $f'(x)$ is surjective at  every solution $x$ of $f(x)=0$.
\end{defn}
The next result gives  a generalization of the previous theorem. We
consider the situation of a proper Fredholm section on a M-polyfold
with possible boundary.
\begin{thm}\label{thm6.13}
Consider  a proper Fredholm section $f$ of the fillable strong M-polyfold
bundle $p:Y\rightarrow X$ which is in good position to the corner
structure. Then the solution set $f^{-1}(0)$ carries in a natural
way the structure of a smooth compact manifold with boundary $\partial  f^{-1}(0)$  with
corners where  $\partial  f^{-1}(0)=f^{-1}(0)\cap \partial X$.
\end{thm}

The proof of the theorem is almost identical to the proof of 
Theorem \ref{transthm}  with the exception that, in addition, the  local expression $f_{\Phi}$  is in  good position to the corner structure of $O$ at $0$ and we use Theorem \ref{LOCAX1} to guarantee the existence of  the parametrization $\Gamma_{\Phi}$ in this case.

\begin{proof}
Near a solution $x$ of $f(x)=0$ we take suitable strong local bundle
coordinates,
$$
\Phi:Y|U(x)\rightarrow K
$$
covering an sc-diffeomorphism $\varphi:U(x)\rightarrow O$ satisfying $\varphi (x)=0$, where
$K\rightarrow O$ is the  local strong M-polyfold bundle. The local expression $f_{\Phi}$  of the section $f$ has
$0$ as a solution and $f_\Phi$, by assumption, is in good position
to the corner structure of $O$ at  the point $0$. Hence, by Theorem \ref{LOCAX1}, there is a good parametrization of the solution set near $0$,
$$
\Gamma_\Phi(n)=n+A(n).
$$
This parametrization is defined on a relatively open neighborhood of $0$
in a partial quadrant $P:=\ker f'_{\Phi}(0 )\cap C$ of the kernel $\ker f'_{\Phi}(0 )$. We define the map $\Gamma_x=\varphi^{-1}\circ\Gamma_\Phi\circ T\varphi(x)$
on a small relatively open neighborhood of $0$ in a partial quadrant
contained in $\ker(f'(x))\subset T_xX$. We can carry out this
construction near every solution $x\in X$ of $f(x)=0$. By our
discussion of the calculus of good parameterizations the transition
maps are smooth ( i.e., of class $C^{\infty}$). Hence the inverses of the maps $\Gamma_x$ define a
smooth atlas equipping  the solution set $f^{-1}(0)$ with the structure of a smooth
compact manifold with  boundary with corners.
\end{proof}

If we have a proper Fredholm section of a fillable strong M-polyfold bundle we can
introduce the notion of being in general position. Intuitively this
means two things. First of all it requires  that at a solution the linearisation
is surjective. Secondly the solution set should  avoid finite intersections of
(local) faces if the dimension of this intersection has too large
codimension. For example,  a one-dimensional solution set would only
hit good  boundary points ($d(x)=1$) and would not hit corners
($d(x)>1$) and similarly a two-dimensional solution set would not
hit boundary points $x$ with $d(x)\geq 3$. The importance of the
notion of being in general position comes from the fact that it can
be achieved by an arbitrarily small perturbation as shown in
section \ref{sec6.5}. We would, however, like to point out that
in SFT sometimes certain symmetry requirements do not allow to bring
a solution set into general position by a perturbation compatible
with the symmetries. Nevertheless one can still achieve a good
position. This will be studied  in \cite{HWZ5} and in an entirely abstract
framework in \cite{HWZ6}.

Consider a proper Fredholm section $f$ of a fillable strong M-polyfold bundle
$p:Y\rightarrow X$ where the M-polyfold $X$ possibly has a  boundary.  If $x$ is a point in $X$ it belongs to $d(x)$-many local
faces.  Fixing
 a smooth point $x$, we  consider the $d(x)$ local faces numbered as ${\mathcal F}^1, \ldots , {\mathcal F}^{d}$ where $d=d(x)$. Each local face ${\mathcal F}^j$ has the  tangent space
$T_x{\mathcal F}^j$ at the point $x\in X$. We denote by $T^\partial_xX$ the intersection
$$
T^\partial_x X=\bigcap_{1\leq j\leq d} T_x {\mathcal F}^j.
$$
If $d(x)=0$, i.e., if $x$ is an interior point, we define
$T^\partial_x X=T_x X$.
\begin{defn}\label{generalpos}
The Fredholm section $f$  is in {\bf general position} to the boundary $\partial X$ provided the following
holds for every  solution $x$ of $f(x)=0$.
\begin{itemize}
\item[(i)] The linearization $f'(x)$ is
surjective.
\item[(ii)] The kernel
of $f'(x)$ is transversal to $T^\partial_x X$ in the tangent space $T_xX$.
\end{itemize}
\end{defn}

\begin{thm}\label{manwbc}
Assume that the proper Fredholm section $f$ of the fillable strong M-polyfold
bundle $p:Y\rightarrow X$  is in general position to the boundary $\partial X$. Then the solution set $f^{-1}(0)$ is a smooth compact manifold with boundary with corners.
\end{thm}
This theorem is an immediate consequence of Theorem \ref{thm6.13}
and the following observation.
\begin{lem}\label{nl5.17}
If the Fredholm section $f$ of the fillable strong M-polyfold bundle
$p:Y\rightarrow X$ is in general position,  then it is in good
position as well.
\end{lem}
\begin{proof}
Take a solution $x$ of $f(x)=0$.   The linearization $f'(x)$ is surjective since $f$ is in general position. Now we  assume  that $d(x)\geq 1$.  By assumption, the kernel of $f'(x)$ is transversal to $T_x^\partial X$ implying
$\dim \ker f'(x)\geq \ \codim T_x^\partial X$ and hence,
$$\dim\ \ker f'(x)\geq d(x).
$$
Going into local coordinates,  we may assume  that the section  $f$
is already  filled,  that $x=0$ and that $0\in C=[0,\infty)^d\oplus {\mathbb R}^{n-d}\oplus
W$ where  $d=d(x)$. The linearisation
$$
f'(0):{\mathbb R}^n\oplus W\rightarrow {\mathbb R}^N\oplus W
$$
is surjective and its index satisfies
$i(f,0)=\dim \ker f'(0)=n-N$ in view of Proposition \ref{prop5.9}. Further,  $\ker f'(0)$ is transversal to $\{0\}^d\oplus {\mathbb R}^{n-d}\oplus W$. This implies that
$\ker f'(0)$ has an $\ssc$-complement $N^{\perp}$ contained in $C$. Therefore,
$\ker f'(0)$ is neat  with respect to $C$ and hence, by Proposition \ref{neat},   in good position to the partial quadrant $C$.
\end{proof}

\subsection{Perturbations}\label{sec6.5}

Using the results from the previous section we shall  prove some useful perturbation results.  In order to have an
sc-smooth partition of unity on the M-polyfold $X$ guaranteed,  we shall assume for simplicity 
throughout this section \ref{sec6.5} that the {\bf local models of $X$ are splicing
cores built on separable  sc-Hilbert spaces.} For the sc-smooth partitions of unity on M-polyfolds we refer to section 4.3 in \cite{HWZ7}.

We begin with a useful  lemma.
\begin{lem}\label{lem7.9.4}
Consider the strong M-polyfold bundle $p:Y\rightarrow X$ equipped with an
auxiliary norm $N:Y_{0,1}\to [0,\infty)$.  Assume that the local models of $X$ are built on separable sc-Hilbert spaces. Let   $x_0$ be a smooth point of $X$, $h_0$ a smooth point belonging to the fiber $Y_{x_0}$,  and  $U\subset X$ an open neighborhood of $x_0$. Then  there exists an $\ssc^+$-section $s$ of the strong
bundle $p$  satisfying
$$
s(x_0)=h_0
$$
and having its support in $U$. Moreover,  if $N(h_0)<\varepsilon$ we can
choose  the section $s$ in such a way that
$$
N(s(y))<\varepsilon.
$$
for all $y\in X$.
\end{lem}
\begin{proof}
The result is local.  Hence we take  a strong M-polyfold bundle chart 
$\Phi:p^{-1}(U)\to K^{\mathcal R}$ covering the sc-diffeomorphism $\varphi:U\to O$ as defined in Definition 4.8  of \cite{HWZ2}.
The set $O$ is an open subset of 
of the splicing core $K^{\mathcal S}=\{(v, e)\in V\oplus E\vert \pi_v (e)=e\}$ associated with the  splicing ${\mathcal S}=(\pi, E, V)$ and  $K^{\mathcal R}=\{((v, e), u)\in O\oplus F\vert \, \rho_{(v, e)}(u)=u\}$ is the splicing core associated with the strong bundle splicing ${\mathcal R}=(\rho, F, (O, {\mathcal S}))$. 
 In these coordinates the smooth  point  $(v_0, e_0)=\varphi (x_0)\in O$  corresponds to  $x_0$. Moreover,   $\Phi (x_0,h_0)=((v_0, e_0), h_0')\in K^{\mathcal R}$ where $h_0'$ is a smooth point in $F$, that is, $h_0'\in F_{\infty}$.  For points $(v, e)\in O$ close to $(v_0, e_0)$ we define the local section $s:K^{\mathcal R}\to O$ by 
$$s(v, e)=((v, e), \rho_{(v, e)}(h_0'))\in K^{\mathcal R}.$$
At the point $(v_0, e_0)$ we then have $\rho_{(v_0, e_0)}(h_0')=h_0'$ and hence $s(v_0, e_0)=((v_0, e_0), h_0')$ as desired. In view of the definition of a strong bundle splicing,  $\rho_{(v, e)}(u)\in F_{m+1}$ if $(v, e)\in O_m\oplus F_{m}$ and $u\in F_{m+1}$. Moreover,  by Definition 4.2 in \cite{HWZ2}, 
the triple ${\mathcal R}^1=(\rho,  F^1, (O, {\mathcal S}))$ is also a general sc-splicing which 
together with the fact that $h_0'$ is a smooth point in $F$ implies that the section $s$ of the bundle $K^{{\mathcal R}^1}\to O$ is sc-smooth. Consequently, $s$ is an $\ssc^+$-section  of the local strong M-polyfold bundle $K^{\mathcal R}\to O$.  We transport this section by means of the map $\Phi$ to obtain a local  $\ssc^+$ -section $s$ of  the bundle $p:Y\to X$. It satisfies $s(x_0)=h_0$. Using Lemma 5.7  from \cite{HWZ7} we find an sc-smooth  bump-function $\beta$ having its support in $U$ and  equal to $1$ near $x_0$ so that the section $\beta \cdot s$  has the desired properties.
\end{proof}

In the following we shall use the notation 
$${\mathcal M}^{f}=\{x\in X\vert \, f(x)=0\}$$
for the solution set of the section $f$ of the strong M-polyfold bundle $p:Y\to X$.

\begin{thm}\label{thmsec6-trans}
Assume that $f$ is a  Fredholm section
of the fillable strong M-polyfold bundle $p:Y\rightarrow X$ and
assume $\partial X=\emptyset$.  We assume that the sc-structure on $X$ is  is built on separable sc-Hilbert spaces. Let $N$ be an auxiliary norm. Assume
that $U$ is the open neighborhood of $f^{-1}(0)$ guaranteed by
Theorem \ref{thm6.12} having the property that for every section
$s\in\Gamma^+(p)$ whose support lies  in $U$ and which satisfies
$N(s(x))\leq 1$,  the section $f+s$ is  a proper Fredholm
section of the fillable strong M-polyfold bundle $p^1:Y^1\to X^1$. Denote the space of all such sections in $\Gamma^+(p)$ as
just described by ${\mathcal O}$. Then given  $t\in
\frac{1}{2}{\mathcal O}$ and $\varepsilon\in (0,\frac{1}{2})$, there
exists  a section $s\in {\mathcal O}$ satisfying $s-t\in
\varepsilon{\mathcal O}$ so that $f+s$ has at every solution $x$  of
$(f+s)(x)=0$ a linearization which is surjective. In particular, the
solution set
$$
{\mathcal M}^{f+s}=\{x\in X^1\vert \, (f+s)(x)=0\}
$$
is a smooth compact manifold without boundary in view of Theorem \ref{transthm}.
\end{thm}
\begin{proof}
If $t\in \frac{1}{2}{\mathcal O}$, then,  in view of Theorem \ref{thm6.12},  $f+t$ is a proper
Fredholm section of the bundle $p^1:Y^1\to
X^1$. The solution set $S=(f+t)^{-1}(0)$
is compact in $X_{\infty}$ in view of Theorem  \ref{infty-p} and at
every solution $x\in S$ the linearization $(f+t)'(x)$ is sc-Fredholm
by Proposition \ref{prop5.9}. 
Fixing $x_0\in S$, there is an sc-splitting  of the image space of the linearization $(f+t)'(x_0)$ into the range and a smooth  finite dimensional  cokernel in which we choose the basis $e_1, \ldots, e_m$. It consists of smooth vectors. By 
 Lemma \ref{lem7.9.4}, there exist finitely many sections $\wt{s}_1, \ldots ,\wt{s}_m\in \Gamma^+(p)$ having their supports in $U$ and satisfying 
 $\wt{s}_1(x_0)=e_1, \ldots ,\wt{s}_m(x_0)=e_m$. Therefore, the linear map
$$(\lambda_1, \ldots ,\lambda_m, h)\mapsto (f+t)'(x_0)h+\sum_{j=1}^m\lambda_je_j$$
is surjective. Now, we view $Y$ as a  fillable strong M-polyfold bundle over $\R^m\oplus X$. Then the mapping $\wt{F}:\R^m\oplus X^1\to Y^1$, defined by
$$\wt{F}(\lambda , x)=(f+t)(x)+\sum_{j=1}^m\lambda_j\wt{s}_j,$$
is a Fredholm section of the fillable bundle $Y^1\to \R^m\oplus X^1$ in view of 
Theorem \ref{prop5.21}. It has the property that  $\wt{F}(0, x_0)=0$ and that the linearization $\wt{F}'(0, x_0)$ is surjective. Using the good parametrizations  guaranteed by Theorem \ref{newtheoremA}, we find an open neighborhood $V((0, x_0))$ of the point $(0, x_0)\in \{0\}\oplus S$ on which the linearization $\wt{F}'(\lambda , x)$ is surjective at points $(\lambda , x)$ in the solution set $\{\wt{F}(\lambda , x)=0\}$.  Since $\{0\}\oplus S$ is compact, we can cover $\{0\}\oplus S$ with finitely many such open neighborhoods $V((0, x_i))$ for $i=1, \ldots ,N$.  Adding up  all the sections constructed in these open sets, we obtain the finitely many sections $s_1,\ldots ,s_k\in \Gamma^+(p)$ all having their supports in $U$ so that the Fredholm section  $F:\R^k\oplus X^1\to Y^1$ defined by 
$$F(\lambda ,x)=(f+t)(x)+\sum_{j=1}^k\lambda_j s_j (x)$$
has the following property. The linearization $F'(\lambda, x)$ is surjective at points $(\lambda , x)$ contained in an open neighborhood of $\{0\}\oplus S$ and solving the equation $F(\lambda , x)=0$. For $\delta >0$ sufficiently small we take the open ball $B_{\delta}(0)\subset \R^k$ centered at the origin and of radius $\delta$. Then the solution set 
$$S_{\delta}:=\{(\lambda , x)\in B_{\delta}(0)\oplus X^1\vert \, F(\lambda , x)=0\}$$
is a smooth finite dimensional manifold using the good parametrizations guaranteed by Theorem \ref{newtheoremA} and arguing as in the proof of Theorem \ref{transthm}.
We define the smooth map   $\beta:S_{\delta}\to \R^k$ as the composition $\beta=\pi\circ j$,
$$S_{\delta}\xrightarrow{j}\R^k\oplus X\xrightarrow{\pi} \R^k$$
where $j$ is the injection map and $\pi$ is the projection $\pi (\lambda , x)=\lambda$. By Sard's theorem, we find a regular value $\lambda^*\in \R^k$ of the map $\beta$ arbitrary close to $0$.
In view of the surjectivity of $F'(\lambda, x)$, the map $x\mapsto F(\lambda^*, x)$ has therefore the property  that the linearization $D_2F(\lambda^*, x)$ is surjective at all solutions $x\in X$ of the equation
 $F(\lambda^*, x)=0$. Moreover, the map $F(\lambda^*, \cdot )$ is a proper Fredholm section of the bundle $p^1:Y^1\to X^1$ in view of Theorem \ref{thm6.12}.
 Hence, by   Theorem \ref{transthm}, the solution set $\{x\in X^1\vert F(\lambda^*, x)=0\}$    is a smooth compact manifold without boundary and  the proof of Theorem \ref{thmsec6-trans} is complete.
 \end{proof}
 
Our main perturbation result concerning   Fredholm problems on 
M-polyfolds  possessing  boundaries  is as follows. 
\mbox{}
\begin{thm}\label{poil}
Let $p:Y\rightarrow X$ be a fillable  strong M-polyfold bundle
having nonempty boundary $\partial X$ and assume that $f:X\to Y$ is
a proper   Fredholm section. We assume that the sc-structure of $X$  is built on separable sc-Hilbert spaces. Fix  an auxiliary norm $N$  and assume that $U\subset X$ is  an open neighborhood of the compact solution set $\mcal M^f$ guaranteed by
Theorem \ref{thm6.12}.   Denote by ${\mathcal O}$ the set of sections 
$$\mcal O=\{s\in\Gamma^+(p)\ \vert
\ \text{$\supp (s)\subset U$ and $N(s(x))\leq 1$ for all $x\in
U$}\}.$$ Then, given $t\in \frac{1}{2}{\mathcal O}$  and $\varepsilon\in (0, \frac{1}{2})$, there
exists   a section $s\in {\mathcal
O}$ satisfying $s-t\in\varepsilon {\mathcal O}$ so that the solution
set
$$\mcal M^{f+s}=\{x\in X^1\vert \, (f+s)(x)=0\}$$
is in general position to
the boundary $\partial X$.  In particular,  $\mcal M^{f+s}$ is a smooth
compact manifold with boundary with corners in view of Theorem \ref{manwbc}.
\end{thm}
\mbox{}
\begin{proof}
Take $t\in \frac{1}{2}\mcal O$. Then the solution set $\mcal M^{f+t}$ is compact in $X$ as well as in $X_{\infty}$ by Theorem \ref{infty-p}.  If $x$ is a smooth point in $X$,  we denote by ${\mathcal F}^1, \ldots, {\mathcal F}^d$ the local faces of $X$ at the point $x$. Here $d=d(x)$.  Each local face ${\mathcal F}^j$ has the  tangent space
$T_x{\mathcal F}^j$ at the point $x\in X$.  Recall that  $T^\partial_xX$ is  the intersection
$
T^\partial_x X=\bigcap_{1\leq j\leq d} T_x{\mathcal  F}^j.
$
If $d(x)=0$, i.e., if $x$ is an interior point, we  set
$T^\partial_x X=T_x X$.
We shall proceed in  several  steps starting with the following lemma.
\begin{lem}\label{lemth5.19}
There exist finitely many
$\ssc^+$-sections $s_j\in \Gamma^+(p)$ for  $1\leq j\leq l$, supported in $U$ so that
the map $F:{\mathbb R}^l\oplus X^1\rightarrow Y^1$ defined by
$$F(\lambda,x):=(f+t)(x)+\sum_{j=1}^l \lambda_js_j(x)
$$
has the following properties.  There exists a small  $\varepsilon>0$ so that the solution set
$$\ov{S}_{\varepsilon}=\{(\lambda, x)\in \R^l\oplus X^1\vert \, \text{$F(x, \lambda )=0$ and $\abs{\lambda}\leq \varepsilon$}\}$$
is compact in $\R^l\oplus X_{\infty}$ and contained in $\ov{B}_{\varepsilon}\oplus U$. Moreover,  at every solution $(\lambda , x)\in \ov{S}_{\varepsilon}$, the section $F$ has the following properties.
\begin{itemize}
\item[(i)] the linearization $F'(\lambda, x)$ is surjective,
\item[(ii)] the kernel of $F'(\lambda,x)$ is transversal to the subspace
$ T_{(\lambda, x)}^\partial (\R^l\oplus X)$ of  the tangent space $T_{(\lambda , x)}({\mathbb
R}^l\oplus X)$,
\item[(iii)]  for every subset $\sigma \subset \{1,\ldots , d(x)\}$, where $d(x)$ is the order of the degeneracy of the point $x$, the lineraization  $F'(\lambda, x)$ restricted to the tangent space ${T_{(\lambda ,x)}(\R^l\oplus \bigcap_{j\in \sigma} \mathcal F}^j)$ is surjective and the kernel of this restriction is transversal to the subspace
$T^{\partial}_{(\lambda, x)}(\R^l\oplus \bigcap_{j\in \sigma} {\mathcal F}^j)$ in the tangent space $T_{(\lambda, x)}(\R^l\oplus \bigcap_{j\in \sigma} {\mathcal F}^j)$.
\end{itemize}
\end{lem}
Note that $T^\partial_{(\lambda ,x)}({\mathbb R}^l\oplus X)={\mathbb R}^l\oplus T_x^{\partial}
X$ and $T_{(\lambda ,x)}(\R^l\oplus \bigcap_{j\in \sigma}{ \mathcal F}^j)=\R^l\oplus T_x\bigcap_{j\in \sigma}{\mathcal F}^j$ where $\sigma\subset \{1, \ldots ,d(x)\}$. The
codimension of $T^\partial_{(\lambda ,x)}({\mathbb R}^l\oplus X)$ in the tangent space
$T_{(\lambda,x)}({\mathbb R}^l\oplus X)$  is
independent of $l$ and is equal to $d(x)$.

\begin{proof}[Proof of Lemma \ref{lemth5.19}] 
Choose a solution $x_0$ of $(f+t)(x_0)=0$.   The linearization $(f+t)'(x_0)$ is sc-Fredholm
by  Proposition \ref{prop5.9}.  Hence,  in view of Lemma \ref{lem7.9.4}, we find finitely many sections
$s_1, \ldots ,s_k\in\Gamma^+(p)$ satisfying $N(s_j(x))<1$ for all $x\in X$ and  having  their supports in $U$ so that the vectors $s_1(x_0), \ldots ,s_k(x_0)$ span the smooth  sc-complement of the range of  the linearization $(f+t)'(x_0)$. Taking additional
sections $s_{k+1}, \ldots ,s_m$ in $\Gamma^+(p)$ having their  supports in $U$,  we  arrange that also  the kernel of the linear map
$$
(\lambda, h)\mapsto (f+t)'(x_0)\cdot h +\sum_{j=1}^m \lambda_js_j(x_0)
$$
is transversal to ${\mathbb R}^m\oplus T_{x_0}^\partial X$ in the tangent space $\R^m\oplus T_{x_0}X$.
This is done as follows. We first observe that the kernel of the above linear map is equal to $\{0\}\oplus \ker (f+t)'(x_0)$ together with the set of points $(\lambda ,h)$ solving the equation $(f+t)'(x_0)h=-\sum_{j=k+1}^m\lambda_j s_j (x_0)$. 
Since $T^{\partial }_{x_0}X$ has finite codimension in the tangent space $T_{x_0}X$,  there 
are finitely many smooth vectors $h_{k+1},\ldots ,h_{m}$ such that 
$\text{span}\{h_{k+1},\ldots ,h_{m}\}\oplus T^{\partial }_{x_0}X=T_{x_0}X$, using Lemma 2.12 in \cite{HWZ2}. Consequently, choosing the additional sections  $s_j$ for $k+1\leq j\leq m$ such that $s_j (x_0)=-(f+t)'(x_0)h_j$, the transversality of the kernel follows. Moreover, multiplying the vectors $h_j$ by suitable small constants we may achieve, in view of
 Lemma \ref{lem7.9.4},  that $N(s_j (x))<1$ for all $x\in X$ and  $k+1\leq j\leq m$. 

Viewing  $Y$ as a  fillable strong $M$-polyfold bundle over $\R^m\oplus X$, the map  $F:\R^m\oplus X^1\to Y^1$, defined by
\begin{equation}\label{fredholmeq}
F(\lambda, x)=(f+t)(x)+\sum_{k=1}^m\lambda_j s_j (x)
\end{equation}
is a Fredholm section of the bundle $Y^1\to \R^m\oplus X^1$
having the property that $F(0, x_0)=0$, that  the linearization $F'(0, x_0)$ is surjective,  and  that the kernel of $F'(0, x_0)$ is transversal to $\R^m\oplus T^{\partial}_{x_0}X$ in $T_{(0, x_0)}(\R^m\oplus X)$.  Note that after
adding $r$  more sections the kernel of the new linearized operator is automatically transversal to $T_{(0,x_0)}^{\partial }({\mathbb R}^{m+r}\oplus X)$ in  the tangent space $T_{(0, x_0)}(\R^{m+r}\oplus X)$, moreover, the linearized operator is still surjective.\\

{\bf  Claim (a)}\,  There is  an open neighborhood $U(0, x_0)$ of the point  $(0, x_0)$ in $\R^m\oplus X$ such that at every solution $(\lambda, x)$  of $F(\lambda , x)=0$ contained  in $U(0, x_0)$ , the linearization $F'(\lambda, x)$ is surjective and the kernel of $F'(\lambda, x)$ is transversal to $T^{\partial }_{(\lambda , x)}(\R^m\oplus X)$ in  $T_{(\lambda , x)}(\R^m\oplus X)$.

In order to  prove this  claim we work in our familiar  local coordinates and assume that  the section $F$ is already filled
and has the normal form
$$F:\R^m\oplus  [0,\infty)^d\oplus {\mathbb R}^{n-d}\oplus W \rightarrow {\mathbb R}^N\oplus W.$$
 The integer $d$ is equal to the order $d=d(x_0)$ of $x_0$.  In these coordinates
the point $x_0$ corresponds to the point $0\in [0,\infty)^d\oplus {\mathbb R}^{n-d}\oplus W$ so that
$(0, x_0)$ corresponds to $(0, 0)\in \R^m\oplus  [0,\infty)^d\oplus {\mathbb R}^{n-d}\oplus W$.
In order to  simplify the notation  we abbreviate, in abuse of  the notation,  $X= [0,\infty)^d\oplus {\mathbb R}^{n-d}\oplus W$.  Then
\begin{align*}
T_{(0, 0)}^{\partial }(\R^m \oplus X)=\R^m\oplus T^{\partial }_0X=\R^m\oplus \{0\}^d\oplus \R^{n-d}\oplus W
\end{align*}
and
$$
T_{(0, 0)}(\R^m \oplus X)=\R^m\oplus \R^d\oplus \R^{n-d}\oplus W.
$$
By construction,  the kernel $N$ of the linearization $F'(0, 0)$ is transversal to  the tangent space
$T_{(0, 0)}^{\partial }(\R^m \oplus X)$. Therefore, there is an sc-complement $N^{\perp}$ of $N$ satisfying
$N^{\perp}\subset \R^m\oplus \{0\}^d\oplus \R^{n-d}\oplus W$.
In particular, the kernel $N$ is neat with respect to the partial cone $C=\R^m\oplus [0,\infty)^d\oplus {\mathbb R}^{n-d}\oplus W$ according to Definition \ref{defn5.6.1}.
In view of  Theorem \ref{LOCAX1}, there is a good parametrization of the solution set $\{F=0\}$ near  the solution $(0, 0)$. Accordingly there exist  an open neighborhood $Q$ of $(0, 0)$ in $N\cap C$ and an sc-smooth map
\begin{equation}\label{eqa}
A:Q\to N^{\perp}
\end{equation}
satisfying $A(0, 0)=0$ and $DA(0, 0)=0$ so that  mapping  $\Gamma (n)=n+A(n)$ parametrizes all the solutions of $F=0$  near  $(0, 0)$.  Moreover,  the linearization $F'(\Gamma (n))$ is surjective for all $n\in Q$.
If $\Gamma (n)=(\lambda , x)$ is such a solution, we write $x=(x_1, \ldots ,x_d, x')\in [0,\infty )^d\oplus (\R^{n-d}\oplus W)$.
Denote by  $\Sigma$ the subset of  indices $i\in \{0, \ldots ,d\}$ for which $x_i=0$ and introduce the subspace
$\R^{\Sigma}\subset \R^d$ consisting of points $z\in \R^d$ for which $z_i=0$ if $i\in \Sigma$. Then   $T^{\partial}_{(\lambda, x)}(\R^m\oplus X)=\R^m\oplus \R^{\Sigma}\oplus \R^{n-d}\oplus W$ and we have to show that the kernel $N'$ of the linearization $F'(\lambda, x)$ is transversal to $\R^m\oplus \R^{\Sigma}\oplus \R^{n-d}\oplus W$. By the properties of the good parametrization $\Gamma$ in Definition  \ref{defn5.6.7},  the kernel $N'$ has the representation  $N'=\{\delta n+DA(n)\delta n\vert \, \delta n\in N\}$.
If  $v\in \R^m\oplus \R^d\oplus \R^{n-d}\oplus W$, then  $v=u+u'$ for some  $u\in N$ and some
$u'\in \R^{m}\oplus \{0\}^d\oplus \R^{n-d}\oplus W$  because the kernel $N$ and the subspace  $\R^m\oplus \{0\}^d\oplus \R^q\oplus W$ are transversal. Now observe that
$$v=(u+DA(n)u)+(u'-DA(n)u)$$
where   $u+DA(n)u\in N'$. Denoting by
$p:\R^m\oplus \R^d\oplus \R^q\oplus W\to \R^d$ the natural  projection, we have
$p(u'-DA(n)u)=0$  because $p(u')=0$ and
$DA(n)u\in N^{\perp}\subset R^m\oplus  \{0\}^d\oplus \R^q\oplus W$ so that
$p(DA(n)u)=0$. In particular, the kernel $N'$ is transversal to
$T^{\partial}_{(\lambda, x)}(\R^m\oplus X)$ as claimed in (a). \\

{\bf  Claim (b)}\,  For every solution $(\lambda, x)\in U(0, x_0)$
 and every subset $\sigma$ of $\{1,\ldots ,d(x)\}$,
 the linearization $F'(\lambda, x)$ restricted to
 $T_{(\lambda, x)}(\R^m\oplus \bigcap_{j\in \sigma}{\mathcal F}^j)$
 is surjective and the kernel of of this map is transversal to
 $T^{\partial}_{(\lambda ,x)}(\R^m\oplus  \bigcap_{j\in \sigma}{\mathcal F}^j)$ in
$T_{(\lambda ,x)}(\R^m\oplus  \bigcap_{j\in \sigma}{\mathcal F}^j)$.

In order to prove the claim (b) we observe that the point $(\lambda, x)$
belongs to $d(x)$- many faces $\R^m\oplus {\mathcal F}^j$.
 The faces $\R^m\oplus {\mathcal F}^j$ are in the coordinates represented  as follows. The set $\Sigma=\{i \in \{1, \ldots , d\}\vert \, x_i=0\}$ has exactly $d(x)$ elements. For $j\in \Sigma$,  the face ${\mathcal F}^j$ is the set
$$
{\mathcal F}^j=[0,\infty )^{(j)}\oplus \R^{n-d}\oplus W
$$
where $[0,\infty )^{(j)}=\{(z=(z_1, \ldots, z_d)\in [0,\infty )^d\ \vert \, z_j=0\}$. The tangent space to $\R^m\oplus {\mathcal F}^j$ at the solution  $(\lambda, x)$ is equal to
$$
T_{(\lambda, x)}(\R^m\oplus {\mathcal F}^j)=\R^m\oplus \R^{(j)}\oplus \R^{n-d}\oplus W
$$
where  the subspace $\R^{(j)}$  consists of those points $z\in \R^d$ whose $j$th coordinate is equal to $0$. Moreover, $T^{\partial}_{(\lambda ,x)}(\R^m\oplus {\mathcal F}^j)$ is equal to the subspace
$$
T^{\partial }_{(\lambda ,x)}(\R^m\oplus {\mathcal F}^j)=\R^m\oplus \R^{\Sigma}\oplus
\R^{n-d}\oplus W.
$$
We shall prove that the linearization of the map $F\vert{\R^m\oplus \mathcal F}^j:\R^m\oplus {\mathcal F}^j\to \R^N\oplus W$ is surjective at the solution $(\lambda, x)$ of $F( \lambda, x)=0$.  In order to do so, we choose  $y\in \R^N\oplus W$.
We already know that the linearization $F'(\lambda, x):\R^{m}\oplus \R^d\oplus \R^{n-d}\oplus W\to \R^N\oplus W$ is surjective and so there exists  $(\delta \lambda,\delta  u)\in \R^{m}\oplus (\R^d\oplus \R^{n-d}\oplus W) \to \R^N\oplus W$  solving $F'(\lambda, x)(\delta \lambda, \delta u)=y$.  We also know that the kernel $N'$ of $F'(\lambda, x)$ is transversal to
$\R^m\oplus \{0\}^d\oplus \R^{n-d}\oplus W$ in $\R^m\oplus \R^d\oplus \R^{n-d}\oplus W$.
Therefore, we have the representation
$(\delta \lambda, \delta u)=(\delta\lambda_1, \delta u_1)+(\delta\lambda_2, \delta u_2)$ where $(\delta\lambda_1, \delta u_1)\in N'$ and $(\delta\lambda_2, \delta u_2)\in \R^m\oplus \{0\}^d\oplus \R^{n-d}\oplus W$. Consequently, $(\delta\lambda_2, \delta u_2)$ belongs to $T_{(\lambda, x)}(\R^m\oplus {\mathcal F}^j)$ and satisfies $F'(\lambda, x)(\delta \lambda_2, \delta x_2)=y$ as we wanted to prove.

Next we shall show that the kernel of the linearization $F'(\lambda, x)\vert T_{(\lambda , x)}(\R^m\oplus {\mathcal F}^j)$ is transversal to the subspace  $T^{\partial }_{(\lambda ,x)}(\R^m\oplus {\mathcal F}^j)$ in $T_{(\lambda ,x)}(\R^m\oplus {\mathcal F}^j)$.
To do so we choose $(\delta \lambda ,\delta u)\in T_{(\lambda ,x)}(\R^m\oplus {\mathcal F}^j)=\R^m\oplus \R^{(j)}\oplus \R^{n-d}\oplus W$ and have the representation
$(\delta \lambda ,\delta u)=(\delta \lambda_1, \delta u_1)+(\delta \lambda_2, \delta u_2)$ where $(\delta \lambda_1, \delta u_1)\in N'$ and $(\delta \lambda_2, \delta u_2)\in \R^m\oplus \{0\}^d\oplus \R^{n-d}\oplus W\subset T^{\partial }_{(\lambda, x)}(\R^m\oplus {\mathcal F}^j)$. Moreover, denoting by
$$p:\R^m\oplus \R^d\oplus \R^{n-d}\oplus W\to \R^d$$
the projection, we conclude that
$p(\delta \lambda_1, \delta u_1)=p(\delta \lambda, \delta u)=a\in \R^d$ where the $j$th coordinate of $a$ is equal to $0$. Hence $(\delta \lambda_1, \delta u_1)\in T_{(\lambda, x)}(\R^m\oplus {\mathcal F}^j)$ as we wanted to show.  The same way one verifies the more general claim in (b).\\

 {\bf (c)}\, To finish the proof of Lemma \ref{lemth5.19}  we  carry out the above construction near every point $(0, x)$ for $x\in \mcal M^{f+t}$. Since the solution set $\mcal M^{f+t}$ is compact, we can select finitely many neighborhoods $U(0, x_j)$, $1\leq j\leq m$,  covering $\{0\}\oplus \mcal M^{f+t}$.
Now taking as a perturbation  the sum of all  the finitely many  sections  constructed in these neighborhoods,  we obtain the section
$$
F(\lambda,x)= (f+t)(x)+\sum_{j=1}^l \lambda_js_j(x)
$$
 of the bundle $Y^1\rightarrow {\mathbb R}^l\oplus X^1$ which is a Fredholm section in view of Theorem \ref{prop5.21}.  We claim that there  exists a sufficiently small $\varepsilon>0$  such that 
 $$
 \ov{S}_{\varepsilon}=\{(\lambda, x)\in \R^l\oplus X^1\vert
 \text{$ F( \lambda, x )=0$ and $\abs{\lambda}\leq \varepsilon$}\}\subset \bigcup_{j=1}^n U(0, x_j).
 $$

 Indeed, otherwise there exists   a sequence $(\lambda^k, x^k)\in \R^l\oplus X^1$ satisfying 
 $F(\lambda^k, x^k)=0$ and $\abs{\lambda^k}\leq \frac{1}{k}$ but 
 $(\lambda^k , x^k)\not \in \bigcup_{j=1}^n U(0, x)$. Since the supports of $t$ and $s_j$ are contained in $U$, we conclude that $x^k\in U$ for all $k$.  From $f(x^k)=-t(x^k)-\sum_{j=1}^l\lambda^n_js_j (x^k)$ we find  $N(f(x^k))\leq 1$  for $k$ large and since the set $U$ has  the property (ii) of Theorem \ref{thm6.12}, we may assume that the sequence $(x^k)$ converges to the point $x\in \ov{U}$. Because $(f+t)(x)=0$,  we have 
 $(0, x)\in \{0\}\oplus M^{f+t}\subset \bigcup_{j=1}^m U(0, x_j)$. 
 This contradicts our assumption $(\lambda^k, x^k)\not \in \bigcup_{j=1}^m U(0, x_j)$ for all $k$. Consequently, the section $F$ has the properties (i) and (ii) of Lemma \ref{lemth5.19}  if $\varepsilon>0$ is sufficiently small.

It remains to show that the solution set  
 $ \ov{S}_{\varepsilon}$ is compact in $\R^l\oplus U$ if $\varepsilon >0$ is sufficiently small.  To see this we  take the solution $(\lambda, x)\in \ov{S}_{\varepsilon}$, then $f(x)=-t(x)-\sum_{j=1}^l\lambda_js_j (x)$. Since $t\in \frac{1}{2}{\mathcal O}$ and $N(s_j (x))\leq 1$ for $1\leq j\leq l$ and all $x\in X$, we conclude 
 from $\abs{\lambda}\leq \varepsilon$ that $N(f(x))\leq 1$ if  $\varepsilon$ is sufficiently small.  In addition, since ${\mathcal M}^f\subset U$ and since the supports of $t$ and $s_j$ are contained in $U$,  it follows that $x\in U$.  Hence if $(\lambda^k, x^k)\in \ov{S}_{\varepsilon}$,  then $x^k\in U$  so that by Theorem \ref{thm6.12},  the sequence $(x^k)$ possesses a convergent subsequence. Since $\abs{\lambda^k}\leq \varepsilon$, we conclude that the sequence $(\lambda^k, x^k)$ has a  subsequence converging  to a  solution $(\lambda , x )$ of $F(\lambda, x )=0$. The solution belongs to $\ov{B}_{\varepsilon}\oplus U$.  Indeed, if $x\in  X\setminus U$, then $t(x)+\sum_{j=1}^l\lambda_j s_j (x)=0$ since the supports of $t$ and $s_j$ are contained in $U$. Hence $f(x)=0$ which implies $x\in f^{-1}(0)\subset U$, contradicting the assumption $x\in X\setminus U$. Consequently, the solution set $\ov{S}_{\varepsilon}$ is compact in $\R^l\oplus X$ and $\ov{S}_{\varepsilon}\subset  \ov{B}_{\varepsilon}\oplus U$. The compactness in $\R^l\oplus X_{\infty}$ follows arguing as in Theorem \ref{infty-p}. The proof of Lemma \ref{lemth5.19} is complete.
 \end{proof}

Abbreviate by $S_\varepsilon$ the set  consisting of  those points $(\lambda , x)$ in $\ov{S}_\varepsilon$ for which  $\abs{\lambda}<\varepsilon$.
\begin{lem}\label{lemthm5.19.2}
The  Fredholm section $F$ of  the fillable
strong bundle $Y^1\rightarrow B_\varepsilon \oplus X^1$ introduced in  Lemma \ref{lemth5.19} is in good
position to the corner structure of $B_\varepsilon\oplus X$.
Consequently, $F^{-1}(0)$ is a smooth manifold with boundary with
corners.
\end{lem}
\begin{proof}
By Lemma \ref{lemth5.19}, the  linearization $F'(\lambda , x)$ is surjective and its  kernel is transversal to
$T^{\partial }_{\lambda, x}(\R^l\oplus X)$ in $T_{(\lambda, x)}( \R^l\oplus X)$ at every point $(\lambda ,x)\in S_{\varepsilon}$. Arguing as in the proof of Lemma \ref{lemth5.19}, the kernel of
$F'(\lambda,x)$ is neat in  $T_{(\lambda,x)}({\mathbb R}^l\oplus X)$ at every point $(\lambda,x)\in S_\varepsilon$.  This implies that the kernel of
$F'(\lambda,x)$ is in good position  with respect to the corner structure of $\R^l\oplus X$ at $(\lambda, x)$. Consequently,  arguing as in the proof Theorem \ref{thm6.13},  the solution set
$S_\varepsilon=F^{-1}(0)$
carries a structure of  a smooth manifold with boundary with corners.
\end{proof}

Continuing with the proof of Theorem \ref{poil} we know from  the proof of 
Lemma \ref{lemth5.19} that the solution set $S_{\varepsilon}$ satisfies 
 $$
 S_{\varepsilon}=\{(\lambda, x)\in \R^l\oplus X^1\vert
 \text{$ F(x, \lambda )=0$ and $\abs{\lambda}< \varepsilon$}\}\subset
  \bigcup_{j=1}^n U(0, x_j)
$$
where $F$ is the  Fredholm section introduced in Lemma  \ref{lemth5.19}.

Every  point $(0, x_j)$, belongs to  $d(x_j)$-many  faces $\R^l\oplus {\mathcal F}^k_{x_j}$ for $1\leq k\leq d(x_j)$. If  $1\leq j\leq n$ and if  $\sigma \subset \{1, \ldots ,d(x_j)\}$,  we  define
$$S^{j, \sigma}_{\varepsilon}:=S_{\varepsilon}\cap U(0, x_j)\cap \bigl(\R^l\oplus \bigcap_{k\in \sigma}{\mathcal F}^k_{ x_j}\bigr).$$
In view of Lemma \ref {lemth5.19} and arguing as in Theorem \ref{thm6.13}, the set, $S^{j, \sigma}_{\varepsilon}$ is a manifold having a  boundary with corners.  We introduce the projections
$\beta:S_{\varepsilon}\to \R^l$ and $\beta_{j, \sigma}:S^{j, \sigma}_{\varepsilon}\to \R^l$
as the compositions
\begin{align*}
&\beta:S_{\varepsilon}\xrightarrow{j}\R^l\oplus X\xrightarrow{\pi}\R^l\\
&\beta_{j,\sigma}:S^{j,\sigma}_{\varepsilon}\xrightarrow{j}\R^l\oplus X\xrightarrow{\pi}\R^l
\end{align*}
where $j$ is the injection mapping and $\pi$ is the projection $\pi (\lambda ,x)=\lambda$.

In view of  Sard's theorem, the set of regular values of projections $\beta$ and $\beta_{j, \sigma}$  has full measure in $B_{\varepsilon}$.  Denote by $\lambda^*$  a common regular value  of  all  these  projections and introduce the $\ssc^+$-section
$$s^*=t+\sum_{j=1}^l \lambda_j^*s_j
$$
so that
$$(f+s^*)(x)=F(\lambda^*, x).$$
We claim that at every solution $x$ of $(f+s^*)(x)=0$, the linearization $(f+s^*)'(x)$ is surjective and that the kernel of $(f+s^*)'(x)$ is transversal to $T^{\partial }_xX$ in $T_xX$.

In order to prove the claim we take  a solution $x$ of $(f+s^*)(x)=0$ and first prove that $(f+s^*)'(x)$ is surjective. We have $F(\lambda^*, x)=0$ so that $(\lambda^*, x)\in S_{\varepsilon}$. By Lemma \ref{lemth5.19}, the linearization $F'(\lambda^*, x)$ is surjective. Hence given $y\in Y^1$, there exists $(\delta \lambda, \delta x)\in T_{(\lambda, x)}(\R^l\oplus X)$ solving $F'(\lambda^*, x)(\delta \lambda , \delta x)=y$.  The point $\lambda^*$ is a   regular value of the projection $\beta:S_{\varepsilon}\to \R^l$ and the tangent space to $S_{\varepsilon}$ at $(\lambda^*, x)$ coincides with the kernel $N$ of $F'(\lambda^*, x)$. Hence there exists $(\delta \lambda_1, \delta x_1)\in N$ solving the equation $d\beta(\lambda^*, x)(\delta \lambda_1, \delta x_1)=\delta \lambda$ and satisfying
$$F'(\lambda^*, x)(\delta \lambda_1, \delta x_1)=0.$$
From $d\beta(\lambda^*, x)(\delta \lambda_1, \delta x_1)=\delta \lambda_1$  we conclude
$\delta \lambda_1=\delta \lambda$  so that
$$F'(\lambda^*, x)(\delta \lambda, \delta x_1)=0.$$  Consequently,
\begin{equation*}
\begin{split}
y&=F'(\lambda^*, x)(0, \delta x-\delta x_1)\\
&=D_2F(\lambda^*, x)(\delta x-\delta x_1)\\
&=(f+s^*)' (\delta x-\delta x_1)
\end{split}
\end{equation*}
showing that  the linearization $(f+s^*)'(x)$ is indeed surjective.

Next we show that the kernel of the linearization $(f+s^*)'(x)$ is transversal to $T_x^{\partial }X$ at every solution $x$ of $(f+s^*)(x)=0$.  Since $F(\lambda^*, x)=0$ and $F^{-1}(0)\subset \bigcup_{1\leq j\leq n}U(0, x_j)$, the point $(\lambda^*, x)$ belongs to some open neighborhood $U(0, x_j)$.  To prove the transversality we work in  local coordinates as in the proof of Lemma \ref{lemth5.19}. We assume that $F$ is already
filled and is  the  form
$$F:\R^l\oplus [0,\infty )^d\oplus \R^{n-d}\oplus W\to \R^N\oplus W.$$
The integer $d$ is equal to the order $d=d(x_j)$ of $x_j$. In these coordinates the point $x_j$ corresponds to the point $0\in [0, \infty )^d\oplus \R^{n-d}\oplus W$ so that $(\lambda^*, x_j)$ corresponds to $(\lambda^*, 0)\in \R^l\oplus [0,\infty )^d\oplus \R^{n-d}\oplus W$.
The point  corresponding to $(\lambda^*, x)$ we shall  denote  again by $(\lambda^*, x)$. Its degree $d(x)$ satisfies $d (x)\leq d$.  We use the  notation $X=[0,\infty )^d\oplus \R^{n-d}\oplus W$.  If  $\Sigma=\{i\in \{1, \ldots , d\}\vert \, x_i=0\}$, the tangent space $T_{(\lambda^*, x)}(\R^l\oplus \bigcap_{j\in \Sigma}{\mathcal F}^j))$ is equal to
$$T_{(\lambda^*, x)}(\R^l\oplus \bigcap_{j\in \Sigma}{\mathcal F}^j)=\R^l\oplus \R^{\Sigma}\oplus \R^{n-d}\oplus W.$$
Here ${\mathcal F}^j$  for $j\in \Sigma$ are the faces containing the point $x$.
By Lemma \ref{lemth5.19}, the kernel $N$ of the linearization $F'(\lambda^*, x)$ is transversal to $T^{\partial }_{(\lambda^*, x)}(\R^l\oplus X)$. Hence,
given  $\delta x\in T_xX$, there exists  $(\delta \lambda_1, \delta x_1)$ belonging to the kernel $N$ of $F'(\lambda^*, x)$ and $(\delta \lambda_2, \delta x_2)$ in $T^{\partial }_{(\lambda^*, x)}(\R^l\oplus X)$ such that
$$(\delta \lambda_1, \delta x_1)+(\delta \lambda_2, \delta x_2)=(0, \delta x).$$
The point  $\lambda^*$ is a  regular value of the projection
$\beta_{j, \Sigma}:S^{j, \Sigma}_{\varepsilon}\to \R^l$. Since $d\beta_{j, \Sigma} (\lambda^*, x)$ is defined on
$N\cap T_{(\lambda^*, x)}(\R^l\oplus \bigcap_{j\in \Sigma}{\mathcal F}^j)=
N\cap (\R^l\oplus \R^{\Sigma}\oplus \R^{n-d}\oplus W)$ we find $(\delta \lambda ', \delta x')$ belonging to $N\cap (\R^l\oplus \R^{\Sigma}\oplus \R^{n-d}\oplus W)$ such that
$d\beta_{j, \Sigma}(\lambda^*, x)(\delta \lambda', \delta x')=\delta \lambda_1$. Using $d\beta_{j, \Sigma} (\lambda^*, x)(\delta \lambda', \delta x')=\delta \lambda_1$, we obtain
$$F'(\lambda^*, x)(0, \delta x_1-\delta x')=D_2F(\lambda^*, x)(\delta x_1-\delta x')=0$$
and
$$(\delta x_1-\delta x')+(\delta x_2+\delta x')=\delta x.$$
Since $\delta x_2+\delta x'$ belongs to $\R^l\oplus \R^{\Sigma}\oplus \R^{n-d}\oplus W$,
the kernel of the linearization $(f+s^*)'(x)$ is indeed transversal to $T_x^{\partial }X$ in $T_xX$ as we wanted to show. 

In order to finish  the proof of Theorem \ref{poil} we observe that so far we have proved that the Fredholm section $f+s^*$ is in general position to the boundary $\partial X$ according to Definition \ref{generalpos}.  Consequently,  in view of 
Theorem \ref{manwbc}, the solution set ${\mathcal
M}^{f+s^\ast}$ is a compact manifold with boundary with corners.
This completes the proof of Theorem \ref{poil}.
\end{proof}

The following result is proved along the lines of the previous result. In contrast to Theorem \ref{poil} we impose conditions on the Fredholm sections at those solutions  which are located at the boundary $\partial X$, while the perturbation has its support away from the boundary.

\begin{thm}\label{newgposition}
Let $p:Y\rightarrow X$ be a fillable strong M-polyfold bundle
having a nonempty boundary $\partial X$ and let  $f$ be a proper Fredholm section of $p$. We assume that the sc-structure on $X$ is built on separable sc-Hilbert spaces and assume that $U$ is an open neighborhood of the solution set ${\mathcal M}^f=f^{-1}(0).$ Moreover, let $N$ be an auxiliary norm guaranteed by Theorem \ref{thm6.12}. If at every point 
$x\in\partial X$ solving  $f(x)=0$ the linearization of $f$ is surjective and the kernel of the linearization  is in good position to the corner structure of $\partial X$, then there exists an arbitrarily small  $\ssc^+$-section
$s$ which has its support in $U$ and  which vanishes near $\partial X$ so that the Fredholm section 
$f+s$ is in good position as defined in Definition \ref{gposition}. In particular,  the solution set ${\mathcal M}^{f+s}=(f+s)^{-1}(0)$ is a smooth compact manifold with boundary with corners.
\end{thm}

The proof goes as follows. By the assumption on the solutions of $f(x)=0$ at the boundary  $\partial X$, 
the solution set of $f=0$ admits a good parametrization near the boundary points. Then one finds enough $\ssc^+$-sections which vanish near the boundary and have their supports  in $U$ so that  the Fredholm section $F(\lambda,x)=f(x)+\sum_{i=1}^k \lambda_i s_i(x)$
has,  at every solution $(0,x)$ of $F(0,x)=0$ with  $x\in U\setminus \partial X$,  a linearization which is surjective.  From the assumption of the theorem it follows that the linearization $F'(0, x)$ is automatically surjective also at the boundary points $x\in \partial X$ satisfying 
$f(x)=0$ and hence  $F(0, x)=0$, in addition, its kernel is again, by assumption, in good position to the corner structure of  $\R^k\oplus X$. Therefore, as in the proof of 
Theorem \ref{poil},  one  finds a suitable $\varepsilon$ so that the solutions space $S_\varepsilon=\{(\lambda,x)\vert \, F(\lambda,x)=0, \abs{\lambda}<\varepsilon\}$ is a smooth manifold with
boundary with corners. A small regular value $\lambda^\ast$ of the projection
$$
(\lambda,x)\mapsto  \lambda
$$
defines the  section $s^\ast=\sum_{i=1}^k \lambda_i^\ast s_i$  for which  $f+s^\ast$ has the desired properties. This completes the proof of Theorem \ref{newgposition}.

\subsection{Some Invariants for Fredholm Sections}\label{someinvariants}
The following discussion  extends  some standard material from the
classical nonlinear Fredholm theory to the M-polyfold context.
We  introduce the notion of an $\ssc$-differential
$k$-form starting with comments about our  notation.
For us the tangent bundle of $X$ is $TX\rightarrow X^1$,  that is, it is  only defined for the base points in $X^1$. An   {\bf sc-vector field} on $X$  is an  sc-smooth section $A$ of  the tangent bundle $TX\rightarrow X^1$ and hence it is  defined on $X^1$. Similarly, an  sc-differential form  on $X$ which we will define next,  is  only defined over the base points in $X^1$. The definition of a vector field and the following  definition of an sc-differential form
are justified since the construction of $TX$, though only defined over $X^1$, requires the knowledge of $X$.  If $X$ is a M-polyfold, then  $\oplus_k TX_\infty$ denotes  the $k$-fold Whitney sum of $k$ copies of the tangent space $TX$ .  

\begin{defn}
An {\bf $\boldsymbol{\ssc}$-differential ${\boldsymbol{k}}$-form on the $M$-polyfold $X$}  is
an sc-smooth map $\omega:\oplus_k TX \to \R$ which is
linear in each argument separately, and skew symmetric. \end{defn}

If $\omega$ is an  sc-differential form on $X$,  we may also  view it  as an  sc-differential form on $X^i$.  Denote by $\Omega^\ast(X^i)$ the graded commutative algebra of sc-differential forms on $X^i$. Then we have the inclusion map
$$
\Omega^\ast(X^i)\rightarrow\Omega^\ast(X^{i+1}).
$$
which  is injective since $X_{i+1}$ is dense in $X_i$ and the forms are sc-smooth. Hence we have a directed system whose  direct limit is denoted by 
$\Omega^\ast_\infty(X)$. An element $\omega$ of  degree $k$ in $\Omega^\ast_\infty(X)$ is a skew-symmetric map $\oplus_k(TX)_\infty\rightarrow {\mathbb R}$ such that it has an  sc-smooth extension to an  sc-smooth k-form
$\oplus_kTX^i\rightarrow {\mathbb R}$ for some $i$. We shall refer to an element of $\Omega^k_\infty(X)$ as an  sc-smooth differential form on $X_\infty$. We note, however, that it is part of the structure that the $k$-form is defined and sc-smooth on some $X^i$.

Next we associate with an sc-differential $k$-form $\omega$ its exterior differential
$d\omega$ which is a $(k+1)$-form on the M-polyfold  $X$. 
Let $A_0, \ldots , A_k$ be $k+1$ many sc-smooth vector fields on $X$. We define
$d\omega$ on $X^1$, using the familiar  formula,   by
\begin{equation*}
\begin{split}
d\omega (A_0, \ldots ,A_k)&=\sum_{i=0}^k(-1)^iD(\omega (A_0, \ldots , \wh{A}_i, \ldots , A_k))\cdot A_i\\
&+\sum_{i<j}(-1)^{(i+j)}\omega ([A_i, A_j], A_0, \ldots , \wh{A}_i,\ldots ,\wh{A}_j, \ldots ,A_k).
\end{split}
\end{equation*}

The right-hand side of the formula above only makes sense at the base points $x\in X_2$. This  explains why  $d\omega$ is a $(k+1)$-form on $X^1$.
By the previous discussion the differential $d$ defines a map
$$
d:\Omega^k(X^i)\rightarrow \Omega^{k+1}(X^{i+1})
$$
and consequently induces a map
$$
d:\Omega^\ast_\infty(X)\rightarrow \Omega^{\ast+1}_\infty(X)
$$
having the usual property $d^2=0$. Then $(\Omega^\ast_\infty(X),d)$ is a  graded differential algebra which we shall call the de Rham complex.

If $\varphi:M\to X$ is an sc-smooth map from a finite-dimensional manifold $M$ into an M-polyfold $X$, then it induces an  algebra homomorphism
 $$
 \varphi^\ast:\Omega^\ast_\infty(X)\rightarrow \Omega^\ast(M)_\infty
 $$
 satisfying
$$d (\varphi^*\omega )=\varphi^*d\omega.$$

Since $d^2=0$, we can define as usual   the deRham cohomology groups
$$H^\ast_{dR}(X,{\mathbb R})=\oplus_{i=0}^{\infty}H^i(X,{\mathbb
R}).$$
 A differential form $\omega$ is in the following a finite formal sum of
forms of possibly different degrees $\omega=\omega_0+\ldots +\omega_n$.

Next we consider a fillable strong M-polyfold bundle
 $p:Y\rightarrow X$ in which the local models of $X$ are built on separable sc-Hilbert spaces. We assume that $\partial X=\emptyset$. Let $f:X\to Y$ be  a  proper 
 Fredholm section  of the bundle $p$. Let  $U$ be the open neighborhood of the solution set $f^{-1}(0)$ and ${\mathcal O}\subset \Gamma^+(p)$ the space of sections as in
Theorem \ref{thmsec6-trans}. In view of this theorem, there exist small
 sections $s\in {\mathcal O}$ having their supports in $U$  and having the property that $f+s$ are 
 proper Fredholm sections of the bundle $p^1:Y^1\to X^1$. In addition, the linearizations of $f+s$ at the solution set
${\mathcal M}^{f+s}=\{x\in X\vert \, (f+s)(x)=0\}$ are surjective. 
\begin{defn}
A perturbation $s\in {\mathcal O}$  is called {\bf generic}, if the Fredholm section $f+s$ is in general position, i.e. the linearization is surjective at every
zero of $f + s$  and, in addition, the kernel of the lineraization is in
good position to the corner structure at the zeros belonging to the
boundary.
\end{defn}
Finally, we assume that the given section $f$ is  orientable and let
 $\mathfrak{o}$ be  an orientation of $f$ as defined in  Appendix \ref{det}. Then  a map
$$
\Phi_{(f,\mathfrak{o})}:H^\ast_{dR}(X,{\mathbb R})\rightarrow
{\mathbb R}
$$
can be defined as  follows. We take  a generic $s\in {\mathcal O}$. Then the associated solution set ${\mathcal M}^{f+s}$ is contained in $X_{\infty}$ and, in view of Theorem \ref{transthm}, is a compact smooth manifold without boundary which, in addition, is oriented since  $(f,\mathfrak{o})$ is an oriented Fredholm section. Hence we can
integrate a differential  form $\omega=\omega_0+\omega_1+\cdots +\omega_n$ over
${\mathcal M}^{f+s}$ where we put the integral  equal to zero on each component of
${\mathcal M}^{f+s}$ if the degrees don't match the local dimensions. This way one obtains the real number
$$
\int_{{\mathcal M}^{f+s}} \omega:=\sum_{i=1}^n \int_{{\mathcal M}^{f+s}}j^*\omega_i
$$
 with the sc-smooth inclusion  mapping $j:{\mathcal M}^{f+s}\to X_{\infty}$.
 
If $s'\in {\mathcal O}$ is a second generic perturbations of $f$, we
find  a generic sc-smooth homotopy  $s_t\in \Gamma^+(p)$  connecting
$s_0=s$ with $s_1=s'$ so that $F(t, x)=f(x)+s_t(x)$ is a 
proper  Fredholm section of the bundle
$Y^1\to [0,1]\times X^1$ transversal to the zero section  ${\mathcal
M}^F=\{(t, x)\in [0,1]\times X^1\vert \, F(t, x)=0\}$.  

This can be seen as follows.  The two $\ssc^+$-sections $s_i$  for $i=0$ and $i=1$ have their  supports comtained in $U$ and  $f+s_i$ are proper Fredholm sections of the bundle $p^1:Y^1\to X^1$. In addition,  $N(s_i) <1$ and the linearizations  $(f+s_i)'(x)$ are surjective  at every $x$ belonging to the solution set ${\mathcal M}^{f+s_i}$.   For every $t\in [0,1]$, the section $(1-t)s_0+ts_1$ is of class  $\ssc^+$ having  its support in $U$ and 
satisfying  $N((1-t)s_0+ts_1)<1$. Consider  the  Fredholm section $\bar{f}(t, x)=f(x)+(1-t)s_0(x)+ts_1(x)$ of the bundle $Y^1\to [0,1]\times X^1$.  Clearly, the section $\bar{f}$ is proper since  $f$ is proper and $[0,1]$ is compact.  It vanishes at the points $(t, x)\in \{i\}\times {\mathcal M}^{f+s_i}$ for $i=0$ and $i=1$. At these points  linearization of $\bar{f}$ is surjective  because the linearizations of $(f+s_i)$ are surjective at points $x$ belonging to ${\mathcal M}^{f+s_i}$ for $i=0$ and $i=1$.  Moreover, at every such point the kernel of the linearization is good position to the corner structure of $\partial ([0,1]\times X)$.
Then one finds a finite number of $\ssc^+$-sections $\bar{s_1}, \ldots \bar{s}_k$  which vanish  near the boundary $\partial ([0,1]\times X)$ and  have their supports in $[0,1]\times U$so that the Fredholm section 
$$F(\lambda , t, x)=\bar{f}(t, x)+\sum_{i=1}^k\lambda_i \bar{s}_i(t, x)$$
 has at every zero $(0, t, x)$ of $F(0, t, x)=0$  with $t\in (0, 1)$ a surjective  linearization. At  every zero of the form $(0, 0, x)$  or $(0, 1, x)$ of $F(0,t, x)=0$, the linearization is also surjective since  the linearization $(f+s_i)'(x)$ is surjective for $i=0$ and $i=1$.
 In addition, at every solution $(0, 0, x)$ or $(0, 1, x)$, the kernel of the linearization of $F$ is in good position the corner structure of $\R^k\oplus [0,1]\oplus X$.  Therefore, we find a suitable $\varepsilon$ so that the solution space $S_\varepsilon=\{(\lambda,t, x)\vert \, F(\lambda,t, x)=0, \abs{\lambda}<\varepsilon\}$ is a smooth manifold with
boundary with corners. A small regular value $\lambda^\ast$ of the projection
$$
(\lambda, t,x)\mapsto  \lambda
$$
defines the  section $s (t, x)=(1-t)s_0(x)+ts_1(x)+\sum_{i=1}^k \lambda_i^\ast \bar{s}_i(t,x)$ having   the desired properties.

The manifold 
${\mathcal M}^F$ is a compact smooth manifold whose boundary is
given by $\partial {\mathcal M}^F= {\mathcal M}^{f+s'}\cup
(-{\mathcal M}^{f+s})$. With the projection $\pi:[0,1]\times
X\rightarrow X$ we obtain for the closed form $\omega$ by Stokes'
theorem
$$
0=\int_{{\mathcal M}^F} d(\pi^\ast\omega ) =\int_{{\mathcal M}^{f+s'}}\omega
-\int_{{\mathcal M}^{f+s}}\omega.
$$
One concludes that
the number $\int_{{\mathcal M}^{f+s}} \omega$  does not  depend on the choice of a
generic  $s$. Therefore,  we can define $\Phi_{(f,\mathfrak{o})}$ by
$$
\Phi_{(f,\mathfrak{o})}([\omega]) =\int_{{\mathcal M}^{f+s}}\omega,
$$
where $s$ is any  generic small perturbation of the 
proper  Fredholm section $f$.

 Consider two
oriented proper Fredholm sections $(f_0,
\mathfrak{o}_0)$ and $(f_1,\mathfrak{o}_1)$ of the above M-polyfold
bundle $p:Y\to X$, and a proper homotopy $f_t$ of Fredholm
sections connecting $f_0$ with $f_1$ such that the map $(x,
t)\mapsto f(t, x)=f_t(x)$ is a proper  Fredholm section of the bundle
$Y\to [0,1]\times X$. The homotopy is called {\bf oriented} if
there exists an orientation $\ov{ \mathfrak{o}}$ inducing the given
orientations $( -\mathfrak{o}_0)\cup \mathfrak{o}_1$ at the ends.
Then an argument as above  shows that
$$
\Phi_{(f_0,\mathfrak{o}_0)}=\Phi_{(f_1,\mathfrak{o}_1)}.
$$
In order to relate this to the usual mapping degree of $f$ we assume that $(f,  \mathfrak{o})$ is an oriented proper Fredholm section whose Fredholm index is equal to $0$. If $s\in {\mathcal O}$ is a generic perturbation, then  $f+s$ is also a proper  Fredholm section whose Fredholm index is equal to $0$  (by Theorem \ref{prop5.21}) and which, in addition, is transversal to the zero section.
Hence, if $x\in {\mathcal M}^{f+s}$  then the linearization $(f+s)'(x)$ is surjective  and injective  and the local analysis of the zero set of a Fredholm map (Theorem \ref{LOCAX1} and Theorem \ref{newtheoremA})  shows that the solution $x$ is isolated. In view of the compactness, the zero set ${\mathcal M}^{f+s}=\{x_1, \ldots ,x_k\}$ consists of finitely many points. Hence we can define the degree of $f$ by
$$
\text{deg}(f,\mathfrak{o})=\Phi_{(f,\mathfrak{o})}([1]).
$$
where $[1]$ is the cohomology of the constant function equal to $1$
and where the integration of $[1]$ over ${\mathcal M}^{f+s}$ is the
signed sum over the finitely many points in ${\mathcal M}^{f+s}$. A
point $x_i\in {\mathcal M}^{f+s}$ counts as $+1$ if its orientation
$ \mathfrak{o}$ agrees with the natural orientation of the
isomorphism $(f+s)'(x_i)$ as defined in Appendix \ref{det}, and $-1$
otherwise. 

The degree is an invariant of oriented
proper  Fredholm sections under oriented proper 
homotopies as the above discussion shows.

\section{Appendix}\label{appendix-x}
We shall first study  subspaces which are in good position  to a partial quadrant according to  Definition \ref{defn5.6.1}.

\subsection{Two Results on Subspaces in  Good Position}
We consider  the sc-Banach space $E=\R^n\oplus W$ containing the partial quadrant
$C=[0,\infty)^n\oplus W$. Our aim is to prove the following  proposition.

\begin{prop}\label{pointt}
If the finite dimensional linear subspace $N$ of the sc-Banach space $E$ is in good
position to the partial quadrant $C$ in $E$, then $C\cap N$ is a partial quadrant in $N$.
\end{prop}
In higher dimensions one easily can construct a subspace $N$ for
which $C\cap N$ has a nonempty interior, but $C\cap N$ is not a
partial quadrant.

Next we prove Proposition \ref{neat} restated as Proposition \ref{ropp}.
\begin{prop}\label{ropp}
If $N\subset E$ is neat with respect to the partial quadrant $C$,
then $N$ is in good position to $C$ and  $C\cap N$ is a partial quadrant in $N$.
\end{prop}
\begin{proof}
By assumption, the subspace $N$ possesses  an $\ssc$-complement $N^{\perp}$ in $E$ which is  contained in $C$. This implies that $N^\perp=\{0\}\oplus Q$ for some sc-subspace  $Q$ in $W$.
 Since $E=N\oplus N^\perp$,   we find a point of the form $(1,1,1,..,1,e)\in N$, which implies that $N\cap C$ has a nonempty interior. 
Next we take  $c=1$  and assume that $(n,m)\in N\oplus N^{\perp}$ satisfies the estimate  $\norm{m}\leq \norm{n}$. Then
$n+m\in C$ if and only if $n=2[\frac{n+m}{2}+\frac{-m}{2}]\in C$,  because  $C$ is convex, $\R^+\cdot C=C$,  and $-m\in C$. Hence $N$ is in good position to the partial quadrant $C$ and,  by
Proposition \ref{pointt}, the subset $C\cap N$ is a partial quadrant in $N$. The proof of
Proposition \ref{ropp} is complete.
 \end{proof}

\subsection{Quadrants and Cones}
 A {\bf closed convex cone} $P$ in a finite-dimensional vector space
 $N$ is a closed convex
subset so that $P\cap (-P)=\{0\}$ and ${\mathbb R}^+\cdot P=P$.
 An
{\bf extreme ray} in a closed convex cone $P$ is a subset $R$ of the
form
$$
R={\mathbb R}^+x
$$
where  $x\in P\setminus \{0\}$ so that if   $y\in P$ and  $x-y\in P$, then $y\in R$. If the cone $P$ has a nonempty interior it generates the vector space  $N$, i.e.,  $N=P-P$.

The following  version of the Krein-Milman theorem
is well-known. See Exercise 30 on page 72 in cite{Schaefer}. 
\begin{lem}\label{kreinmilman}
A closed convex cone $P$ in a finite-dimensional vector space $N$ is
the  closed convex hull of its extreme rays.
\end{lem}
\begin{proof}
Take a  hyperplane  $A$ in  $N$ such that $A\cap P=\{0\}$. Choose  a point  $a\in P\setminus \{0\}$ and define the affine subspace $A'=a+A$. Note that  for every $x\in P\setminus \{0\}$, there is $t>0$ such that $tx\in P\cap A'$.
Indeed,  if  $x\in P\setminus \{0\}$, then  $x=sa+y$ for some $s\in \R$ and $y\in A$.  If $s=0$, then $x=y\in P\cap A=\{0\}$ contradicting $x\neq 0$. . Hence $s\neq 0$. If $s<0$, then $x+(-s)a=y$ and since $x$ and $(-s)a\in P$, we conclude that $y\in P\cap A=\{0\}$.  So, $x=sa$ with $s<0$ contradicting that $x\in P\setminus \{0\}$.  The  set $P\cap A'$ is convex and closed. It is also bounded since if $x_n\in P\cap A'$ and $\norm{x_n}\to \infty$, then $x_n=a+y_n$ with $x_n\in A$ satisfying $\norm{y_n}\to \infty$. We may assume that $y_n/\norm{y_n}\to y\in A$. Then  $x_n/\norm{y_n}\in P$  and $x_n/\norm{y_n}\to y$. Hence $y\in P\cap A=\{0\}$ contradicting $\norm{y}=1$.  By the  Krein-Milman theorem applied to the set $K=P\cap A'$, the set  $K$ is equal to the closure of the convex hull of the extreme points of $K$.  Hence to prove the lemma it suffices to show that  an extreme point $x$ of $K$  generates an extreme ray $R=\R^+\cdot x$ of $P$.  To see this take $y\in P\setminus \{0\}$ such that  $x-y\in P$.  We have  to show that $y=t\cdot x$ for some $t>0$. If  $x=y$, then we are done.  Otherwise, $y=ta+y'$ with $t>0$ and $y\in A$,  and  $x-y=sa+y''$ with $s>0$ and $y''\in A$. Then $x=y+(x-y)=(t+s)a+(y'+y'')\in K$ and therefore  $s+t=1$. 
By assumption the point $x$ is an extreme point of $K$. Consequently, since $x$ can be written as 
$$x=t  \biggl(\frac{1}{t}y \biggr)+s \biggl(\frac{1}{s} (x-y)\biggr)$$
with the points 
$\frac{1}{t}y=a+\frac{1}{t}y'$ and $\frac{1}{s}(x-y)=a+\frac{1}{s}y''$ belonging  to $K$,  it 
follows that $x=\frac{1}{t}y=\frac{1}{s} (x-y)$.  Hence $y=t\cdot x$ as claimed. 
\end{proof}

We observe that a quadrant
in $N$ has precisely $\dim(N)$ many extreme rays.  A
closed convex cone $P$ is called {\bf finitely generated} provided $P$ has
finitely many extreme rays. In that case $P$ is the convex hull of
its finitely many extreme rays. For example,  if $C$ is a partial
quadrant in $E$ and $N\subset E$ is a finite-dimensional subspace of
$E$ so that $C\cap N$ is a closed convex cone, then $C\cap N$ is
finitely generated.

\begin{lem}\label{nomer}
Let $N$ be a finite-dimensional vector space and $P\subset N$ a
closed convex cone with nonempty interior. Then $P$ is a quadrant if
and only if it has $\dim(N)$-many extreme rays.
\end{lem}
\begin{proof}
Assume that $P$ has $\dim(N)$-many extreme rays, $R_1= \R^+\cdot x_j,\ldots , R_d=\R^+\cdot x_d$ where $d=\text{dim}\ N.$ In view of Lemma  \ref{kreinmilman}, the cone $P$ is the closed convex hull of $R_1,\ldots , R_d$. Since $P$ has nonempty interior, the vectors $x_1,\ldots , x_d$ are linearly independent.  Take the linear isomorphism $F:N\to \R^d$ mapping $x_j$ to the standard vector $e_j$. Then $F(P)$ is the standard quadrant  in $\R^d$. The converse is proved similarly.
\end{proof}

If $a\in C=[0,\infty )^n\oplus W\subset \R^n\oplus W$, we have the representation
$a=(a_1,\ldots ,a_n, a_{\infty})$ where $(a_1,\ldots ,a_n)\in \R^n$ and $a_{\infty}\in W$.
By $\sigma_a$ we shall denote the collection of all  indices $i\in \{1,\ldots ,n\}$ for which
$a_i=0$ and denote the complementary set of indices in $\{1,\ldots ,n\}$ by $\sigma_a^c$. Correspondingly, we introduce the following subspaces in $\R^n$,
\begin{align*}
\R^{\sigma_a}&=\{x\in \R^n\ \vert \ \text{$x_j=0$ for all $j\not \in \sigma_a$}\}\\
\R^{\sigma^c_a}&=\{x\in \R^n\ \vert \ \text{$x_j=0$ for all $j\not \in \sigma^c_a$}\}.
\end{align*}

\begin{lem}\label{roxy}
Let $N\subset E_{\infty}$ be a finite-dimensional smooth linear subspace of
$E={\mathbb R}^n\oplus W$ so that $C\cap N$ is a closed convex cone.
If $a\in C\cap N$ is nonzero and generates an extreme ray $R$ in
$C\cap N$, then
$$
\dim(N)-1\leq \sharp\sigma_a.
$$
If,  in addition, $N$ is in good position to $C$, then
$$
\dim(N)-1=\sharp\sigma_a.
$$
\end{lem}
\begin{proof}
Assume $R=\R^+\cdot a$ is an extreme ray in $C\cap N$. Abbreviate $\sigma=\sigma_a$ and let $\sigma^c$ be a complement of $\sigma$ in $\{1, \ldots ,n\}$. Then $R\subset C\cap N\cap ({\mathbb R}^{\sigma^c}\oplus W)$. Let $y\in C\cap N\cap ({\mathbb R}^{\sigma^c}\oplus
W)$ be a nonzero element. Since $a_i>0$ for all $i\in\sigma^c$,  there exists $\lambda>0$ so that $\lambda a-y\in C\cap N\cap
({\mathbb R}^{\sigma^c}\oplus W)\subset C\cap N$. We conclude $y\in R$ because  $R$ is an
extreme ray. Given any element $z\in
N\cap({\mathbb R}^{\sigma^c}\oplus W)$ we find $\lambda>0$ so that
$\lambda a+z\in C\cap N\cap({\mathbb R}^{\sigma^c}\oplus W) $ and infer,
by the previous argument,  that $\lambda a+z\in R$. This implies that $z\in
{\mathbb R}a$. Hence
\begin{equation}\label{pol}
 \dim(N\cap({\mathbb R}^{\sigma^c}\oplus W))=1.
\end{equation}
The projection $p:\R^n\oplus W=\R^{\sigma}\oplus (\R^{\sigma^c}\oplus W) \to \R^{\sigma}$ induces a linear map
\begin{equation}\label{ohx}
p:N\rightarrow {\mathbb R}^{\sigma}
\end{equation}
which by (\ref{pol}) has  an one-dimensional kernel. Therefore,
$$
\sharp\sigma=\dim({\mathbb R}^{\sigma})\geq \dim(N)-1.
$$
Next assume $N$ is in good position to $C$. Hence there exist a constant $c>0$ and an sc-complement $N^{\perp}$ such that $N\oplus N^{\perp}=\R^n\oplus W$ and if $(n, m)\in N\oplus N^{\perp}$ satisfies $\norm{m}\leq c\norm{n}$, then $n+m\in C$ if and only if $n\in C$.  We claim that  $N^{\perp}\subset {\mathbb R}^{\sigma^c}\oplus W$. Indeed, let $m$ be any element of $N^{\perp}$.  Multiplying $m$ by a real number we may assume  $\norm{m}\leq c\norm{a}$ .  Then $a+m\in C$ since $a\in C$. This implies that
$m_i\geq 0$ for  all indices $i\in\sigma_a$. Replacing $m$ by $-m$,  we conclude
$m_i=0$ for all $i\in\sigma_a$.  So $N^{\perp}\subset  \R^{\sigma^c}\oplus W$ as claimed.
Take  $k\in {\mathbb R}^{\sigma_a}$
and write $(k, 0)=n+m\in N\oplus N^{\perp}$.  From  $N^{\perp}\subset {\mathbb
R}^{\sigma^c}\oplus W$, we conclude $k=p(n)$. Hence the map $p$  in
(\ref{ohx}) is surjective  and the desired result follows.
\end{proof}
\subsection{Proof of  Propositions \ref{pointt}}
Assume that $N$ is a smooth  finite-dimensional subspace of
$E={\mathbb R}^n\oplus W$ in good position to the partial quadrant $C=[0,\infty)^n\oplus
W$. Thus, by definition, there is  an sc-complement  $N^{\perp}$ of $N$ in $E$ and a constant $c>0$ so
that if  $(n,m)\in N\oplus N^{\perp}$ satisfies  $\norm{m}\leq c\norm{n}$, then  the statements
$n\in C$ and $n+m\in C$ are equivalent.

We introduce the subset
$$\Sigma=\bigcup_{a\in C\cap N, a\neq 0}\sigma_{a}\subset \{1,\ldots ,n\}.$$
and denote by $\Sigma^c$ the complement $\{1,\ldots ,n\}\setminus \Sigma$. The associated subspaces of $\R^n$ are defined by
$\R^{\Sigma}=\{x\in \R^n\ \vert \ \text{$ x_j=0$ for $j\not \in \Sigma$}\}$ and
$\R^{{\Sigma}^c}=\{x\in \R^n\ \vert \ \text{$ x_j=0$ for $j\not \in \Sigma^c$}\}$.

\begin{lem}\label{cone1}
$N^{\perp}\subset {\mathbb
R}^{\Sigma^c}\oplus W$.
\end{lem}
\begin{proof}
Take   $m\in N^{\perp}$. We  have to show that $m_i=0$ for all $i\in\Sigma$.
So fix an index  $i\in\Sigma$ and let $a$ be a nonzero element of  $C\cap N$ such that
$i\in\sigma_a$. Multiplying $a$ by a suitable positive number we may
assume $\norm{m}\leq c\norm{a}$.  Since $a\in C$,   we infer that $a+m\in C$. This implies
that $a_i+m_i\geq 0$. By definition of $\sigma_a$, we have $a_i=0$  implying $m_i\geq 0$.
Replacing $m$ by $-m$ we find $m_i=0$. Hence $N^{\perp}\subset {\mathbb
R}^{\Sigma^c}\oplus W$ as claimed.
\end{proof}

Identifying $W$ with $\{0\}\oplus W$ we take   an algebraic complement $\wt{N}$ of $N\cap W$ in $N$ so that
\begin{equation}\label{hoferx}
N=\wt{N}\oplus (N\cap W)\quad \text{and}\quad\quad E= \wt{N}\oplus(N\cap W)\oplus N^{\perp}.
\end{equation}

\begin{lem}\label{ntilde}
If the subspace $N$ of $E$ is in good position to the quadrant  $C$, then
$\wt{N}$ is also  in good position to $C$ and the  subspace $\wt{N}^{\perp}=(N\cap W)\oplus N^{\perp}$ is a good complement of $\wt{N}$ in $E$.
\end{lem}
\begin{proof}
Since $N$ is in good position to the quadrant $C$ in $E$, there exist a constant $c>0$ and an sc-complement $N^{\perp}$ of $N$ in $E$ such that if $(n, m)\in N\oplus N^{\perp}$ satisfies $\norm{m}\leq c\norm{n}$, then
 the statements $n\in C$ and $n+m\in C$ are equivalent.  Since $E$ is a Banach space and $N$ is a finite dimensional subspace of $E$, there is a constant $c_1>0$ such that
 $\norm{n+m}\geq c_1[ \norm{n}+\norm{m}]$ for all $(n, m)\in N\oplus N^{\perp}$.  To prove  that $\wt{N}$ is in good position to $C$,  we shall show that $\wt{N}^{\perp}:=(N\cap W)\oplus N^{\perp}$ is a good complement of $\wt{N}$ in $E$ in the sense of
 Definition \ref{defn5.6.1}. Let $(\wt{n}, \wt{m})\in \wt{N}\oplus \wt{N}^{\perp}=E$ and assume that  $\norm{\wt{m}}\leq c_1c\norm{\wt{n}}$. Write $\wt{m}=n_1+n_2\in (N\cap W)\oplus N^{\perp}$. Since $c_1[\norm{n_1}+\norm{n_2}]\leq \norm{n_1+n_2}=\norm{\wt{m}}\leq c_1c\norm{\wt{n}}$, we get $\norm{n_2}\leq c\norm{\wt{n}}$.  Note that $\wt{n}+\wt{m}=\wt{n}+n_1+n_2\in C$ if and only if $\wt{n}+n_2\in C$ since $n_1\in  \{0\}\oplus W$.  Since  $\norm{n_2}\leq c\norm{\wt{n}}$, this is equivalent to $\wt{n}\in C$. It remains to show that $\wt{N}\cap C$ has a nonempty interior. By assumption $N\cap C$ has nonempty interior. Hence there is a point $n\in N\cap C$ and $r>0$ such that the ball $B^{N}_r(n)$ in $N$ is contained in $N\cap C$. Write $n=\wt{n}+w$ where $\wt{n}\in \wt{N}$ and $w\in N\cap W$. Since $n\in C$ and  $w\in W$, we conclude that $\wt{n}\in C$. Hence $\wt{n}\in \wt{N}\cap C$. 
Take $\nu \in B^{\wt{N}}_r(\wt{n})$,  the open ball in $\wt{N}$ centered at 
  $\wt{n}$ and of radius $r>0$.  We want to prove that $\nu \in C$. Since $C=[0,\infty )^n\oplus W\subset \R^n\oplus W$, we have to prove for $\nu= (\nu', \nu'')\in \R^n\oplus W$ that $\nu'\in [0,\infty )^n$.  We estimate  $\norm{(\nu +w)-n}=\norm{(\nu +w)-(\wt{n}+w)}=\norm{\nu -\wt{n}}<r$ so that $\nu +w\in B^{N}_r(n)$ and hence  $\nu +w\in N\cap C$ . Having identified $W$ with $\{0\}\oplus W$, we have $w=(0, w'')\in \R^n\oplus W$.  Consequently, $\nu+w=(\nu', \nu''+w'')\in N\cap C$ implies $\nu' \in [0,\infty )^n$.  Since also $\nu \in \wt{N}$, one concludes  that $\nu \in \wt{N}\cap C$ and that $\wt{n}$ belongs to the interior of $\wt{N}\cap C$ in $\wt{N}$.  The proof of Lemma \ref{ntilde} is complete.
 \end{proof}
Since by  Lemma \ref{cone1},  $N^{\perp}\subset {\mathbb R}^{\Sigma^c}\oplus W$,  the  good complement   $\wt{N}^{\perp}$ satisfies
$$
\wt{N}^{\perp}\subset {\mathbb R}^{\Sigma^c}\oplus W.
$$
In particular,  $\dim \wt{N}=\text{codim}\  \wt{N}^{\perp}\geq \sharp \Sigma$.

Moreover, since   $C$ is a closed convex cone in $E$ and $\wt{N}$ a subspace of $E$,  we have proved  the following lemma.
\begin{lem}
The intersection $C\cap\wt{N}$ is a closed convex cone in
$\wt{N}$.
\end{lem}
Using the above lemma and  Lemma \ref{ntilde}, we conclude from Lemma  \ref{roxy}
\begin{equation}\label{Nsigma}
\dim \wt{N}-1 =\sharp{\sigma_a}
\end{equation}
for every generator $a$ of an extreme ray in $C\cap \wt{N}$.

The position of $\wt{N}$ with
respect to ${\mathbb R}^{\Sigma^c}\oplus W$ is as follows.
\begin{lem}\label{ll1}
Either $\wt{N}\cap ({\mathbb R}^{\Sigma^c}\oplus W)=\{0\}$ or
$\wt{N}\subset {\mathbb R}^{\Sigma^c}\oplus W$.  In the second case
$\dim \wt{N}=1$ and $\Sigma=\emptyset$.
\end{lem}
\begin{proof}
Assume that  $\wt{N}\cap ({\mathbb R}^{\Sigma^c}\oplus
W)\neq\{0\}$. Take  a nonzero point $x\in \wt{N}\cap (\R^{\Sigma^c}\oplus W)$. We know that $\wt{N}\cap C$ has a
nonempty interior in $\wt{N}$ and is therefore generated as the convex hull of
its extreme rays by Lemma \ref{kreinmilman}.  Let $a\in C\cap \wt{N}$ be a generator of an
extreme ray $R$. Then $a_i>0$ for all $i\in\Sigma^c$ and  hence $\lambda a +x\in C\cap \wt{N}$ for large  $\lambda >0$. Taking another large  number  $\mu>0$, we get   $\mu a- (\lambda a+x)\in C\cap \wt{N}$. Since $R=\R^+\cdot a$ is an extreme ray, we conclude $\lambda a+x\in \R^+\cdot a$ so that $x\in {\mathbb R}\cdot a$. Consequently, 
there is only one extreme ray in $\wt{N}\cap C$, namely 
 $R=\R^+\cdot a$ with  $a\in \R^{\Sigma^c}\oplus W$.  Since $\wt{N}\cap C$ has a nonempty interior in $\wt{N}$, we conclude that $\dim \wt{N}=1$  . Hence $\wt{N}=\R\cdot a$ and 
 $\wt{N}\subset  \R^{\Sigma^c}\oplus W$. From equation \eqref{Nsigma} we also conclude that $a_i>0$ for all $1\leq i\leq n$.  This in turn implies that $\Gamma=\emptyset$ since  $a\in \R^{\Sigma^c}\oplus W$. The proof of Lemma \ref{ll1} is complete.
\end{proof}

In order to complete the proof  of Proposition \ref{pointt},  we have to
consider, according to Lemma \ref{ll1},  two cases.
Starting with the first case  we assume that $\wt{N}\cap ({\mathbb R}^{\Sigma^c}\oplus
 W)=\{0\}$. The projection $p:\wt{N}\oplus \wt{N}^{\perp}=\R^{\Sigma}\oplus (\R^{\Sigma^c}\oplus W)\to \R^{\Sigma}$  induces the linear map
 \begin{equation}\label{opl}
 p:\wt{N}\rightarrow {\mathbb R}^{\Sigma}.
 \end{equation}
Take $k\in{\mathbb R}^{\Sigma}$ and write
 $(k,0)=n+m\in\wt{N}\oplus\wt{N}^{\perp}$. Since $\wt{N}^{\perp}\subset
\R^{\Sigma^c}\oplus W$,  we conclude that
 $$
 p(n+m)=p(n)=k
 $$
so that $p$ is surjective. If $n\in \wt{N}$ and $p(n)=0$, then $n\in  \wt{N}\cap ({\mathbb
 R}^{\Sigma^c}\oplus W)=\{0\}$ by assumption. Hence the map in
 (\ref{opl}) is a bijection.  By  Lemma \ref{nomer},
$C\cap\wt{N}$ is a quadrant in $\wt{N}$. We shall show that $p$ maps the quadrant $C\cap \wt{N}$ onto the standard quadrant $Q^{\Sigma}=[0,\infty )^{\Sigma}$ in $\R^{\Sigma}$.
Let $a$ be a nonzero element in $C\cap \wt{N}$ generating  an extreme ray $R=\R^+\cdot a$. Then, by Lemma \ref{roxy}, 
 $$
 \dim \wt{N}-1=\sharp\sigma_a
 $$
 and since $\sharp \Sigma=\text{dim}\ \wt{N}$ there is
exactly one index $i\in\Sigma$ for which
$a_i>0$. Further,  $a_i>0$ for all $i\in\Sigma^c$ by definition of $\Sigma$.  This implies that there can be
at most $\dim(\wt{N})$-many extreme rays. Indeed, if $a$ and $a'$
generate extreme rays and
$a_i,a_i'>0$  for some $i\in \Sigma$, then $a_k=a_k'=0$ for all
$k\in\Sigma\setminus\{i\}$.  Hence,  from  $a_j>0$ for all $j\in \Sigma^c$,   we conclude $\lambda
a-a'\in C$ for large $\lambda>0$. Therefore,
$a'\in {\mathbb R}^+a$ implying that $a$ and $a'$ generate the same
extreme ray. As a consequence,  $\wt{N}\cap C$ has  precisely $\dim \wt{N}$-many extreme rays because $\wt{N}\cap C$ has a nonempty interior in view of Lemma \ref{ntilde}.  Hence  the map $p$ in \eqref{opl}  induces an isomorphism
$$
(\wt{N},C\cap\wt{N})\rightarrow ({\mathbb R}^\Sigma,Q^\Sigma).
$$
This implies that $(N,C\cap N)$ is isomorphic to
$\bigl( {\R}^{\dim(N)}, [0,\infty)^{\sharp\Sigma}\oplus {\R}^{\dim(N)-\sharp\Sigma}\ \bigr).$

In the second case we assume  that $\wt{N}\subset {\mathbb R}^{\Sigma^c}\oplus W$.  From Lemma  \ref{ll1},  $\Sigma=\emptyset$ and $\wt{N}=\R\cdot a$ for an element $a\in
C\cap\wt{N}$  satisfying  $a_i>0$ for all $1\leq i\in\leq n$. Hence
$(\wt{N},\wt{N}\cap C)$ is isomorphic to
$( {\R}, {\R}^+)$ and therefore $(N,C\cap N)$ is isomorphic to
$( {\R}, \R^+)$ since in this case $N=\wt{N}$. The proof of Proposition \ref{pointt} is complete. \hfill  \qedsymbol

\subsection{Determinants and Orientation.}\label{det}
We shall outline the definition of  determinant bundles
associated to families of linearized polyfold Fredholm operators. While such
constructions are  well-known in the classical case there
are some subtleties in our case. The essential problem is that the
linearisations do not depend as operators  continuously on the
points where they were linearized. Nevertheless what makes the
construction possible is the fact that a Fredholm section has the
contraction germ property. We shall discuss the construction in more
detail in \cite{HWZ8} and just give some outline in the following.
 We recall some facts about determinants of
linear Fredholm operators. For more details we refer to  \cite{DK} and
\cite{FH}. If $A$ is a finite-dimensional real vector space we denote
by $\wedge^{\max}A$ its maximal (nontrivial)  wedge. In case
$A=\{0\}$ we put $\wedge^{\max}A={\mathbb R}$. We begin with the
linear algebra fact that given an exact sequence
$$
0\to  A_0\to  A_1\to A_2\to \cdots \to A_n\to  0
$$
between vector
spaces,
 there is a natural isomorphism (constructed from the maps in the
exact sequence by formulae  depending smoothly on the ingredients)
$$
\otimes_{i\ \text{even}} (\wedge^{\max}A_i) \rightarrow \otimes_{i\ \text{odd}}
(\wedge^{\max}A_i).
$$
There exist  also  natural isomorphisms $A\otimes B\rightarrow B\otimes
A$, $(\wedge^{\max}E)\otimes (\wedge^{\max}E)^\ast\rightarrow {\mathbb
R}$ and $\wedge^{\max}(E^\ast)\rightarrow (\wedge^{\max}E)^\ast$ where $\ast$ refers to the dual space. Of
course,  there are natural isomorphisms $A\otimes {\mathbb R}\rightarrow A$.

Assume that $T:E\rightarrow F$ is a linear Fredholm operator between
Banach spaces. Then we define its {\bf determinant} $\det(T)$ as the
 one-dimensional real vector space
$$
\det(T)=(\wedge^{\max}\ker(T))\otimes
(\wedge^{\max}\text{coker}(T))^\ast.
$$
An orientation of the Fredholm operator  $T$
is by definition an orientation of  the vector space $\det(T)$. An orientation of
$\det(T)$ can be given by a pair of orientations for
$\wedge^{\max}\ker(T)$ and $\wedge^{\max}\text{coker}(T)$. Of course,
reversing both orientations  defines the same orientation for $T$.
If  $T$ is an isomorphism, then  $\det(T)={\mathbb R}\otimes
{\mathbb R}^\ast$ which is naturally isomorphic to ${\mathbb R}$ by the map $e\otimes e^*\mapsto e^*(e)$. In
this case $1$ orients $T$. We call it  the {\bf natural orientation
of an isomorphism}. If $T:E\rightarrow F$ is a surjective Fredholm
operator, then  $\det(T)=(\wedge^{\max}\ker(T))\otimes {\mathbb
R}^\ast$. In this case  an orientation $\mathfrak{o}$ of $\det(T)$ determines
an orientation of $\ker (T)$ by requiring that,  paired with the
canonical orientation of ${\mathbb R}^\ast$ by $1^\ast$,  it gives the
orientation of $T$. Hence an orientation for a surjective Fredholm
operator can be viewed as being equivalent to an orientation of
$\ker(T)$. Since $T:E\rightarrow F$ is Fredholm,  we can
take a projection $P:F\rightarrow F$ so that the range $R(P)$ has finite
codimension and $PT:E\rightarrow P(F)$ is surjective. From
 the exact sequence of  maps
$$
0\rightarrow \ker(T)\rightarrow
\ker(PT)\xrightarrow{T}(I-P)F\rightarrow  F/R(T)\rightarrow 0
$$
 we  derive  the natural isomorphism
$$
(\wedge^{\max}\ker(T))\otimes(\wedge^{\max}(I-P)F)\rightarrow
(\wedge^{\max}\ker(PT))\otimes (\wedge^{\max}\hbox{coker}(T)).
$$
Tensoring with $(\wedge^{\max}(I-P)F)^\ast$ from the right we obtain
the natural isomorphism
\begin{eqnarray*}
&\wedge^{\max}\ker(T)\rightarrow (\wedge^{\max}\ker(PT))\otimes
(\wedge^{\max}\hbox{coker}(T))\otimes(\wedge^{\max}(I-P)F)^\ast&\\
&\rightarrow
(\wedge^{\max}\ker(PT))\otimes(\wedge^{\max}(I-P)F)^\ast\otimes(\wedge^{\max}\hbox{coker}(T)).&
\end{eqnarray*}
By  tensoring from the right with
$(\wedge^{\max}\hbox{coker}(T))^\ast$ we finally end up with the
natural isomorphism
$$
\det(T)\rightarrow \det(PT).
$$
Next we consider  a continuous (in the operator topology) arc of Fredholm operators $t\mapsto T_t $ for 
$t\in [0,1]$. We claim that  there is a projection $P$ so that the  operator $PT_t:E\rightarrow PF$ is surjective for every $0\leq t\leq 1$.

To prove the claim, take  $0\leq t^*\leq 1$. The Banach spaces  $E$ and $F$ split as  follows $E=E' \oplus \ker (T_{t^*})$ and $F=C_{t^*}\oplus R(T_{t^*})$ where 
$C_{t^*}=\text{coker}\ T_{t^*}$ is the cokernel of $T_{t^*}$. Denoting by $Q$ the projection $Q:C_{t^*}\oplus R(T_{t^*})\to R(T_{t^*})$, the map 
$QT_{t^*}\vert E':E'\to R(T_{t^*})$ is an isomorphism and we  
 find an open interval $I(t^*)$  around $t^*$ in $[0,1]$  such that  $QT_{t}\vert E':E'\to R(T_{t^*})$ is an isomorphism for $t\in I(t^*)$. Hence  $QT_{t}:E\to R(T_{t^*})$ is a surjection for $t\in I(t^*)$. 
Now we  cover $[0,1]$ by a finite number of such intervals $I(t_1), \ldots , I(t_n)$ and denote by $Q_i:C_i\oplus  R_i \to R_i$ the corresponding projections onto $R_i=R(Q_i)=R(T_{t_i})$  so that operators  $Q_iT_{t}:E\to  R_i$ are  surjective for every $t\in I(t_i)$. 
With  $C=\text{span} \{C_1,\ldots ,C_n\}$, there exists a topological complement $R$ of $C$ so that $F=C\oplus R$.  Denoting by $P:C\oplus R\to R$ the projection onto $E$ along $C$, we claim that $PT_t:E\to R$ is a surjection for every $t\in [0,1]$. Indeed,  fix  $t\in [0,1]$ and choose $y\in R$. Then $t\in I(t_i)$ for some $i$ and $y=a+b$ where $a\in C_i$ and $b\in R_i$. Hence there exists  a point $x\in E$ solving  $Q_iT_{t}(x)=b$.  This implies that $T_t(x)=a'+b$ with $a'\in C_i$. 
Hence $T_t(x)=(a'-a)+y$ and since $(a'-a)\in C_i\subset C$, it follows that $PT_t (x)=y$.
Consequently, $PT_t:E\to P(F)$ is surjective for every $0\leq t\leq 1$  as claimed.

Proceeding as before  we obtain a natural family of isomorphisms
$$
\varphi_t:\det(T_t)\rightarrow \det(PT_t).
$$
Now we observe that the family $t\mapsto  PT_t$ is a continuous arc of surjective Fredholm operators  so that $\dim \ker PT_t$ is constant. Consequently, the family defines in a natural
way a topological line bundle over $[0,1]$ by
$$
L=\bigcup_{t\in [0,1]} \{t\}\times (\wedge^{\max}\ker(PT_t))\otimes (\wedge^{\max}(I-P)F)^\ast.
$$

We  have a natural bijection of  the line bundle $L$ to the line bundle
$$
K=\bigcup_{t\in [0,1]} \{t\}\times (\wedge^{\max}\ker(T_t))\otimes (\wedge^{\max}\text{coker} (T_t))^\ast.
$$
It  is linear in the
fibers. The punch-line is that taking another projection $Q$ so that
still $QT_t:E\rightarrow R(Q)$ is surjective we obtain another
topological line bundle $L'$ and the transition map $L\rightarrow
L'$ is a topological bundle isomorphism. This implies that $K$
carries in a natural way the structure of a topological line bundle.
This fact is a priori not obvious since the kernel and cokernel
dimensions vary. It is our aim to carry these ideas over to the
polyfold framework. One has to be somewhat careful since our notion
of differentiability is so weak. For example,  we might have a section
$f$ whose linearisations $f'(x(t))$ along a path $x(t)$ are all
Fredholm, but these operators need not be continuously depending on $x(t)$ as linear
operators. As we will see it will  nevertheless be possible to carry out the above
ideas.

Assume we are given a  fillable M-polyfold bundle
$p:E\rightarrow X$ and a  Fredholm section
$f$. From now on we assume everything is built on  separable sc-Hilbert spaces.
We want to introduce the notion of an orientation $\mathfrak{o}$ for
$f$. Assume that $x\in X$ is a smooth point, i.e.,~$x\in X_\infty$.
Then we can look at the set of all linearisations of $f$ at $x$. Any
two such linearisation differ by an $\ssc^+$-operator.  Every  loop
of linearisations is contractible. Hence if we have oriented one
linearisation, then  there is a natural orientation for all the
other linearisations. Therefore,  we can talk at the point $x$ of an
orientation for the linearisation.

Next we study the question of
continuation. Assume that $\gamma:[0,1]\rightarrow X_\infty$ is an
$\ssc^\infty$-path connecting $x=\gamma(0)$ with $y=\gamma(1)$. We
find  an  sc-smooth $\ssc^+$-map $s:[0,1]\times X\rightarrow E$
so that 
$f(\gamma(t))+s(t,\gamma(t))=0$ for all $0\leq t\leq 1$.  Fix $t_0\in (0,1)$ (the cases $t_0=0$ or $t_0=1$ are done in a similar way) and consider the linearisations at $\gamma(t_0)$,
$$
(f+s_{t_0})'(\gamma(t_0)):T_{\gamma(t_0)}X\rightarrow E_{\gamma(t_0)}.
$$
We consider $f$ as a section of the bundle $Y\to [0,1]\times X$. Locally around the point 
$(t_0, \gamma (t_0))\in [0,1]\times X$, the section $f$ has a  filler and  so the section   $f+s$ also has a filler.  It can be chosen in
such a way that in suitable strong bundle coordinates we have a
contraction germ. The linearisation of a filled section is a
stabilization of the linearized unfilled section by an isomorphism of the complements of the fibers.
Hence the determinant $\det((f+s_t)'(\gamma(t)))$  for  $t$ close to $t_0$ can be identified
using the local coordinates with that of the filled object.   Therefore,  we may  assume that  the section $f+s$ is already filled and has the contraction normal form. More precisely we assume that the  section $f+s$ is of the form 
$$f+s: O\subset (\R\times [0,\infty )^k\oplus  R^{n-k})\oplus W\to \R^N\oplus W$$
where $O$ is a relatively open neighborhood of $(t_0, 0, 0)$  which corresponds to $(t_0, \gamma (t_0))$. With the projection $P:\R^N\oplus W\to W$, the map 
$$P(f+s)( t, a, b)=b-B(t, a, b), $$
in which  $a\in [0,\infty )^k\oplus \R^{n-k}$ and $b\in W$,  has the contraction germ property 
near $(t_0, 0, 0)$. 

Denote by $(a(t), b(t))\in [0,\infty )^k\oplus  R^{n-k})\oplus W$ a point which corresponds in our local coordinates to $\gamma (t)$ for $t$ close to $t_0$. Keeping $t$ fixed and  linearizing the above map at the point $(t, a(t), b(t))$ with respect to the variable $(a,b)$ we find that 
$$P(f+s_t)'(a(t), b(t))=\text{1}-D_2(t, a(t), b(t))- D_3B(t, a(t), b(t))$$
where $1$ stands for the identity map $W\to W$.  In view of the proof of Theorem 2.3, the linear map $1-D_3B(t, a(t), b(t)):W\to W$  is an isomorphism so that 
the linearization $P(f+s_t)'(a(t), b(t)):\R^n\oplus W\to W$ is a surjection.
Consequently, we obtain a family of
surjective sc-Fredholm operators
$$
P(f+s_t)'(a(t), b(t)):\R^n\oplus W\rightarrow W,
$$
for $t$ close to $t_0$. Moreover, the kernel is changing smoothly with $t$ near $t_0$.
This implies that the associated locally defined determinant bundle
is locally a topological line bundle which,  by the discussion above
and rolling back the coordinate changes,  implies that  the original family
$t\mapsto  \det(f+s_t)'(\gamma(t))$ defines  a   topological line
bundle
$$
L(f)_\gamma:=\bigcup_{t\in [0,1]} \{t\}\times
\det (f+s_t)'(\gamma(t))
$$
in a natural way.

Two  orientations $\mathfrak{o}_x$ and
$\mathfrak{o}_y$ for the linearisation of $f$ at the points  $x$ and $y$,
respectively, are called {\bf related by continuation} along an
$\ssc$-smooth path connecting $x$ and $y$ if the associated topological line bundle $L(f)_\gamma$
can be oriented in such a way that it induces the given orientations
at the ends.
\begin{defn}
Let $f$ be a Fredholm section of the fillable strong M-polyfold
bundle $p:E\rightarrow X$ whose local models are  built on separable sc-Hilbert spaces. Then $f$ is called
{\bf orientable}  if at every smooth point the linearizations can be
oriented in such a way that they are related by continuation along
arbitrary pathes. If $f$ is orientable a coherent choice of
orientations $x\rightarrow \mathfrak{o}_x$ of the linearisations of
$f$ at $x$ is called an {\bf orientation} for $f$. We write
$(f,\mathfrak{o})$ for an {\bf oriented Fredholm section}.
\end{defn}
\section{Glossary}\label{glossary}
In this section we  recall some of the basic notions from \cite{HWZ2}.
\begin{itemize}
\item[$\bullet$] {\bf sc-structure}.\quad An sc-structure on the Banach space $E$  is 
 a nested sequence 
$$E=E_0\supset E_1\supset E_2\cdots \supset E_{\infty}=\bigcap_{k\geq 0}E_k$$
of Banach spaces $E_m$,  $m \in \N_{0}= \N\cup\{0\}$,  having the following properties.
\begin{itemize}
\item[(1)] If  $m < n$ , the inclusion  $E_n\to E_m$ is a compact operator. 
\item[(2)] The vector space $E_{\infty}$  is dense in $E_m$  for every $m\geq 0$. 
\end{itemize}
A Banach space $E$ equipped with an sc-structure $(E_m)$  is called {\bf sc-smooth}.  Each of the spaces $E_m$ is an sc-Banach space and  denoted by $E^m$. The sc-structure on $E^m$ is given by $(E_{m+k})_{k\geq 0}$. 
Points and sets contained in $E_{\infty}$ are called {\bf smooth points} and {\bf smooth sets}.

\item[$\bullet$] {\bf  direct sum and $\triangleleft$-sum}\quad 
If  $(E_n)_{n\geq 0}$ and $(F_n)_{n\geq 0}$ are sc-smooth structures of $E$ and $F$, then $E\oplus F$ carries the sc-structure defined by $(E\oplus  F)_n=E_n\oplus F_n$. By $E\triangleleft F$ we denote the Banach space $E\oplus F$ equipped with the bi-filtration 
$(E\triangleleft F)_{m, k}=E_m\oplus F_k$ for pairs $(m, k)$ satisfying $m\geq 0$ and $0\leq k\leq m+1$.

\item[$\bullet$]  {\bf $\ssc^0$-map}.\quad A map $\varphi:U\to V$ between open subsets of sc-Banach spaces is said to be  class $\ssc^0$ or  $\ssc^0$,   if 
 $\varphi (U_m)\subset V_m$  and the induced maps $\varphi:U_m\to V_m$ are continuous for all $m\geq 0$. 

\item[$\bullet$]{\bf sc-operator}.\quad    A bounded linear operator $T:E\to F$ 
between sc-Banach spaces is called an {\bf sc-operator} if $T$ is of class $\ssc^0$.  If, in addition,  $T$ is bijective and $T^{-1}:F\to E$ is $\ssc^0$, then $T$ is called an {\bf sc-isomorphism}.

\item[$\bullet$]{\bf partial quadrant}. \quad Given an sc-Banach space $W$, a subset $C\subset W$ is a partial quadrant of $W$ if there is an sc-Banach space $Q$ and an  sc-isomorphism 
$T:W\to \R^n\oplus Q$ such that $T(C)=[0,\infty )^n\oplus Q$.
 
 \item[$\bullet$] {\bf induced sc-structure}\quad 
 If $U$ is a relatively  open subset of a partial quadrant $C$ in the sc-Banach space $E$, then the nested sequence of sets $U_m=U\cap E_m$ is called  the {\bf induced sc-structure} on $U$. The set $U_m$ inherits sc-smooth structure defined by $(U_{m+k})_{k\geq 0}$. The set $U_m$ equipped with this induced sc-structure is denoted by 
 $U^m$.

\item[$\bullet$] {\bf tangent bundle TU}.\quad  Given a relatively  open subset $U$ of the partial quadrant $C$ in the  sc-Banach space $E$,  the tangent bundle $TU$ of $U$ is defined as $TU=U^1\oplus E$. That is, the sc-smooth structure of $TU$, is given by the  nested sequence $(TU)_m:=U_{m+1}\oplus E_m$ for all $m\geq 0$. The canonical projection $p:TU\to U^1$ is of class  $\ssc^0$. The higher order tangent bundles  $T^kU$ are defined iteratively as   $T^1U=TU$ and  $T^kU=T(T^{k-1}U)$ for $k\geq 2$. 

\item[$\bullet$] {\bf sc-subspace}. \quad  If $E$ is an sc-Banach space, then a closed subspace $F$ of $E$ is called sc-subspace of $E$ if the nested sequence $F_m=F\cap E_m$ is an sc-structure for $F$. An sc-subspace $F\subset E$  {\bf  splits} $E$  if there exists another sc-subspace $G$ of $E$ so that $E_m=F_m\oplus G_m$ for all $m\geq 0$.
\item[$\bullet$]{\bf Fredholm operator}. \quad An sc-operator $T:E\to Y$ is called Fredholm provided that there are sc-splittings $E=K\oplus X$ and $F=Y\oplus C$ having the following properties.
\begin{itemize}
\item[(1)] 
$K=\text{kernel} (T)$ is finite dimensional.
\item[(2)] $C$  is finite dimensional.
\item[(3)] $Y=T(X)$ and $T:X\to Y$ is an sc-isomorphism.
\end{itemize}
The finite dimensional vector spaces $K\subset E$ and $C\subset F$ are smooth.
\item[$\bullet$]{\bf $\ssc^+$-operator}.\quad An sc-operator $T:E\to F$ is called an $\ssc^+$-operator  if $T (E_m)\subset E_{m+1}$ for every $m\geq 0$ and $T:E\to E^1$ is of class $\ssc^0$.
\item[$\bullet$]{\bf $\ssc^1$-map}\quad Let $E, F$ be sc-Banach spaces and let $U$ be a relatively  open subset 
of a partial quadrant $C$ contained in the sc-Banach space $E$.  An $\ssc^0$-map $f:U\to F$ is said to be {\bf $\ssc^1$} or of {\bf class $\ssc^1$}  if the following holds.
\begin{itemize}
\item[(1)] For every $x\in U_1$, there exists a  bounded linear map $Df (x)\in {\mathcal L}(E_0, F_0)$ satisfying (with $x+h\in U_1$)
$$\dfrac{1}{\norm{h}_1}\norm{f(x+h)-f(x)-Df (x)h}_0\to 0\quad \text{as $\norm{h}_1\to 0$.}$$
\item[(2)] The {\bf tangent map}  $Tf:TU\to TF$, defined by 
$$Tf(x, h)=(f(x), Df (x)h)$$
is an $\ssc^0$-map.
\end{itemize}
\item[$\bullet$]{\bf $\ssc^k$-map}.\quad  Let $U$ be a relatively  open subset 
of a partial quadrant $C$ contained in the sc-Banach space $E$ and let $F$ be another sc-Banach space. A map  $f:U\subset E\to F$ is an $\ssc^k$-map or of class $\ssc^k$ if the $\ssc^0$-map $T^{k-1}f:T^{k-1}U\to T^{k-1}F$ is of
 class $\ssc^1$. In this case  the tangent map $T(T^{k-1}f):T(T^{k-1}U)\to T(T^{k-1}F)$ is denoted by $T^kf$. 
If $f:U\subset E\to F$ is of class $\ssc^k$ for every $k\geq 0$, then it is called {\bf sc-smooth} or of {\bf class 
$\ssc^{\infty}$}.
 \item[$\bullet$]{\bf sc-diffeomorphism}.\quad A homeomorohism $f:U\to V$ between relatively open subsets of partial quadrants in sc-Banach spaces is called sc-diffeomorphism if $f$ and $f^{-1}$ are sc-smooth.
\
\item[$\bullet$]{\bf sc-smooth splicing}. \quad  Let  $V$  be an open subset of a partial quadrant $C\subset W$, let $E$ be  an sc-Banach space and  let $\pi_v:E\to E$ be a bounded linear projection for every $v\in V$ such that  the map
$$\pi :V\oplus E\to E,  \quad (v, e)\mapsto \pi_v (e)$$
is sc-smooth. Then the triple ${\mathcal S}=(\pi, E, V)$ is called an {\bf sc-smooth splicing}.
\item[$\bullet$]{\bf splicing core}.\quad  Let ${\mathcal S}=(\pi, E, V)$ be an sc-smooth splicing. The associated splicing core is  the  subset of $V\oplus E$ defined by 
$$K^{\mathcal S}=\{(v, e)\in V\oplus E\vert \, \pi_v (e)=e\}.$$
\item[$\bullet$] {\bf tangent splicing of ${\mathcal S}$}.\quad Given a splicing ${\mathcal S}=(\pi, E, V)$, the tangent splicing of ${\mathcal S}$ is the triple defined by 
$$T{\mathcal S}=(T\pi, TE, TV).$$
\item[$\bullet$]  The  {\bf splicing core of the tangent splicing $T{\mathcal S}$} is the set 
$$K^{T{\mathcal S}}=\{(v, \delta v, e, \delta e)\in TV\oplus TE\vert \, (T\pi)_{(v, \delta v)}(e, \delta e)=(e,\delta e)\}.$$
\item[$\bullet$] A {\bf local M-polyfold model}\quad consists of a pair $(O, {\mathcal S})$ in which $O$ is an open subset of the splicing core $K^{\mathcal S}$ associated with the sc-smooth splicing ${\mathcal S}=(\pi, E, V)$. The {\bf tangent of the local M-polyfold model $(O, {\mathcal S})$} is the object defined by 
$$T(O, {\mathcal S})=(K^{T{\mathcal S}}\vert O^1, T{\mathcal S})$$
where $K^{T{\mathcal S}}\vert O^1$ denotes the collection of all points in $K^{T{\mathcal S}}$ which project under the canonical projection $K^{T{\mathcal S}}\to (K^{\mathcal S})^1$ onto the points in $O^1$.
\item[$\bullet$]{\bf  smooth maps between splicing cores}.\quad 
Given  open subsets $O, O'$ of splicing cores $K^{\mathcal S}\subset V\oplus E$ and $K^{\mathcal S'}\subset V'\oplus E'$ where $V$and $V'$ are open subsets of partial quadrants in the sc-Banach spaces $W$ and $W'$, define the open set $\wh{O}\subset V\oplus E$ by  $\wh{O}=\{(v, e)\in V\oplus E\vert (v, \pi_v (e))\in O\}$. An $\ssc^0$-map $f:O\to O'$ is  of class $\ssc^1$
provided  the map 
$$\wh{f}: \wh{O}\subset V\oplus E\to W'\oplus E', \quad \wh{f}(v, e)=f(v, \pi_v (e))$$
is of class $\ssc^1$.  The  tangent map $T\wh{f}$ associated with the $\ssc^1$-map $\wh{f}$  satisfies $T\wh{f}(K^{T{\mathcal S}}\vert O^1)\subset K^{T{\mathcal S}'}\vert O'$ and induces a map $TO\to TO'$ which is denoted by $Tf$ and called the  {\bf tangent map} of $f$. The tangents $TO$ and $TO'$ are open subsets of the splicing cores $K^{T{\mathcal S}}$ and 
$K^{T{\mathcal S}'}$, and the notion of $f$ to be of {\bf class $\ssc^k$} is defined iteratively.
\item[$\bullet$] {\bf M-poyfold}.  Let $X$ be a second countable Hausdorff space. An {\bf M-polyfold  chart} for $X$ is a triple $(U, \varphi, {\mathcal S})$ in which $U$ is an open subset of $X$, ${\mathcal S}=(\pi, E, V)$ is an sc-smooth splicing and $\varphi:U\to K^{\mathcal S}$ is a homeomorphism onto an open subset of the splicing core $K^{\mathcal S}=\{(v, e) \in V\oplus E\vert \, \pi_v (e)=0\}$. Two such charts are compatible if the transition maps between open subsets of splicing cores are sc-smooth. A maximal atlas of sc-smoothly compatible M-poyfold charts is called an {\bf M-polyfold structure} on $X$, and $X$ equipped with such a structure is called {\bf M-polyfold of type 0}. By definition,  an M-polyfold looks locally like an open subset of a splicing core.
\item[$\bullet$] {\bf sc-smooth map between M-polyfolds}. A map $f:X\to X'$  is called of class $\ssc^0$, resp. $\ssc^k$ or sc-smooth if for every point $x\in X$  there exist a chart $(U, \varphi, {\mathcal S})$ around $x$ and a chart $(U', \varphi', {\mathcal S}')$ around $f(x)$ so that $f(U)\subset U'$ and 
$$\varphi'\circ f\circ \varphi (U)\to \varphi' (U')$$
is of class $\ssc^0$, resp. $\ssc^k$ or sc-smooth.
\item[$\bullet$] A {\bf general sc-smooth splicing }  is a triple ${\mathcal R}=(\rho, F, (O, {\mathcal S}))$ in which $(O, {\mathcal S})$ is a local M-polyfold model associated with the sc-smooth splicing 
${\mathcal S}=(\pi, E, V)$ and $O$ is an open subset of the splicing core $K^{\mathcal S}=\{(v, e)\in V\oplus E\vert \, \pi_v (e)=e\}$. The space $F$ is an sc-Banach space and the  map
$$\rho: O\oplus F\to F,\quad ((v, e), u)\mapsto \rho (v, e, u)$$
is sc-smooth. Moreover,  for every $(v, e)\in O$, the map $\rho_{(v, e)}=\rho (v, e, \cdot ):F\to F$ is a bounded  linear projection. A second countable Hausdorff space equipped with a maximal atlas where the local models are open subsets of general splicings are called {\bf M-polyfolds of type 1}.
\item[$\bullet$] The {\bf tangent of a general splicing ${\mathcal R}=(\rho, F, (O, {\mathcal S}))$} is the triple 
$$T{\mathcal R}=(T\rho, TF, (TO, T{\mathcal S})).$$
\item[$\bullet$] A {\bf strong bundle splicing} is a general sc-smooth splicing 
$${\mathcal R}=(\rho, F, (O, {\mathcal S}))$$
having the following {\bf additional property}. If $(v, e)\in O_m$ and $u\in F_{m+1}$, then $\rho ((v, e), u)\in F_{m+1}$, and the triple ${\mathcal R}^1=(\rho, F^1, (O, {\mathcal S}))$ is also a general sc-smooth splicing. The {\bf complementary strong bundle splicing} ${\mathcal R}^{\text{c}}$ is defined by  ${\mathcal R}^{\text{c}}=(1-\rho, F, (O, {\mathcal S})).$ 
\item[$\bullet$]{\bf splicing core of the strong bundle splicing}.  Given a strong bundle splicing ${\mathcal R}=(\rho, F, (O, {\mathcal S})$, the set 
$$K^{\mathcal R}=\{(w, u)\in O\oplus F\vert \, \rho (w, u)=u\}$$
is called the splicing core of the strong bundle splicing ${\mathcal R}$. The splicing core 
$K^{\mathcal R}$ has  the  bi-filtration 
$$K^{\mathcal R}_{m,k}=\{(w, u)\in K^{\mathcal R}\vert \, w\in O_m,u\in F_k\}$$
where $m\geq 0$ and $0\leq k\leq m+1$.  The splicing core $K^{\mathcal R}$ can be viewed as a subset of $O\triangleleft F$ and its bi-filtration is the induced one. The bundle $K^{\mathcal R}\to O$ is called a {\bf local strong bundle}.
With the strong bundle splicing ${\mathcal R}$ there are associated two splicing cores $K^{{\mathcal R}^0}$ and $K^{{\mathcal R}^1}$, denoted by 
$K^{\mathcal R}(0)$ and $K^{\mathcal R}(1)$, and equipped with the filtrations
$$K^{\mathcal R}(0)_m=K^{\mathcal R}_{m,m}\quad \text{and}\quad 
K^{\mathcal R}(1)_m=K^{\mathcal R}_{m,m+1}$$
for $m\geq 0$.
A {\bf  special local strong bundle} is associated with the special strong bundle splicing $ {\mathcal R}=(\text{id}, F, (O, {\mathcal S}))$, where $O\subset K^{\mathcal S}$ is an open set in the splicing core associated with the splicing ${\mathcal S}=(\text{id}, E, V)$. The special local strong bundle is then given by 
$$K^{\mathcal R}=O\triangleleft  F\to O$$
with the filtrations $K^{\mathcal R}(0)_m=O_m\oplus F_m$ and $K^{\mathcal R}(1)_m=O_m\oplus F_{m+1}$. We can view $O$ as a local model of an sc-manifold and $O\triangleleft F$ as a model of a local sc-bundle having the base $O$.
\item[$\bullet$] {\bf $\ssc^1_{\triangleleft}$-maps}. \quad Let ${\mathcal R}=(\rho, F, (O, {\mathcal S}))$ and $ {\mathcal R}'=(\rho', F', (O', {\mathcal S}'))$  be general sc-smooth splicings  with associated splicing cores $K^{\mathcal R}\subset O\oplus F$ and $K^{{\mathcal R}'}\subset O'\oplus F'$.  Let the bundle map $f:K^{\mathcal R}\to K^{{\mathcal R}'}$ be of the form 
$$f(w, u)=(\varphi (w), \Phi (w, u))$$
where $\varphi:O\to O'$ and $\Phi :O\oplus F\to F'$. Then 
\begin{itemize}
\item[(1)] $f$ is  of {\bf class $\ssc^0_{\triangleleft}$} if it induces $\ssc^0$-maps $K^{\mathcal R}(i)\to K^{{\mathcal R}'}(i)$ for $i=0$ and $i=1$.
\item[(2)] $f$ is of {\bf class $\ssc^1_{\triangleleft}$} it it is  $\ssc^0_{\triangleleft}$ and 
 induces $\ssc^1$-maps $K^{\mathcal R}(i)\to K^{{\mathcal R}'}(i)$ for $i=0$ and $i=1$.
\end{itemize}
If $f: K^{\mathcal R}\to K^{{\mathcal R}'}$ is  a map of class $\ssc^1_{\triangleleft}$, then tangent map $Tf:TK^{\mathcal R}\to TK^{{\mathcal R}'}$ is of class $\ssc^0_{\triangleleft}$. If  the tangent map $Tf$ is of class $\ssc^1_{\triangleleft}$, then $f$ is said to be of {\bf class 
$\ssc^2_{\triangleleft}$}.  The $\ssc^k_{\triangleleft}$-classes are defined inductively. The map $f: K^{\mathcal R}\to K^{{\mathcal R}'}$  is of {\bf class $\ssc^{\infty}_{\triangleleft}$} or $\ssc_{\triangleleft}$-smooth if it is of class $\ssc^k_{\triangleleft}$ for every $k$.

The map $f:K^{\mathcal R}\to K^{{\mathcal R}'}$  as above is called a {\bf strong bundle map of class $\ssc^0_{\triangleleft}$} if it induces $\ssc^0$-maps $K^{\mathcal R}(i)\to K^{{\mathcal R}'}(i)$ for $i=0$ and $i=1$. It is called a {\bf strong bundle map of class $\ssc^1_{\triangleleft}$} if it is of class $\ssc^0_{\triangleleft}$ and induces $\ssc^1$-maps $K^{\mathcal R}(i)\to K^{{\mathcal R}'}(i)$ between the splicing cores of the splicings 
${\mathcal R}$ and ${\mathcal R}'$,  for $i=0$ and $i=1$. Proceeding inductively one defines strong bundle maps of class $\ssc^{\infty}_{\triangleleft}$.

\item[$\bullet$] {\bf $\ssc$-smooth section of a local strong bundle $p:K^{{\mathcal R}}\to O$}.\quad Given a 
a local strong bundle $p:K^{{\mathcal R}}\to O$, a section $f$ of  $p$ is called sc-smooth, if $f$ is an sc-smooth section of the bundle $K^{{\mathcal R}}(0)\to O$. The section $f$  is called an  {\bf $\ssc^+$-smooth section}, if it defines an sc-smooth section of the bundle $K^{{\mathcal R}}(1)\to O$. 
\item[$\bullet$]{\bf strong M-polyfold bundle}. \quad  Let $Y$ be an $M$-polyfold of type $1$,  let $X$ an M-polyfold of type $0$, and let $p:Y\to X$ be a surjective sc-smooth map.  It is assumed that each fiber $p^{-1}(x)=Y_x$ is a Banach space. A {\bf strong M-polyfold bundle chart} for the $p:Y\to X$ is a triple $(U, \Phi, (K^{\mathcal R}, {\mathcal R}))$ in which $U\subset X$ is an open set and ${\mathcal R}=(\rho, F, (O, {\mathcal S}))$ a strong bundle splicing with the local model $(O, {\mathcal S})$ of the M-polyfold $X$. The map $\Phi$ is an sc-diffeomorphism $p^{-1}(U)\to K^{{\mathcal R}}$ which is linear on the fibers and which covers the sc-diffeomorphism $\varphi:U\to O$ so that the following diagram commutes,

\mbox{}\\
\begin{equation*}
\begin{CD}
p^{-1}(U)@>\Phi>>K^{{\mathcal R}}\\
@VVpV   @VV\text{pr}_1V  \\
U'@>\varphi >>O. \\
\end{CD}
\end{equation*}
Moreover, $\Phi$ resp. $\varphi$ are smoothly compatible with the M-polyfold structures on $Y$ and $X$, respectively.\\
Two M-polyfold bundle charts $(U, \Phi, (K^{\mathcal R}, {\mathcal R}))$ and $(U', \Psi, (K^{{\mathcal R}'}, {\mathcal R}'))$ with $\Phi$ covering the sc-diffeomorphism $\varphi:U\to O$ and $\Psi$ covering the sc-diffeomorphism $\psi:U'\to O'$ are $\ssc_{\triangleleft}$-compatible  if the the transition map
$$\Psi\circ \Phi^{-1}:K^{\mathcal R}\vert \varphi (U\cap U')\to K^{{\mathcal R}'}\vert \psi (U\cap U')$$
between their splicing cores $K^{\mathcal R}$ and $K^{{\mathcal R}'}$ are $\ssc_{\triangleleft}$-smooth.\\
An {\bf M-polyfold bundle atlas} consists of a family of M-polyfold bundle charts $(U, \Phi, (K^{\mathcal R}, {\mathcal R}))$ so that the underlying open sets cover $X$ and so  that  any two charts are $\ssc_{\triangleleft}$-compatible. A maximal atlas of M-polyfold bundle charts is called an M-polyfold bundle structure and the map
$$p:Y\to X$$
is called a {\bf strong M-polyfold bundle}.
\item[$\bullet$] {\bf sc-smooth section}.\quad   Given a strong M-polyfold  bundle  $p:Y\to X$, a section $f:X\to Y$ is called sc-smooth, if its local representations in the strong M-polyfold bundle charts are sc-smooth.  It is  called  an {\bf  $\ssc^+$-smooth section} if 
its local representations in the strong M-polyfold bundle charts are $\ssc^+$-smooth sections.
\item[$\bullet$] {\bf linearization of an sc-smooth section}. \quad Given a strong M-polyfold bundle $p:Y\to X$ and an sc-smooth section $f:X\to Y$. If $q\in X$ is a smooth point at which the section $f$ vanishes, the linearization of $f$ at $q$ is defined by 
$$f'(q):T_qX\to Y_q,\quad h\mapsto P_q\circ Tf (q)h$$
where $P_q$ is the projection $T_qX\oplus Y_q\to Y_q$.
If at the smooth point $q\in X$ the section does not vanish, then the linearization of $f$ at $q$ is defined as follows.  Take any $\ssc^+$-section defined near $q$ satisfying $s(q)=f(q)$.  Then  the section $f-s$ vanishes at the smooth point $q$  and the {\bf linearization of $f$ with respect to $s$} is defined by 
$$f'_{[s]}(q):T_qX\to Y_q,\quad h\mapsto P_q\circ T(f-s) (q)h.$$
If $s$ and $t$ are two $\ssc^+$-sections such that $s(q)=t(q)=f(q)$, then the linearizations $f'_{[s]}(q)$ and $f'_{[t]}(q)$ differ by an $\ssc^+$-operator.  In particular, if one linearization is an sc-Fredholm operator, then also the other linearization  is an sc-Fredholm operator having the same Fredholm index in view of Proposition 2.11 in \cite{HWZ2}.
\item[$\bullet$]{\bf linearized Fredholm section}.\quad  An sc-smooth section $f$ of the strong M-polyfold bundle $p:Y\to X$  is called linearized Fredholm at the smooth point $q\in X$ if the linearization of $f$  at $q$ is  an sc-Fredholm operator. The section is called linearized Fredholm, if this holds true at all smooth points $q$.

\end{itemize}

\end{document}